\documentclass[12pt]{article}
\usepackage[letterpaper,margin=1in]{geometry}
\usepackage{amsthm,amsmath,natbib}
\RequirePackage[colorlinks,citecolor=blue,urlcolor=blue]{hyperref}
\usepackage{multirow}
\usepackage{amssymb}
\usepackage{setspace}
\usepackage{xr}
\usepackage{natbib}
\usepackage{bm}

\newtheorem{lemma}{Lemma}[section]
\newtheorem{thm}{Theorem}[section]

\newcommand{\GooG}[1]{}

\newcommand{\bx}{\boldsymbol{x}}

\newcommand{\bX}{\boldsymbol{X}}

\newcommand{\brho}{{\rho}}
\newcommand{\bmu}{\boldsymbol{\mu}}

\def\P{\mathbb{P}}
\def\E{E}
\usepackage{mathabx}

\usepackage[normalem]{ulem}

\title{Bootstrap inference for the finite population total under complex sampling designs}
\author{Zhonglei Wang\thanks{Wang Yanan Institute for Studies in Economics and School of Economics,
		Xiamen University, Xiamen, Fujian 361005, P.R.C.}  \and Jae Kwang Kim\thanks{Department of Statistics, Iowa State University, Ames, IA 50011, U.S.A.; Email: jkim@iastate.edu} \and Liuhua Peng\thanks{School of Mathematics and Statistics, the University of Melbourne, Victoria 3010, Australia}}
\date{}
\begin{document}
\doublespacing
\maketitle
\begin{abstract}
	Bootstrap is a useful tool for making statistical inference, but it may provide erroneous results  under complex survey sampling. Most studies about bootstrap-based inference are developed under  simple random sampling and stratified random sampling. In this paper, we propose a unified bootstrap method applicable to some complex sampling designs, including Poisson sampling  and  probability-proportional-to-size sampling. Two main features of the proposed bootstrap method  are that studentization is used to make inference, and the finite  population is bootstrapped based on a multinomial distribution by incorporating the sampling information. We show that the proposed bootstrap method is second-order accurate using the Edgeworth expansion. Two simulation studies are conducted to compare the proposed bootstrap method with the  Wald-type method, which is widely used in survey sampling. Results show that the proposed bootstrap method is better in terms of coverage rate especially when sample size is limited.
	
\end{abstract}
%
%
%

\section{Introduction}

Bootstrap, first proposed by  \citet{efron1979},  is a simulation-based approach for  accessing uncertainty of estimates and for constructing confidence intervals. Bootstrap is widely used in that    it is easy to implement and is second-order accurate under mild conditions \citep[\S3.3]{hall2013bootstrap}.
However,   classical bootstrap methods are not applicable under most sampling designs since the independent or identical distributed assumption may fail.

Under complex sampling, bootstrap methods   have been proposed to handle variance estimation. In survey sampling, one of the most  popular bootstrap approaches  is the rescaling bootstrap method proposed by \citet{rao1988resampling} under stratified random sampling, and they demonstrated that their bootstrap-$t$ intervals are second-order accurate if the variance component is \emph{known}. Such a variance, however, is seldom known in practice. \citet{rao1992} generalized the rescaling bootstrap method to cover the non-smooth statistics, but they did not discuss the second-order accuracy. \citet{sitter1992resampling} considered a mirror-match bootstrap method for sampling designs without replacement and discussed the second-order accuracy based on the \emph{known} population variance as \citet{rao1988resampling}.  \citet{sitter1992comparing} extended the without-replacement bootstrap method \citep{gross1980median} to complex sampling designs and compared the proposed method with the rescaling bootstrap method \citep{rao1988resampling} and the mirror-match bootstrap method \citep{sitter1992resampling}.  \citet{shao1996bootstrap} proposed a bootstrap method for the case when survey data are subject to missingness.  
\citet{soton38498} proposed to use a multinomial distribution to reconstruct the finite population to estimate the mean square error.
\citet{beaumont2012} proposed a generalized bootstrap method   for variance estimation under Poisson sampling. 
\citet{antal2011direct} proposed one-one resampling methods to estimate the variance for some complex sampling designs. 
 \citet{mashreghi2016} gave a comprehensive overview of the bootstrap methods in survey sampling for variance estimation.


 In survey sampling, the literature on bootstrap-based approaches for interval estimation is very limited.  \citet{BickelandFreedman1984}  first considered interval estimation under  stratified random sampling. \citet{booth1994bootstrap} generalized the method of \citet{BickelandFreedman1984} to show  that the constructed confidence interval for a smooth function of the  finite population mean is second-order accurate.    {However, all of the theoretical results, including that of \citet{rao1988resampling} are restricted to stratified random sampling. Although \citet{beaumont2012} discussed a generalized bootstrap method for survey sampling with special attention to Poisson sampling, they did not provide rigorous results for the second-order accuracy of their methods. }


In this paper, we focus on  interval estimation under complex sampling. The goal of this study is to develop a unified  bootstrap method to approximate the sampling distribution of the design-based estimator  under some popular sampling designs, including Poisson sampling, simple random sampling (SRS) and probability-proportional-to-size (PPS) sampling.  The proposed bootstrap methods apply multinomial distributions to generate the bootstrap finite populations by incorporating the sampling information, and the same sampling design  is conducted to  obtain a bootstrap  sample from each bootstrap finite population. A similar idea has been successfully applied to SRS by \citet{gross1980median} and \citet{chao1985bootstrap}.  Our bootstrap methods differ from that proposed by \citet{soton38498} in the sense that the  finite population is iteratively bootstrapped, and an asymptotically pivotal statistic is used to make statistical inference for the finite population total. We also study the theoretical properties of the proposed bootstrap methods for different sampling designs using the Edgeworth expansion. We summarize our contributions in this paper below:
\begin{enumerate}
\item We have proposed a unified bootstrap method for interval estimation under some popular complex sampling designs, including Poisson sampling, SRS and PPS sampling. A simulation study also confirms that the proposed method works even under two-stage cluster sampling.
\item For three commonly used sampling designs, we have provided a rigorous proof for the second-order accuracy of the proposed bootstrap methods and shown that the estimation error is $o_p(n^{-1/2})$ \citep{diciccioromano1995}  under mild conditions.  Wald-type method is widely used in survey sampling, so the proposed bootstrap method is an important contribution  since it provides more accurate inference compared with the Wald-type method under mild conditions. Besides, to our knowledge, we are the first to provide the Edgeworth expansion of a \emph{studentized} estimator under Poisson sampling.
\end{enumerate}

The remaining part of the paper is organized as follows. Sampling designs and design-based estimators under consideration are briefly reviewed in Section 2. In the following three sections, we propose  bootstrap methods for Poisson sampling, SRS and PPS sampling, respectively, and theoretical properties are also investigated. Two simulation studies are conducted to compare the proposed bootstrap method with the Wald-type method  in Section 6. Some concluding remarks are made in Section 7.

\section{Sampling designs and estimates}

In survey sampling, the  finite population is often assumed to be fixed, and the randomness is due to the sampling design.
Let $ \mathcal{F}_N =\{ y_1, \ldots, y_N\}$ be  the finite  population of size $N$, and we are interested in making inference for the finite population total $Y=\sum_{i=1}^Ny_i$. For simplicity, we  assume that  the elements in $\mathcal{F}_N$  are scalers. To avoid unnecessary details, we also assume that the    population size $N$ is known, so it is equivalent to make statistical inference for the finite population mean  $\bar{Y}=N^{-1}Y$.

We  consider three commonly used sampling designs, including Poisson sampling, SRS and PPS sampling. For without-replacement sampling designs, such as Poisson sampling and SRS, $I_i$ is the sampling indicator with $I_i=1$ indicating that the $i$-th element is in the sample and 0 otherwise, and $\pi_i = E(I_i)$ is  the first-order inclusion probability of the $i$-th element for $i=1,\ldots,N$, where   the expectation is taken with respect to the sampling design. Let $\Pi_N = \{\pi_1,\ldots,\pi_N\}$ be the set of first-order inclusion probabilities, and it is assumed to be known.  Poisson sampling generates a sample based on $N$ independent Bernoulli experiments, one for each element in the finite  population. That is, $I_i\sim \mathrm{Ber}(\pi_i)$ for $i=1,\ldots,N$, where $\mathrm{Ber}(\pi_i)$ is a Bernoulli distribution with  success probability $\pi_i\in(0,1)$, and a  sample is $\{y_i:I_i=1,i=1,\ldots,N\}$. Let $n=\sum_{i=1}^NI_i$ be a realized sample size and  $n_0= {E}(n) = \sum_{i=1}^N\pi_i$ be the expected sample size under Poisson sampling.    For SRS, a  without-replacement sample of size $n$ is selected with equal probabilities, and we can show  $\pi_i=nN^{-1}$ for $i=1,\ldots,N$ under SRS.  Denote  $\hat{Y}_{Poi}=\sum_{i=1}^N  y_i\pi_i^{-1}I_i $ to be the Horvitz-Thompson estimator \citep{horvitz1952generalization} of $Y$ under Poisson sampling, and the corresponding one is $\hat{Y}_{SRS} = \sum_{i=1}^N  y_i\pi_i^{-1}I_i =Nn^{-1}\sum_{i=1}^NI_iy_i$ under SRS. The sample size $n$ is  random under Poisson sampling, but it is fixed under SRS. Without loss of generality, assume that the first
$n$ elements are sampled under Poisson sampling or SRS, and the design-unbiased variance estimators are  $\hat{V}_{Poi} = \sum_{i=1}^ny_i^2(1-\pi_i)\pi_i^{-2}$ and $\hat{V}_{SRS} = N(N-n)n^{-1}s_{SRS}^2$, respectively, where $s_{SRS}^2=n^{-1}\sum_{i=1}^n(y_i-\bar{y})^2$  is the sample variance of $\{y_1,\ldots,y_n\}$, and $\bar{y} = n^{-1}\sum_{i=1}^ny_i$.

PPS sampling generates a sample of size $n$ by  independently and identically selecting an element from $\mathcal{F}_N$  $n$ times with selection probabilities $\{p_i: i=1,\ldots,N\}$, where   $p_i\in(0,1)$ is  the \emph{known} selection probability of  $y_i$ for $i=1,\ldots,N$ and $\sum_{i=1}^Np_i =1$.  Replicates may occur in the sample under PPS sampling, and the population total $Y$ is estimated by the Hansen--Hurwitz estimator \citep{HansenHurvitzs194}, which is denoted as $\hat{Y}_{PPS} = n^{-1} \sum_{i=1}^n Z_i$, where $Z_i = p_{a,i}^{-1}y_{a,i}$, $p_{a,i}=p_k$ and $y_{a,i} = y_k$ if $a_i=k$, and $a_i$ is the index of the selected element for the $i$-th draw. A  design-unbiased variance estimator is $\hat{V}_{PPS} = n^{-2}\sum_{i=1}^n(Z_i-\hat{Y}_{PPS})^2$.   

Throughout the paper, assume that  the (expected) sample size is less than the population size. Since we study a sequence of finite populations and inclusion probabilities in the following three sections,  assume that $y_i$ and $\pi_i$ are indexed by $N$ implicitly, and samples are generated independently for different finite populations. We use the notation ``$a_n\asymp b_n$'' to indicate that $a_n$ and $b_n$ have the same asymptotic order. That is,  $a_n\asymp b_n$ is equivalent to $a_n=O(b_n)$ and $b_n=O(a_n)$.
\section{Bootstrap method for Poisson sampling} \label{sec: basic setup}
We propose the following bootstrap method to approximate the sampling distribution of $T_{Poi} = \hat{V}_{Poi}^{-1/2}( \hat{Y}_{Poi} - Y)$ under Poisson sampling.

\begin{enumerate}
	\item[Step 1.] Based on the sample $\{y_1,\ldots,y_n\}$, generate $(N_1^*,  \ldots, N_n^*)$ from a multinomial distribution $\mathrm{MN}(N; {\brho})$ with $N$ trials and a probability vector $\brho$, where $\brho=(\rho_1, \cdots, \rho_n)$ and
	$$\rho_i = \frac{\pi_i^{-1}}{ \sum_{j=1}^n \pi_j^{-1} }$$
	for $i=1,\ldots,n$. Denote $\mathcal{F}_N^*=\{y_1^*,\ldots,y_N^*\}$ and  $\Pi_N^* = \{\pi_1^*,\ldots,\pi_N^*\}$, and they consist of $N_i^*$ copies of $y_i$ and $\pi_i$, respectively. Let the bootstrap  finite population total  be  $Y^* = \sum_{i=1}^Ny_i^* = \sum_{i=1}^nN_i^*y_i$.
	\item[Step 2.] For $i=1,\cdots, n$, generate $m_i^*$ independently from a binomial distribution $\mathrm{Bin} (N_i^*, \pi_i)$ with $N_i^*$ trials and a success probability $\pi_i$.  The bootstrap sample consists of $m_i^*$ replicates of $y_i$ under Poisson sampling. Denote $\hat{Y}_{Poi}^* = \sum_{i=1}^nm_i^*y_i\pi_i^{-1}$ and $T_{Poi}^* = (\hat{V}_{Poi}^*)^{-1/2}( \hat{Y}_{Poi}^* - Y^*)$, where $\hat{V}_{Poi}^*=\sum_{i=1}^{n}m_i^*y_i^2(1-\pi_i)\pi_i^{-2}$ is the bootstrap  variance estimator.
	\item[Step 3.] Repeat the two steps above independently  $M$ times.
\end{enumerate}

Step 1 corresponds to generating a bootstrap finite population $\mathcal{F}_N^*$ and  bootstrap first-order inclusion probabilities $\Pi_N^*$ by incorporating the sampling information.   Based on $\mathcal{F}_N^*$ and $\Pi_N^*$,
Step 2 is used to generate a bootstrap sample,  from which a bootstrap replicate of $T_{Poi}$ is obtained.  Instead of sampling from the bootstrap finite population $\mathcal{F}^*_N$ directly, Step 2 provides a more efficient way to generate a sample using $N^*_1,\ldots,N^*_n$ under Poisson sampling. In Step 2, we center $T_{Poi}^*$ by the bootstrap population total $Y^*$ not by $\hat{Y}_{Poi}$. The reason is that  the finite population should be fixed, and the randomness is due to Poisson sampling. Thus, the statistic should be centered using the corresponding population total $Y^*$. If we center $T_{Poi}^*$ by $\hat{Y}_{Poi}$, it causes additional variability due to generating different bootstrap finite populations. The same argument applies for the other two sampling designs. We use the empirical distribution of $T_{Poi}^*$ to approximate that of $T_{Poi}$ and make inference for $Y$.

Before discussing the theoretical properties of the proposed bootstrap method,
we introduce some mild conditions on $ \mathcal{F}_N$ and $\Pi_N$.
\begin{enumerate}
	\renewcommand{\labelenumi}{(C\arabic{enumi})}
	\item There exist constants $\alpha\in(2^{-1},1]$ and $0<C_1\leq C_2$ such that $n_0\asymp N^{\alpha}$, and  $\pi_i$  satisfies
	\begin{eqnarray}
		\notag &C_1 \leq {n_0}^{-1}{N}\pi_i \leq C_2
	\end{eqnarray}
	for $i=1,\ldots,N$. \label{cond: C1poi}
	
	\item The sequence of finite populations and first-order inclusion probabilities satisfy
	\begin{equation*}
		\lim_{N\to\infty}(n_0N^{-2} {V}_{Poi})= \sigma_1^2,
	\end{equation*}
	where ${V}_{Poi} = \sum_{i=1}^Ny_i^2(1-\pi_{i})\pi_{i}^{-1}$, and $\sigma_1^2$ is a positive constant.\label{con: variance}
	\item  The following condition holds for finite populations, that is, 
	\begin{equation*}
	\lim_{N\to \infty}N^{-1}\sum_{i=1}^N y_i^{8}= C_3,
	\end{equation*}
	where  $C_3$ is a positive constant. \label{cond: C4poi}
	\item Denote $X_{i}=V_{Poi}^{-1/2}y_{i}\pi_{i}^{-1}(I_{i}-\pi_{i})$ for $i=1,\ldots,N$, and let $m=\lfloor n_0^{-1/2}N/(\log n_0)\rfloor$ be the integer part of $n_0^{-1/2}N/(\log n_0)$. Then, there exist constants $t_0>0$ and $a>2$ such that, for any subset $\{X_{\ell_1},\ldots,X_{\ell_m}\}$ of $\{X_{1},\ldots,X_{N}\}$,
	$$\left|\prod_{i=1}^{m}\E\{\exp(\iota tX_{\ell_i})\mid \mathcal{F}_N\}\right|=O(m^{-a})$$ uniformly in $|t|>t_0>0$, where $\iota$ is the imaginary unit.
	\label{cond: Pois 5}\label{cond: extra poi}
\end{enumerate}

We briefly comment on these conditions. Condition (C\ref{cond: C1poi}) is commonly used in survey sampling \citep{fuller2009sampling}; the first part of (C\ref{cond: C1poi}) is a mild restriction on the expected sample size, and the second part regulates the first-order inclusion probabilities.  Condition (C\ref{con: variance})  rules out the degenerate case of the Horvitz--Thompson estimator under Poisson sampling. The  moment condition  in  (C\ref{cond: C4poi})  guarantees  the convergence of the variance estimators and other quantities, and it is also required for SRS and PPS sampling that we will discuss in the following two sections. To illustrate the existence of $\mathcal{F}_N$ and $\Pi_N$ satisfying (C\ref{cond: C1poi}) and (C\ref{con: variance}) simultaneously, consider $\pi_{i} = n_0N^{-1}$, so (C\ref{cond: C1poi}) holds, where $n_0 = \lfloor N^{2/3}\rfloor$ and $C_1=C_2=1$, for example. Then, we have
$
	\lim_{N\to\infty}(n_0N^{-2} {V}_{Poi}) =\lim_{N\to\infty}{N}^{-1}\sum_{i=1}^Ny_i^2\left(1-n_0N^{-1}\right).
$ If $N^{-1}\sum_{i=1}^Ny_i^2$ converges as $N\to\infty$, then (C\ref{con: variance})  holds.  Condition (C\ref{cond: extra poi})  is a counterpart of non-lattice assumption and is useful in deriving Edgeworth expansions. Specifically, for any subset $\{X_{\ell_1},\ldots,X_{\ell_m}\}$ of $\{X_{1},\ldots,X_{N}\}$, condition (C\ref{cond: extra poi}) 
ensures that $\{X_{\ell_i}\}_{i=1}^{m}$ have subsequences of length $O(\log m)$ with different spans; see \citet[\S 16.6]{feller2008introduction} for more discussion on a similar assumption.

Denote $(\mathcal{F}_N,\mathcal{B}_N,P_{N,Poi})$ to be a probability space, where  $\mathcal{B}_N$ and $P_{N,Poi}(\cdot)$ are the $\sigma$-algebra and the probability measure on $\mathcal{F}_N$ associated with Poisson sampling, respectively.  That is, $\mathcal{F}_N=\bigtimes_{i=1}^N\Omega_i$, $\mathcal{B}_N = \bigotimes_{i=1}^N\mathcal{A}_i$ and $P_{N,Poi}(A_1\times A_2\times\cdots\times A_N) = \prod_{i=1}^N \mu_{i}(A_i)$, where $\Omega_i = \{0,1\}$, $\mathcal{A}_i$ is the power set of $\Omega_i$, and $\mu_{i}(\{1\}) = 1-\mu_{i}(\{0\})= \pi_i$  for $i=1,\ldots,N$.  Let $\mathcal{F} = \bigtimes_{N=1}^\infty \mathcal{F}_N$ be the product space and $\mathcal{B} = \bigotimes_{N=1}^\infty \mathcal{B}_N$ be the product-$\sigma$-algebra; see \citet[\S14.1]{klenke2007probability} for details about the notations. By Corollary 14.33 of \citet{klenke2007probability}, there exists a uniquely determined probability measure $\mathbb{P}_{Poi}$ on $(\mathcal{F} ,\mathcal{B} )$ such that $\mathbb{P}_{Poi}(F_1\times F_2\times\cdots\times F_n\times \bigtimes_{N=n+1}^\infty \mathcal{F}_N) = \prod_{i=1}^n P_{i,Poi}$, where $F_i\in \mathcal{F}_i$ for $i=1,\ldots,n$ and $n\in\mathbb{N}$.
\begin{lemma} \label{lemma: almost sure variance poi}
	Suppose that (C\ref{cond: C1poi})--(C\ref{cond: C4poi}) hold. Then,
	\begin{eqnarray}
		n_0N^{-2}(\hat{V}_{Poi} - {V}_{Poi})\to 0 \label{eq: lemma1 2}
	\end{eqnarray}
	as $N\to \infty$ almost surely $(\mathbb{P}_{Poi})$.
	
	Let $\mu^{(3)}_{Poi} = \sum_{i=1}^Ny_i^3(1-\pi_{i})\{(1-\pi_{i})^{2}\pi_{i}^{-2}-1\}$ and $\hat{\mu}^{(3)}_{Poi} = \sum_{i=1}^ny_i^3(1-\pi_{i})\pi_{i}^{-1}\{(1-\pi_{i})^{2}\pi_{i}^{-2}-1\}$. Then,
	\begin{equation}
		n_0^2N^{-3} {\mu}_{Poi}^{(3)} = O(1) \text{~~~and~~~} 	\frac{n_0^2}{N^3}(\hat{\mu}_{Poi}^{(3)} - {\mu}_{Poi}^{(3)})=O_p(n_0^{-1/2}). \label{eq: lemma1 1.2}
	\end{equation}
	
	In addition, denote $\tau_{Poi}^{(3)}=\sum_{i=1}^{N}y_{i}^3(1-\pi_{i})^2\pi_{i}^{-2}$ and $\hat{\tau}_{Poi}^{(3)}=\sum_{i=1}^{n}y_{i}^3(1-\pi_{i})^2\pi_{i}^{-3}$. Then,
	\begin{align}
 		\frac{n_0^2}{N^3} {\tau}_{Poi}^{(3)} = O(1)\text{~~~and~~~} \frac{n_0^2}{N^3}(\hat{\tau}_{Poi}^{(3)} - {\tau}_{Poi}^{(3)})=O_p(n_0^{-1/2}). \label{eq: lemma1 1}
	\end{align}
\end{lemma}

Lemma \ref{lemma: almost sure variance poi} shows some basic properties of the finite population quantities and their design-based estimators. Specifically,  the ratio $\hat{V}_{Poi}^{-1}V_{Poi}\to 1$ almost surely $(\mathbb{P}_{Poi})$ by (C\ref{con: variance}). Under Poisson sampling, $\mu^{(3)}_{Poi}$ is the third central moment of $\hat{Y}_{Poi}$, and $\tau_{Poi}^{(3)}$ 
is a quantity involved in the Edgeworth expansion of the distribution of $T_{Poi}$.

\begin{thm} \label{theorem: edgeworth expansion Poisson}
	 Assume that conditions  (C\ref{cond: C1poi})--(C\ref{cond: extra poi}) hold. Let $\hat{F}_{Poi}(z)=\mathbb{P}_{Poi}\left(T_{Poi}\leq z\right)$ be the cumulative distribution function of $T_{Poi}$ under Poisson sampling. Then,
	 \begin{equation}
	 \frac{\hat{\mu}_{Poi}^{(3)}}{\hat{V}_{Poi}^{3/2}} =O_p(n_0^{-1/2}) \text{~~~and~~~} \frac{\hat{\tau}_{N,Poi}^{(3)}}{\hat{V}_{N,Poi}^{3/2}} =O_p(n_0^{-1/2}). 
	 \label{eq: theorem 1 5}
	 \end{equation}
	 Furthermore,
	\begin{eqnarray}
 		& & \hat{F}_{Poi}(z) = \Phi(z)+\left\{\frac{\hat{\mu}_{N,Poi}^{(3)}}{6\hat{V}_{N,Poi}^{3/2}}(1-z^2)+\frac{\hat{\tau}_{N,Poi}^{(3)}}{2\hat{V}_{N,Poi}^{3/2}}z^2\right\}\phi(z) +o_p(n_0^{-1/2}) \label{eq: theorem 1 6}
	\end{eqnarray}
	uniformly in $z\in \mathbb{R}$, where $\Phi(z)$ is the cumulative distribution function of the standard normal distribution with the probability density function $\phi(z)$.
\end{thm}
we make brief comments on the $o_p(\cdot)$ notation in (\ref{eq: theorem 1 6}) of Theorem \ref{theorem: edgeworth expansion Poisson}. The probability $\hat{F}_{Poi}(z)$ on the left side of (\ref{eq: theorem 1 6}) is not random. However, we use estimators in the Edgeworth expansion to make it easier to compare (\ref{eq: theorem 1 6}) with the result in the following theorem, so instead of $o(\cdot)$, it is reasonable to use $o_p(\cdot)$ on the right side of (\ref{eq: theorem 1 6}). Similar argument can be made for Edgeworth expansions under the other two sampling designs.

In order to establish the Edgeworth expansion for the conditional distribution of $T_{Poi}^*$, we need the following assumption, which is similar to condition (C\ref{cond: extra poi}) but with $m$ replaced by $n_0$.
We isolate (C\ref{cond: extra poi}) and (C\ref{cond: extra poi 02}) since (C\ref{cond: extra poi 02}) is not needed for Theorem \ref{theorem: edgeworth expansion Poisson}.
\begin{enumerate}
	\renewcommand{\labelenumi}{(C\arabic{enumi})}
	\setcounter{enumi}{4}
	\item There exist constants $t_0>0$ and $a>2$ such that, for any subset $\{X_{\ell_1},\ldots,X_{\ell_{n_0}}\}$ of $\{X_{1},\ldots,X_{N}\}$ with cardinality $n_0$,
	$$\left|\prod_{i=1}^{n_0}\E\{\exp(\iota tX_{\ell_i})\mid \mathcal{F}_N\}\right|=O(n_0^{-a})$$ uniformly in $|t|>t_0>0$.
	\label{cond: extra poi 02}
\end{enumerate}

The next theorem presents the Edgeworth expansion for the distribution of  $T_{Poi}^*$ based on the proposed bootstrap method. 

\begin{thm}\label{theorem: e 2}
	Suppose that conditions (C\ref{cond: C1poi})--(C\ref{cond: extra poi 02}) hold. Let $\hat{F}^*_{Poi}(z)$ be the cumulative distribution function of $T_{Poi}^*$ conditional on the bootstrap finite population $\mathcal{F}_{Poi}^*$. Then,
	\begin{eqnarray}
	& & \hat{F}^*_{Poi}(z) = \Phi(z)+\left\{\frac{\hat{\mu}_{N,Poi}^{(3)}}{6\hat{V}_{N,Poi}^{3/2}}(1-z^2)+\frac{\hat{\tau}_{N,Poi}^{(3)}}{2\hat{V}_{N,Poi}^{3/2}}z^2\right\}\phi(z) +o_p(n_0^{-1/2}) \label{eq: theorem 1 2}
	\end{eqnarray}
	uniformly in $z\in \mathbb{R}$.
\end{thm}

By comparing  (\ref{eq: theorem 1 6}) in Theorem \ref{theorem: edgeworth expansion Poisson} with (\ref{eq: theorem 1 2}) in Theorem \ref{theorem: e 2}, we show that the proposed bootstrap method is second-order accurate, but the Wald-type method, which is based on the asymptotic normality of $T_{Poi}$, is not if   ${\mu}_{Poi}^{(3)}$ and ${\tau}_{Poi}^{(3)}$ are nonzero by noting the fact that $\hat{\mu}_{Poi}^{(3)}$ and $\hat{\tau}_{Poi}^{(3)}$ are  design-unbiased estimators of ${\mu}_{Poi}^{(3)}$ and ${\tau}_{Poi}^{(3)}$, respectively. Typically, the cumulative distribution function $\hat{F}^*_{Poi}(z)$ is hard to study analytically, so we use an empirical distribution to approximate it.

Now, consider establishing confidence intervals for the population total $Y$. An approximate two-sided confidence interval at significance level $\alpha$ based on the Wald-type method can be constructed as
\begin{eqnarray}\label{eq:confid_01}
\left(\hat{Y}_{Poi}-q_{1-\alpha/2}\hat{V}_{Poi}^{1/2}, \hat{Y}_{Poi}-q_{\alpha/2}\hat{V}_{Poi}^{1/2}\right),
\end{eqnarray}
where $q_{\alpha/2}$ and $q_{1-\alpha/2}$ are  the  $(\alpha/2)$ and $(1-\alpha/2)$ quantiles of the standard normal distribution, respectively. According to Theorem \ref{theorem: edgeworth expansion Poisson}, though the upper and lower confidence limits of \eqref{eq:confid_01} have error rates of order $O_p(n_0^{-1/2})$, this two-sided confidence interval has error rate of order $o_p(n^{-1/2})$ since ${\hat{\mu}_{N,Poi}^{(3)}}/({6\hat{V}_{N,Poi}^{3/2}})(1-z^2)+{\hat{\tau}_{N,Poi}^{(3)}}({2\hat{V}_{N,Poi}^{3/2}})z^2$ is an even function of $z$, and the $n_0^{-1/2}$ order term in the Edgeworth expansion of $T_{Poi}$ cancel in the error rate. However, the $n_0^{-1/2}$ order term leads to an error rates of order $O_p(n_0^{-1/2})$  for one-sided confidence intervals based on the normal approximation. 

The confidence interval of $Y$ based on the proposed bootstrap methods is 
\begin{eqnarray}\label{eq:confid_02}
\left(\hat{Y}_{Poi}-q^*_{1-\alpha/2}\hat{V}_{Poi}^{1/2}, \hat{Y}_{Poi}-q^*_{\alpha/2}\hat{V}_{Poi}^{1/2}\right),
\end{eqnarray}
where $q^*_{\alpha/2}$ and $q^*_{1-\alpha/2}$ are the $(\alpha/2)$ and $(1-\alpha/2)$ quantiles of $\hat{F}^*_{Poi}(z)$. By Theorem \ref{theorem: e 2}, the coverage error of \eqref{eq:confid_02} is of order $o_p(n_0^{-1/2})$. Moreover, the upper and lower limits of \eqref{eq:confid_02} have error rates $o_p(n_0^{-1/2})$, which outperforms the confidence interval \eqref{eq:confid_01} based on Wald-type method. In addition, the one-sided confidence interval by the proposed bootstrap method is more accurate than the one-sided confidence interval obtained by the Wald-type method. Furthermore, as discussed in Section 3.6 of \citet{hall2013bootstrap}, an asymmetric equal-tailed confidence interval may convey important information. The same arguments can be used for the other two sampling designs.

\section{Bootstrap method for SRS}
We propose  the following procedure to make statistical inference for $T_{SRS} = \hat{V}_{SRS}^{-1/2}( \hat{Y}_{SRS} - Y)$ under SRS.
\begin{enumerate}
	\item[Step 1.] Generate $(N_1^*,  \ldots, N_n^*)$ from $\mathrm{MN}(N; {\brho})$, where
	$\rho_i = n^{-1}$
	for $i=1,\ldots,n$. Then, $\mathcal{F}_N^*$ contains $N_i^*$ copies of $y_i$ for $i=1,\ldots,n$, and the bootstrap finite population total is $Y^* = \sum_{i=1}^NN_i^*y_i$.
	\item[Step 2.] Generate a bootstrap  sample of size $n$, denoted as $\{y_1^*,\ldots,y_n^*\}$, from  $\mathcal{F}_N^*$ using SRS. Then, we can obtain $T_{SRS} ^* = (\hat{V}_{SRS}^*)^{-1/2}( \hat{Y}_{SRS}^* - Y^*)$, where $\hat{Y}_{SRS}^*=Nn^{-1}\sum_{i=1}^ny_i^*$, $\hat{V}_{SRS}^*=N(N-n)n^{-1}s_{SRS}^{*2}$, $s_{SRS}^{*2} = n^{-1}\sum_{i=1}^n(y_i^* - \bar{y}^*)^2$, and $\bar{y}^*=n^{-1}\sum_{i=1}^ny_i^* $.
	\item[Step 3.] Repeat the two steps above independently  $M$ times. 		
\end{enumerate}

The three steps for SRS are similar to those under Poisson sampling, but we do not need $\Pi_N^*$ since $\pi_i^* = nN^{-1}$ for $i=1,\ldots, N$. Different from that under Poisson sampling, the bootstrap sample is generated directly from $\mathcal{F}_N^*$. One commonly used algorithm to generate a sample of size $n$ under SRS is to select elements sequentially from the finite population without replacement. If $n=o(N)$, the computational complexity of selecting each element is $O(N)$. Besides the above bootstrap procedure, we propose the following one. It can be shown that these two procedures are equivalent under SRS, but the computational complexity of the latter is $O(n)$ for selecting each element.
	\begin{enumerate}
			\item[Step 1'.] The same as Step 1 above.
		\item[Step 2'.] Initialize $N_i^{*(0)}=N_i^*$  and $m^*_i=0$ for $i=1,\ldots,n$.
		\item[Step 3'.] Generate a bootstrap sample of size $n$ from $\mathcal{F}^*_N$ under SRS.
		\begin{enumerate}
			\item[Step 3.1'.] Initialize $k=1$.
			\item[Step 3.2'.] Select an index, say $l^{(k)}$, from $\{1,\ldots,n\}$ with selection probability $p_i^{(k)} = N_i^{*(k-1)}/\sum_{j=1}^nN_j^{*(k-1)}$ for $i=1,\ldots,n$.
			\item[Step 3.3'.] Update $m_i^* = m_i^* + 1$ if $i=l^{(k)}$. Set $N_i^{*(k)} = N_i^{*(k-1)}$ if $i\in\{1,\ldots,n\}\setminus \{l^{(k)}\}$, and $N_i^{*(k)} = N_i^{*(k-1)}-1$ if $i=l^{(k)}$, where $A\setminus B=\{x\in A:x\notin B\}$ for two sets $A$ and $B$.
			\item[Step 3.4'.] Set $k=k+1$, and go back to Step 3.2' until $k>n$.
			\item[Step 3.5'.] Obtain $T_{SRS} ^* = (\hat{V}_{SRS}^*)^{-1/2}( \hat{Y}_{SRS}^* - Y^*)$, where $\hat{Y}_{SRS}^* = Nn^{-1}\sum_{i=1}^nm_i^*y_i$, $\hat{V}_{SRS}^* =N(N-n)n^{-1}s_{SRS}^{*2}$, $s_{SRS}^{*2}= n^{-1}\sum_{i=1}^nm_i^*(y_i-\bar{y}^*)^2$, and $\bar{y}^* = n^{-1}\sum_{i=1}^nm_i^*y_i$.
		\end{enumerate}
		\item[Step 4'.] Repeat the above three steps independently   $M$ times.
	\end{enumerate}

We list some necessary conditions for studying the theoretical properties of the proposed bootstrap  method under SRS.

\begin{enumerate}
	\renewcommand{\labelenumi}{(C\arabic{enumi})}
	\setcounter{enumi}{5}
	\item There exist $\beta\in(2^{-1},1]$ and $\kappa\in(0,1)$ such that $n\asymp N^\beta$ and $nN^{-1}\leq 1-\kappa$ as $N\to\infty$. \label{cond: srs2}
	\item The finite population satisfies
	$$
	\lim_{N\to\infty}\sigma_{SRS}^2 =\sigma_2^2
	,
	$$
	where $\sigma_{SRS}^2 = N^{-1} \sum_{i=1}^N(y_i-N^{-1}Y)^2$, and $\sigma_2^2$ is a positive constant.\label{cond: variance srs}
	\item The distribution $G_{N,SRS}$ converges weakly to a strongly non-lattice distribution $G_{SRS}$, where $G_{N,SRS}$ assigns probability $1/N$ to $y_1,\ldots,y_N$. \label{cond: srs1}
	
\end{enumerate}

Condition (C\ref{cond: srs2}) is a counterpart of (C\ref{cond: C1poi}), and it is used to rule out the trivial case when the sample size equals to that of the finite population. Condition (C\ref{cond: variance srs}) regulates the variance of $\mathcal{F}_N$ with respect to the distribution $G_{N,SRS}$, and it concentrates our discussion on the non-degenerate case under SRS. The non-latticed assumption in  (C\ref{cond: srs1}) is used to make the discussion easier, and  a distribution $G(x)$ is strongly non-latticed if $\lvert \int \exp(\iota tx)\mathrm{d}G(x)\rvert \neq 1$ for all $t\neq 0$; see \citet{Babu1984} for details.

We can use a similar argument made in Section \ref{sec: basic setup} to show that   there exists a probability measure $\mathbb{P}_{SRS}$ on the product space $\mathcal{F}=\bigtimes_{N=1}^\infty \mathcal{F}_N$ equipped with the product $\sigma$-algebra $\mathcal{B}$.

\begin{lemma}\label{lemma: srs 1}
	Suppose that (C\ref{cond: C4poi}), (C\ref{cond: srs2}) and (C\ref{cond: variance srs}) hold. Then,
	\begin{equation}
	{\mu}^{(3)}_{SRS}  = O(1),\label{eq: lemma 2.2}
	\end{equation}
	where  ${\mu}^{(3)}_{SRS}=N^{-1}\sum_{i=1}^N(y_i-N^{-1}Y)^3$  is the third central moment of $\mathcal{F}_N$ with respect to the distribution $G_{N,SRS}$.  Besides,
	\begin{equation}
	s_{SRS}^2 - \sigma_{SRS}^2 \to 0 \label{eq: lemma srs 1.1}
	\end{equation}
	as $N\to\infty$ almost surely $(\mathbb{P}_{SRS})$. In addition,
	\begin{eqnarray}
	\hat{\mu}^{(3)}_{SRS} - {\mu}^{(3)}_{SRS} = o_p(1),\label{eq: lemma srs 1.2}
	\end{eqnarray}
	where $\hat{\mu}^{(3)}_{SRS} = n^{-1}\sum_{i=1}^ny_i^3+ 2\bar{y}_n^3 - 3\bar{y}_nn^{-1}\sum_{i=1}^ny_i^2$, and $\bar{y}_n = n^{-1}\sum_{i=1}^ny_i$ is the sample mean.
\end{lemma}

Lemma \ref{lemma: srs 1} is the counterpart of Lemma \ref{lemma: almost sure variance poi} under SRS, and it shows the convergence properties of the sample variance and third central moment under mild conditions. We do not use  scaling factors in (\ref{eq: lemma 2.2})--(\ref{eq: lemma srs 1.2}) since the quantities involved are with respect to the distribution $G_{N,SRS}$. 

\begin{thm} \label{th: Edgeworth SRS}
	Suppose that  (C\ref{cond: C4poi}) and (C\ref{cond: srs2})--(C\ref{cond: srs1}) hold. Then,
	\begin{equation}
	\hat{F}_{SRS}(z) = \Phi(z)  + \frac{(1-n/N)^{1/2}\hat{\mu}^{(3)}_{SRS}}{6n^{1/2} s_{SRS}^3}\left\{3z^2 -\frac{1-2n/N}{1-n/N}(z^2-1)\right\}\phi(z) + o_p(n^{-1/2})\label{eq: SRS ori edg}
	\end{equation}
	uniformly in $z\in \mathbb{R}$, where $\hat{F}_{SRS}(z)=\mathbb{P}_{SRS}\left(T_{SRS}\leq z\right)$ is the cumulative distribution function of $T_{SRS}$.
\end{thm}
Theorem \ref{th: Edgeworth SRS} shows the Edgeworth expansion for the distribution of $T_{SRS}$, and this result is obtained by one result in Section 2 of \citet{babu1985edgeworth}. Instead of using $\mu_{SRS}^{(3)}$ and $\sigma_{SRS}$ as done by \citet{babu1985edgeworth}, we use their estimators in (\ref{eq: SRS ori edg}) based on  Lemma \ref{lemma: srs 1}.

\begin{thm}\label{theorem: srs 2}
	Suppose that  (C\ref{cond: C4poi}) and (C\ref{cond: srs2})--(C\ref{cond: srs1}) hold. Then, we have
	\begin{equation}
	\hat{F}^*_{SRS}(z) = \Phi(z)  + \frac{(1-n/N)^{1/2}\hat{\mu}^{(3)}_{SRS}}{6n^{1/2} s_{SRS}^3}\left\{3z^2-\frac{1-2n/N}{1-n/N}(z^2-1)\right\}\phi(z) + o_p(n^{-1/2})\label{eq: srs boot edg}
	\end{equation}
	uniformly in $z\in \mathbb{R}$, where $\hat{F}^*_{SRS}(z)$ is the cumulative distribution function of $T_{SRS}^*$ conditional on the bootstrap finite population $\mathcal{F}_{N}^*$.
\end{thm}

Theorem \ref{theorem: srs 2} shows the Edgeworth expansion for the distribution of $T^*_{SRS}$ obtained by the proposed bootstrap method. By comparing (\ref{eq: SRS ori edg}) in Theorem \ref{th: Edgeworth SRS} with (\ref{eq: srs boot edg}) in Theorem \ref{theorem: srs 2}, we have shown the second-order accuracy of the proposed bootstrap method.

\section{Bootstrap method for PPS sampling}
We consider PPS sampling in this section and propose  the following bootstrap method to approximate the sampling distribution of $T_{PPS} = \hat{V}_{PPS}^{-1/2}(\hat{Y}_{PPS} - Y)$.
\begin{enumerate}
	\item[Step 1.] Obtain $(N_{a,1}^*,  \ldots, N_{a,n}^*)$ from a multinomial distribution $\mathrm{MN}(N; {\brho})$, where
	$\rho=(\rho_1,\ldots,\rho_n)$
	 and $\rho_{i} = p_{a,i}^{-1}(\sum_{j=1}^np_{a,j}^{-1})^{-1}$ for $i=1,\ldots,n$. Then, $\mathcal{F}_N^*=\{y_1^*,\ldots,y_N^*\}$  consists of $N_{a,i}^*$ copies of $y_{a,i}$, and
	the bootstrap  finite population total is  $Y^* = \sum_{i=1}^Ny_i^* = \sum_{i=1}^nN^*_{a,i}y_{a,i}$. The  bootstrap selection probabilities are  $\{(C_N^*)^{-1}p_1^*,\ldots,(C_N^*)^{-1}p_N^*\},$ where $C_N^* = \sum_{i=1}^Np_i^*=\sum_{i=1}^nN_{a,i}^*p_{a,i}$, and $\{p_1^*,\ldots,p_N^*\}$ consists of $N_{a,i}^*$ copies of $p_{a,i}$ for $i=1,\ldots,n$.
	\item[Step 2.] Based on $\mathcal{F}_N^*$, generate a sample of size $n$ by  independently and identically selecting an element from $\mathcal{F}_N^*$  $n$ times with selection probabilities $\{(C_N^*)^{-1}p_i^*: i=1,\ldots,N\}$. Then, we have  $T_{PPS}^*= (\hat{V}_{PPS}^*)^{-1/2}( \hat{Y}_{PPS}^* - Y^*)$, where  $\hat{Y}^*_{PPS} = n^{-1}\sum_{i=1}^nC_N^*(p_{b,i}^*)^{-1}y_{b,i}^*$ , $y_{b,i}^*=y_k^*$ and $p_{b,i}^*=p_k^*$ if the index of the $i$-th draw is $k$, and $\hat{V}_{PPS}^* = n^{-2}\sum_{i=1}^n\{C_N^*(p_{b,i}^*)^{-1}y_{b,i}^*-\hat{Y}_{PPS}^*\}^2$ is the counterpart of $\hat{V}_{PPS}$ based on the bootstrap  sample.
	\item[Step 3.] Repeat the two steps above independently   $M$ times.
\end{enumerate}

To implement the proposed bootstrap method for PPS sampling, the bootstrap selection probability should be standardized before drawing a sample.
Similarly to the previous two sections, we use the empirical distribution of  $T_{PPS}^*$ to make statistical inference for $T_{PPS}$.

The computational complexity of selecting an element in Step 2 is $O(N)$. 
An equivalent way of carrying out the proposed bootstrap method under PPS sampling is described below, and its computational complexity is $O(n)$ for selecting an element. 

	\begin{enumerate}
	\item[Step 1'.] The same as Step 1 above.
	\item[Step 2'.]  Obtain an independent and identical sample of size $n$ from $\{1,\ldots,n\}$, and the selection probability of $i$ is $p_{i}^{\dagger} = (C_N^*)^{-1}N_i^*p_{a,i}$ for $i=1,\ldots,n$. Denote $m_i^*$ to be the number of  $i$'s in the sample.  Then, we have $T_{PPS}^*= (\hat{V}_{PPS}^*)^{-1/2}( \hat{Y}_{PPS}^* - Y^*)$, where $\hat{V}_{PPS}^* = n^{-2}\sum_{i=1}^nm_i^*(C_N^*p_{a,i}^{-1}y_{a,i} - \hat{Y}_{PPS}^*)^2$ and $\hat{Y}^*_{PPS} = n^{-1}\sum_{i=1}^n m_i^*C_N^*p_{a,i}^{-1}y_{a,i}$.
	\item[Step 3'.] Repeat the above three steps independently   $M$ times.
\end{enumerate}

The following regularity conditions are required to validate the proposed bootstrap method under PPS sampling.
\begin{enumerate}
	\renewcommand{\labelenumi}{(C\arabic{enumi})}
	\setcounter{enumi}{8}
	\item There exists $\gamma\in(2^{-1},1]$ such that $n \asymp N^{\gamma}$, and the selection probabilities satisfy
	$$C_4\leq Np_i \leq C_5$$
	for $i=1,\ldots,N$, where $C_4$ and $C_5$ are positive constants.\label{cond: C4}\label{cond: C5}
	\item The sequence of finite populations and selection probabilities   satisfy
	$$
		\lim_{N\to\infty}(N^{-2}\sigma_{PPS}^2) =\sigma_3^2,
	$$
	where $\sigma_{PPS}^2 = \sum_{i=1}^Np_i(p_i^{-1}y_i - Y)^2$, and $\sigma_3^2$ is a positive number. \label{cond: C9}
\item
{The distribution $G_{N,PPS}$ is non-lattice, where $G_{N,PPS}$ assigns probability $p_i$ to $p_i^{-1}y_i$ for $i=1,\ldots,N$.}
\label{cond: PPS1}
\end{enumerate}

Condition (C\ref{cond: C4}) regulates the sample size and selection probabilities, and (C\ref{cond: C9}) rules out the degenerate case under PPS sampling. To show (C\ref{cond: C4}) and (C\ref{cond: C9}) can be satisfied simultaneously, take $p_i = N^{-1}$ for $i=1,\ldots,N$. Then, (C\ref{cond: C4}) holds with $0<C_4<1<C_5$, and $\sigma_{PPS}^2 = \sum_{i=1}^Np_i(p_i^{-1}y_i - Y)^2 = N\sum_{i=1}^Ny_i^2 -N^2\bar{Y}^2 $, where $\bar{Y} = N^{-1}Y$.  Thus, $N^{-2}\sigma_{PPS}^2 = N^{-1}\sum_{i=1}^Ny_i^2 -\bar{Y}^2$ converges if both $N^{-1}\sum_{i=1}^Ny_i^2$ and $\bar{Y}$ converge as $N\to\infty$.  Since $G_{N,PPS}$ corresponds to the PPS sampling procedure, condition (C\ref{cond: PPS1}) focuses our attention to the non-lattice case. 

Based on a similar argument made under Poisson sampling, there exists a probability measure $\mathbb{P}_{PPS}$ on $\mathcal{F}=\bigtimes_{N=1}^\infty \mathcal{F}_N$ equipped with the product $\sigma$-algebra $\mathcal{B}$ under PPS sampling.

\begin{lemma}\label{lemma: pps 1}
	Suppose that (C\ref{cond: C4poi}), (C\ref{cond: C4}) and (C\ref{cond: C9}) hold. Then,
	\begin{equation}
	N^{-2}(s_{PPS}^2 - \sigma_{PPS}^2) \to 0 \label{eq: lemma pps 1.2}
	\end{equation}
	as $N\to\infty$ almost surely $(\mathbb{P}_{PPS})$, where  $s_{PPS}^2=n^{-1}\sum_{i=1}^{n}(Z_i-\bar{Z}_n)^2$ is the sample variance of $\{Z_1,\ldots,Z_n\}$. Let ${\mu}^{(3)}_{PPS} = \sum_{i=1}^Np_i(p_i^{-1}y_i - Y)^3$ and $\hat{\mu}^{(3)}_{PPS} = n^{-1}\sum_{i=1}^nZ_i^3+ 2\bar{Z}_n^3 - 3\bar{Z}_nn^{-1}\sum_{i=1}^nZ_i^2$, then
	\begin{equation}
	N^{-3}{\mu}^{(3)}_{PPS} = O(1) \text{~~~and~~~} N^{-3}(\hat{\mu}^{(3)}_{PPS} - {\mu}^{(3)}_{PPS}) = O_p(n^{-1/2}).\label{eq: lemma PPS 1.2}
	\end{equation}
\end{lemma}

Lemma \ref{lemma: pps 1}  shows  convergence properties of estimators of the variance and third central moment. The next theorem deals with the Edgeworth expansion for the distribution of $T_{PPS}$.
\begin{thm}\label{theorem: pps edgeworth origin}
	Suppose that  (C\ref{cond: C4poi}), (C\ref{cond: C4})--(C\ref{cond: PPS1}) hold. Then,
	\begin{equation}
	\hat{F}_{PPS}(z) = \Phi(z)  + \frac{\hat{\mu}^{(3)}_{PPS}}{6\sqrt{n} s_{PPS}^3}(2z^2+1)\phi(z) + o_p(n^{-1/2})\label{eq: pps ori edg}
	\end{equation}
	uniformly in $z\in \mathbb{R}$, where $\hat{F}_{PPS}=\mathbb{P}_{PPS}\left(T_{PPS}\leq z \right)$ is the cumulative distribution function of $T_{PPS}$ under PPS sampling.
\end{thm}

Based on the result in Theorem \ref{theorem: pps edgeworth origin}, the Wald-type method may provide inefficient inference results for $Y$ compared with the proposed bootstrap method if the sample size is small and  ${\mu}^{(3)}_{PPS}\neq0$.

\begin{thm}\label{th: last one}
	Suppose that  (C\ref{cond: C4poi}), (C\ref{cond: C4})--(C\ref{cond: PPS1}) hold. Then, we have
	\begin{equation}
	\hat{F}_{PPS}^*(z) = \Phi(z)  + \frac{\hat{\mu}^{(3)}_{PPS}}{6\sqrt{n} s_{PPS}^3}(2z^2+1)\phi(z) + o_p(n^{-1/2})\label{eq: pps boot edg}
	\end{equation}
	uniformly in $z\in\mathbb{R}$, where $\hat{F}_{PPS}^*(z)$ is the  cumulative distribution function of $T_{PPS}^*$ conditional on the bootstrap finite population $\mathcal{F}_N^*$.
	
\end{thm}
Theorem \ref{th: last one} shows the Edgeworth expansion for the cumulative distribution function of $T_{PPS}^*$ based on the proposed bootstrap method. By comparing (\ref{eq: pps ori edg}) in Theorem \ref{theorem: pps edgeworth origin} with (\ref{eq: pps boot edg}) in Theorem \ref{th: last one}, we have shown that the proposed bootstrap method is second-order accurate under PPS sampling.
\section{Simulation study}
\subsection{Single-stage sampling designs} \label{sec: single stage}
We conduct a simulation study based on single-stage sampling designs in this section. A  finite population $ \mathcal{F}_N = \{y_1,\ldots,y_N\}$ is generated by
\begin{eqnarray}
y_i \sim \mathrm{Exp(10)}\notag
\end{eqnarray}
for $i=1,\ldots,N$,
where $\mathrm{Exp(\lambda)}$ is an exponential distribution with a scale parameter $\lambda$, and the population size is $N=500$, which is assumed to be known.  The size measure is simulated by $z_i= \log(3 + s_i)$ for $i=1,\ldots,N$, where $s_i\mid y_i \sim \mbox{Exp}(y_i)$. The  expected sample size is $n_0\in\{10,100\}$. We are interested in constructing a 90\% confidence interval for the finite population mean $\bar{Y}$ by survey data under the following sampling designs, and its true value is around 9.7.
\begin{enumerate}
	\item Poisson sampling. The first-order inclusion probability is $\pi_i = n_0z_i\left(\sum_{j=1}^Nz_j\right)^{-1}$ for $i=1,\ldots,N$, and its expected sample size is $n_0$.
	\item SRS with sample size $n_0$.
	\item PPS sampling. The selection probability for this design is $p_i = z_i\left(\sum_{j=1}^Nz_j\right)^{-1}$ for $i=1,\ldots,N$, and the sample size is $n_0$.
\end{enumerate}

Based on a sample, denote  $\tilde{V}$ to be the design-unbiased variance estimator of $\tilde{Y}$, where $\tilde{Y}$ is the design-unbiased estimate of $\bar{Y}$ under a specific sampling design. 
 We consider the following methods to construct the 90\% confidence interval.
\begin{itemize}
	\item[Method I.] Proposed bootstrap method by setting $M=1\,000$.  Denote $q_{B,0.05}$ and $q_{B,0.95}$   to be the 5\%-th and 95\%-th sample quantiles of  $\{(\tilde{V}^{*(m)})^{-1/2}(\tilde{{Y}}^{*(m)} -\bar{Y}^{*(m)}): m=1,\ldots,M\}$ obtained by the proposed bootstrap method, where $\tilde{V}^{*(m)}$, $\tilde{Y}^{*(m)}$ and  $\bar{Y}^{*(m)}$ are the bootstrap counterparts of $\tilde{V}$, $\tilde{Y}$ and $\bar{Y}$ in the $m$-th repetition. Then, a 90\% confidence interval for $\bar{Y}$ can be constructed by
	$$
	(\tilde{Y} - q_{B,0.95}\tilde{V}^{1/2}, \tilde{Y} - q_{B,0.05}\tilde{V}^{1/2}).
	$$
	
	\item[Method II.] Wald-type method. A Wald-type 90\% confidence interval for $\bar{Y}$ is obtained by
	$$
	(\tilde{Y} - q_{0.95}\tilde{V}^{1/2}, \tilde{Y} - q_{0.05}\tilde{V}^{1/2}),
	$$
	where $q_{0.05}$ and $q_{0.95}$ are the 5\%-th and 95\%-th quantiles of the standard normal  distribution.

\end{itemize}

We conduct $1\,000$ Monte Carlo simulations for each sampling design, and the two methods are compared in terms of the coverage rate and the length of the constructed 90\% confidence interval.
Table \ref{tab: confidence interval} summarizes the simulation results. When the sample size is small, the proposed bootstrap method is more preferable in the sense that its coverage rates are closer to 0.9 compared with the Wald-type method under the three sampling designs. The confidence interval constructed by the proposed bootstrap  method is wider compared with that by the Wald-type method.  As the sample size increases to $n_0=100$, the performance of the two methods is approximately the same  in the sense that the coverage rates of both methods are close to 0.9, and confidence interval lengths are approximately the same.  
\begin{table}[!t]
	\centering
	\caption{Coverage rate and length of the constructed 90\% confidence interval for the proposed bootstrap method (Bootstrap) and the Wald-type method (Wald-type) under single-stage sampling designs, including Poisson sampling  (Poisson), SRS and PPS sampling (PPS).  ``C.R.'' stands for the coverage rate, and ``C.L.'' presents the Monte Carlo mean of the lengths of the constructed confidence interval. }
	\label{tab: confidence interval}
	\begin{tabular}{ccrrrrr}
		\hline
		\multirow{2}{*}{Design} & \multirow{2}{*}{Method} & \multicolumn{2}{c}{$n_0=10$}&&\multicolumn{2}{c}{$n_0=100$}\\
		&& C.R. & C.L. && C.R. & C.L. \\
		\hline
		
		\multirow{2}{*}{Poisson} & Bootstrap & 0.90 & 15.5 &  & 0.90 & 3.6 \\
		& Wald-type & 0.84 & 12.1 &  & 0.88 & 3.6 \\
		&  &  &  &  &  &  \\
		\multirow{2}{*}{SRS} & Bootstrap & 0.90 & 13.0 &  & 0.89 & 2.8 \\
		& Wald-type & 0.83 &  9.1 &  & 0.89 & 2.8 \\
		&  &  &  &  &  &  \\
		\multirow{2}{*}{PPS} & Bootstrap & 0.88 & 10.3 &  & 0.90 & 2.6 \\
		& Wald-type & 0.83 &  7.5 &  & 0.89 & 2.6 \\
		\hline

	\end{tabular}
\end{table}

In addition, we also compare the two methods in terms of approximating the probability $\mathbb{P}\{\tilde{V}^{-1/2}(\tilde{Y}-\bar{Y})\leq z\}$, which is obtained by 10\,000 Monte Carlo simulations. We set $z\in\{-0.5,-0.25,-0.1,0,0.1,0.25,0.5\}$ as done by \citet{Lai1993Edgeworth}. Table \ref{tab: distribution} summarizes the simulation results. For both sample sizes, the proposed bootstrap method can approximate the target distribution well, but the performance of the Wald-type method is not as good as the proposed one when sample size is small.

\begin{table}[!t]
	\centering
	\caption{Values of $P_z =\mathbb{P}\{\tilde{V}^{-1/2}(\tilde{Y}-\bar{Y})\leq z\}$,  the normal approximation $\Phi(z)$ and the bootstrap approximation Boot$_z$ for three sampling designs including Poisson sampling  (Poisson), SRS and PPS sampling (PPS). For convenience, we include the values $\Phi(z)$ for both sample sizes. }
	\label{tab: distribution}
	\begin{tabular}{ccccccccc}		\hline
		\multirow{2}{*}{Design} &\multirow{2}{*}{$z$} &  \multicolumn{3}{c}{$n_0=10$}&&\multicolumn{3}{c}{$n_0=100$}\\
		&&$P_z$& $\Phi(z)$ & Boot$_z$&& $P_z$& $\Phi(z)$ & Boot$_z$ \\
		\hline
\multirow{7}{*}{Poisson} & -0.5 & 0.37 & 0.31 & 0.36 &  & 0.32 & 0.31 & 0.32 \\
& -0.25 & 0.45 & 0.40 & 0.44 &  & 0.41 & 0.40 & 0.41 \\
& -0.1 & 0.50 & 0.46 & 0.49 &  & 0.47 & 0.46 & 0.47 \\
& 0 & 0.54 & 0.50 & 0.53 &  & 0.51 & 0.50 & 0.51 \\
& 0.1 & 0.58 & 0.54 & 0.57 &  & 0.55 & 0.54 & 0.55 \\
& 0.25 & 0.64 & 0.60 & 0.63 &  & 0.61 & 0.60 & 0.61 \\
& 0.5 & 0.73 & 0.69 & 0.73 &  & 0.70 & 0.69 & 0.70 \\
&  &  &  &  &  &  &  &  \\
\multirow{7}{*}{SRS} & -0.5 & 0.37 & 0.31 & 0.34 &  & 0.32 & 0.31 & 0.32 \\
& -0.25 & 0.45 & 0.40 & 0.42 &  & 0.41 & 0.40 & 0.41 \\
& -0.1 & 0.50 & 0.46 & 0.48 &  & 0.47 & 0.46 & 0.47 \\
& 0 & 0.54 & 0.50 & 0.52 &  & 0.51 & 0.50 & 0.51 \\
& 0.1 & 0.58 & 0.54 & 0.56 &  & 0.55 & 0.54 & 0.55 \\
& 0.25 & 0.63 & 0.60 & 0.62 &  & 0.61 & 0.60 & 0.61 \\
& 0.5 & 0.73 & 0.69 & 0.71 &  & 0.70 & 0.69 & 0.70 \\
&  &  &  &  &  &  &  &  \\
\multirow{7}{*}{PPS} & -0.5 & 0.37 & 0.31 & 0.34 &  & 0.33 & 0.31 & 0.32 \\
& -0.25 & 0.45 & 0.40 & 0.42 &  & 0.42 & 0.40 & 0.41 \\
& -0.1 & 0.50 & 0.46 & 0.48 &  & 0.47 & 0.46 & 0.47 \\
& 0 & 0.54 & 0.50 & 0.52 &  & 0.51 & 0.50 & 0.51 \\
& 0.1 & 0.57 & 0.54 & 0.56 &  & 0.55 & 0.54 & 0.55 \\
& 0.25 & 0.63 & 0.60 & 0.62 &  & 0.61 & 0.60 & 0.61 \\
& 0.5 & 0.72 & 0.69 & 0.71 &  & 0.70 & 0.69 & 0.70 \\

		\hline

	\end{tabular}
\end{table}

\subsection{Two-stage sampling designs}
In this section, we test the performance of the proposed method under two-stage sampling designs. A  finite population $ \mathcal{F}_N = \{y_{i,j}:i=1,\ldots,H; j=1,\ldots,N_i\}$ is generated by
\begin{eqnarray*}
	y_{i,j} &=& 50 + a_i + e_{i,j},\\
	a_i &\sim&  N(0,50), \\
	e_{i,j} &\sim& \mathrm{Exp}(20),\\
	N_i\mid a_i &\sim& \mathrm{Poisson}(q_i) + c_0
\end{eqnarray*}
for $i=1,\ldots,H$ and $j=1,\ldots,N_i$,
where $\mathrm{Poisson}(\lambda)$ is a Poisson distribution with a rate parameter $\lambda$, $q_i = (a_i-25)^2/20$, $c_0=40$ is the minimum cluster size, and  $H=100$ is the number of clusters in the finite population. The finite population size is $N=7\,129$, and the cluster sizes range from 43 to 129. We assume that the finite population size $N$ and cluster sizes $N_1,\ldots,N_H$ are known.
We are interested in constructing a 90\% confidence interval for the finite population mean  $\bar{Y} = N^{-1}\sum_{i=1}^H\sum_{j=1}^{N_i}y_{i,j}$, where the true value of $\bar{Y}$ is approximately 70.5.

We consider two different sampling designs for the first stage; one is  Poisson sampling, and the other one is  PPS sampling. The first-order inclusion probability (selection probability) of the $i$-th cluster  is proportional to its cluster size $N_i$ under  Poisson (PPS) sampling for $i=1,\ldots,H$. SRS is conducted within each selected cluster independently in the second stage. The expected sample size of the first-stage sampling is $n_1$, and that of the second-stage sampling is $n_2$.  In this simulation, we consider two scenarios for the sample sizes, that is,  $(n_1,n_2)=(5,10)$ and $(n_1,n_2)=(10,30)$.

The derivations of the design-unbiased estimator $\tilde{Y}$ and its variance estimator $\tilde{V}$ under the two-stage sampling designs  in this simulation study are presented in Appendix \ref{appd: B}.  We consider the following methods to construct the 90\% confidence intervals for the parameters of interest.
\begin{itemize}
	\item[Method I.] The proposed method extended to a two-stage sampling design. This method is approximately the same as that mentioned in Section \ref{sec: single stage} with the following two steps to bootstrap the  finite population, and we set $M=1\,000$ for this method.
	\begin{enumerate}
		\item[Step 1.] Use the proposed method to bootstrap the $H$ clusters by treating them as ``elements'', and the original sample within each selected cluster are replicated accordingly.
		\item[Step 2.] For each bootstrap cluster, apply the proposed method to bootstrap the cluster finite population independently.
	\end{enumerate}
	\item[Method II.] Wald-type method, and it is the same as the one discussed in Section \ref{sec: single stage}.
	
\end{itemize}

We conduct $1\,000$ Monte Carlo simulations for each  scenario. Table \ref{tab: two stage} summarizes the coverage rate and average length of the constructed 90\% confidence interval for the finite population mean. The coverage rates of the proposed bootstrap method are closer to 0.9 even when the sample size is limited. However, the coverage rates of the commonly used Wald-type method are not as good as the proposed bootstrap method. Specifically, the coverage rates of the Wald-type method are only around 0.86 for three scenarios, and it improves to 0.88 when sample size is large under Poisson sampling. The confidence intervals of the proposed bootstrap method are wider than those of the Wald-type method when sample size is small.

\begin{table}[!t]
	\centering
	\caption{Coverage rate and length of the 90\% confidence interval for $\bar{Y}$ by the proposed bootstrap method (Bootstrap) and the Wald-type method (Wald-type)  under two-stage sampling designs. The first column show the first-stage sample designs, that is, Poisson sampling  (Poisson) and PPS sampling (PPS), and  SRS is used in the second stage.  ``C.R.'' shows the coverage rate, and ``C.L.'' presents the Monte Carlo mean of the length for the 90\% confidence interval. }\label{tab: two stage}
	
	\begin{tabular}{cccrr}
		\hline
		Design & $(n_1,n_2)$ & Method & C.R. & C.L.\\
		\hline
	\multirow{5}{*}{Poisson} & \multirow{2}{*}{(5,10)} & Bootstrap & 0.90 & 114.08 \\
	&  & Wald-type & 0.85 &  98.58 \\
	&  &  &  &  \\
	& \multirow{2}{*}{(10,30)} & Bootstrap & 0.90 & 73.92 \\
	&  & Wald-type & 0.88 & 68.66 \\
	&  &  &  &  \\
	\multirow{5}{*}{PPS} & \multirow{2}{*}{(5,10)} & Bootstrap & 0.89 & 17.56 \\
	&  & Wald-type & 0.85 & 14.57 \\
	&  &  &  &  \\
	& \multirow{2}{*}{(10,30)} & Bootstrap & 0.90 & 9.40 \\
	&  & Wald-type & 0.86 & 8.24 \\
	
		\hline
	\end{tabular}
\end{table}

As in  Section \ref{sec: single stage}, we also compare those two methods in terms of approximating  $\mathbb{P}\{\tilde{V}^{-1/2}(\tilde{Y}-\bar{Y})\leq z\}$, which is obtained by 10\,000 Monte Carlo simulations. We set $z\in\{-0.5,-0.25,-0.1,0,0.1,0.25,0.5\}$. Table \ref{tab: distribution two} summarizes the simulation results. For both sample sizes, the proposed bootstrap method can approximate the target distribution well, but the performance of the Wald-type method is not as good as the proposed one especially when the sample size is small.

\begin{table}[!t]
	\centering
	\caption{Values of $P_z =\mathbb{P}\{\tilde{V}^{-1/2}(\tilde{Y}-\bar{Y})\leq z\}$, the normal approximation $\Phi(z)$ and the bootstrap approximation Boot$_z$ under two-stage sampling designs. The first column show the first-stage sample designs, that is, Poisson sampling  (Poisson) and PPS sampling (PPS), and  SRS is used in the second stage.  }
	\label{tab: distribution two}
	\begin{tabular}{cccccc}		\hline
		Design & $(n_1,n_2)$ & $z$ & $P_z$& $\Phi(z)$ &Boot$_z$\\
		\hline
		\multirow{15}{*}{Poisson} & \multirow{7}{*}{(5,10)} & -0.5 & 0.35 & 0.31 & 0.35 \\
		&  & -0.25 & 0.43 & 0.40 & 0.43 \\
		&  & -0.1 & 0.48 & 0.46 & 0.49 \\
		&  & 0 & 0.53 & 0.50 & 0.52 \\
		&  & 0.1 & 0.58 & 0.54 & 0.56 \\
		&  & 0.25 & 0.62 & 0.60 & 0.62 \\
		&  & 0.5 & 0.73 & 0.69 & 0.72 \\
		&  &  &  &  &  \\
		& \multirow{7}{*}{(10,30)} & -0.5 & 0.34 & 0.31 & 0.33 \\
		&  & -0.25 & 0.43 & 0.40 & 0.42 \\
		&  & -0.1 & 0.48 & 0.46 & 0.48 \\
		&  & 0 & 0.52 & 0.50 & 0.52 \\
		&  & 0.1 & 0.56 & 0.54 & 0.56 \\
		&  & 0.25 & 0.61 & 0.60 & 0.62 \\
		&  & 0.5 & 0.72 & 0.69 & 0.71 \\
		&  &  &  &  &  \\
		\multirow{15}{*}{PPS} & \multirow{7}{*}{(5,10)} & -0.5 & 0.32 & 0.31 & 0.32 \\
		&  & -0.25 & 0.41 & 0.40 & 0.41 \\
		&  & -0.1 & 0.47 & 0.46 & 0.47 \\
		&  & 0 & 0.51 & 0.50 & 0.51 \\
		&  & 0.1 & 0.55 & 0.54 & 0.55 \\
		&  & 0.25 & 0.61 & 0.60 & 0.62 \\
		&  & 0.5 & 0.71 & 0.69 & 0.71 \\
		&  &  &  &  &  \\
		& \multirow{7}{*}{(10,30)} & -0.5 & 0.32 & 0.31 & 0.32 \\
		&  & -0.25 & 0.41 & 0.40 & 0.41 \\
		&  & -0.1 & 0.46 & 0.46 & 0.47 \\
		&  & 0 & 0.50 & 0.50 & 0.50 \\
		&  & 0.1 & 0.54 & 0.54 & 0.54 \\
		&  & 0.25 & 0.60 & 0.60 & 0.60 \\
		&  & 0.5 & 0.69 & 0.69 & 0.69 \\

		\hline

	\end{tabular}
\end{table}


\section{Conclusion}
In this paper, we propose  bootstrap methods for Poisson sampling, SRS  and  PPS sampling, and we show that the proposed bootstrap methods are second-order accurate. The first step of the proposed bootstrap methods corresponds to an inverse sampling procedure by incorporating the sampling information. Since the proposed bootstrap method is based on an asymptotically pivotal statistic, it is necessary to estimate the variance of the design-unbiased estimator. Simulation results show that the proposed bootstrap method provides more conservative confidence interval than the Wald-type method when the sample size is small, and the  90\% confidence interval constructed by the proposed bootstrap method has a better coverage rate. Although the proposed bootstrap method is discussed under the single-stage sampling designs, simulation shows that it works well under some two-stage sampling designs, and Edgeworth expansion for two-stage sampling designs are under investigation. It may be extended to other complex sampling designs when the asymptotic distribution of the design-unbiased estimator exists, but the second-order accuracy may not be guaranteed. Besides, the proposed bootstrap method can be easily parallelized in practice.

\section{Acknowledgment}
We would like to thank Dr. J.~N.~K. Rao for the suggestion to discuss the simple random sampling and the two anonymous reviewers for the detailed and constructive comments.

\section{Supplement}
\subsection{Proofs}
\renewcommand{\theequation}{A.\arabic{equation}}
\setcounter{equation}{0}
For the purpose of clarity, we explicitly express $y_{N,i}$, $Y_N$, $I_{N,i}$, $\pi_{N,i}$ and $p_{N,i}$ for $y_i$, $Y$, $I_i$, $\pi_i$ and $p_i$ to highlight that they are indexed by $N$, and the same notation is used for other quantities without further mentioning. Denote $E(\cdot\mid \mathcal{F}_N)$ and $\mathrm{var}(\cdot\mid \mathcal{F}_N)$ to be the expectation and variance with respect to the probability measure of a specific sampling design, say $\mathbb{P}_{Poi}$ under Poisson sampling, $E_*(\cdot)$ and $\mathrm{var}_*(\cdot)$ to be the conditional mean and variance with respect to the multinomial distribution in the first steps of the proposed bootstrap method conditional on  the realized sample $\{y_{N,1},\ldots,y_{N,n}\}$, and $E_{**}(\cdot)$ and $\mathrm{var}_{**}(\cdot)$ to be the expectation and variance   with respect to the sampling design in the second step conditional on the bootstrap  finite population $ \mathcal{F}_N^*$.

\begin{proof}[Proof of Lemma \ref{lemma: almost sure variance poi}]
	
	Denote $X_{N,i}^{(1)} = {n_0}{N^{-2}}y_{N,i}^2(1-\pi_{N,i})\pi_{N,i}^{-2}(I_{N,i}-\pi_{N,i})$, then $n_0N^{-2}\big(\hat{V}_{N,Poi}-V_{N,Poi}\big)=\sum_{i=1}^{N}X_{N,i}^{(1)}$.
	Let $D_N^{(1)}$ be the event $\big\{\big\lvert \sum_{i=1}^NX_{N,i}^{(1)}\big\rvert > \epsilon\big\}$ for $N\in\mathbb{N}_+$, where $\epsilon\in(0,\infty)$ and $ \mathbb{N}_+$ is the set of positive integers. 
	
	By the Borel-Cantelli Lemma \citep[Thereom 7.2.2]{athreya2006measure}, to show \eqref{eq: lemma1 2}, it is enough to prove
	\begin{equation}
	\sum_{N=1}^\infty\mathbb{P}_{Poi}(D_N^{(1)})<\infty\label{eq: lemma 1 1}
	\end{equation}
	for  $\epsilon>0$. By the Markov's inequality \citep[Proposition 6.2.4]{athreya2006measure}, we have 
	\begin{eqnarray}
	\mathbb{P}_{Poi}(D_N^{(1)})&\leq&  {\epsilon^{-4}}E\left\{\left(\sum_{i=1}^NX_{N,i}^{(1)}\right)^4\mid \mathcal{F}_N\right\}\notag \\
	&=& {\epsilon^{-4}}\bigg[
	\sum_{i=1}^NE\left\{\left(X_{N,i}^{(1)}\right)^4\mid \mathcal{F}_N\right\}\notag \\
	& & ~~~~~~~~+ \sum_{(i,j)\in\Gamma_N}E\left\{\left(X_{N.i}^{(1)}\right)^2\mid \mathcal{F}_N\right\}E\left\{\left(X_{N,j}^{(1)}\right)^2\mid \mathcal{F}_N\right\}
	\bigg],\notag 
	\end{eqnarray}
	where the last equality holds since $E(X_{N,i}^{(1)}\mid \mathcal{F}_N) = 0$ for $i=1,\ldots,N$, and $X_{N,i}^{(1)}$ is independent of $X_{N,j}^{(1)}$ for $(i,j)\in\Gamma_N$ with $\Gamma_N=\{(i,j): i,j=1,\ldots,N \text{~and~} i\neq j\}$.
	
	Consider 
	\begin{eqnarray}
	& & E\left\{\left(X_{N,i}^{(1)}\right)^4\mid \mathcal{F}_N\right\} \notag \\
	&=& n_0^4N^{-8}y_{N,i}^8(1-\pi_{N,i})^5\pi_{N,i}^{-4}\{(1-\pi_{N,i})^3\pi_{N,i}^{-3}+1\} \notag \\
	&\leq&  C_{1,1} n_0^{-3}N^{-1}y_{N,i}^8, \label{eq: lemma1 3}
	\end{eqnarray}
	where $C_{1,1}$ is a positive constant determined by (C\ref{cond: C1poi}).
	
	Next, consider 
	\begin{eqnarray}
	E\left\{\left(X_{N,i}^{(1)}\right)^2\mid \mathcal{F}_N\right\} &=& n_0^2N^{-4}y_{N,i}^4\pi_{N,i}^{-3}(1-\pi_{N,i})^{3} \notag \\
	&\leq&  C_{1,2}n_0^{-1}N^{-1}y_{N,i}^4,\label{eq: lemma1 4}
	\end{eqnarray}
	where $C_{1,2}$ is a positive constant.
	
	Based on some algebra and (C\ref{cond: C4poi}), we have 
	\begin{eqnarray}
	\sum_{(i,j)\in\Gamma_N}y_{N,i}^4y_{N,j}^4 &=& O(N^2). \label{eq: sum 4,4}
	\end{eqnarray}
	
	By \eqref{eq: lemma1 3}, \eqref{eq: lemma1 4} and \eqref{eq: sum 4,4}, we have 
	\begin{eqnarray}
	\notag & & \mathbb{P}_{Poi}(D_N^{(1)}) \\
	&\leq&  \epsilon^{-4}C_{1,1}n_0^{-3}N^{-1}\sum_{i=1}^Ny_{N,i}^8 + \epsilon^{-4}C_{1,2}^2n_0^{-2}N^{-2}\sum_{(i,j)\in\Gamma_N}y_{N,i}^4y_{N,j}^4\notag\\
	&=& O(N^{-2\alpha})\notag
	\end{eqnarray}
	for any fixed $\epsilon>0$, where the last inequality holds by (C\ref{cond: C4poi}).
	Since $\alpha\in(2^{-1},1]$ by (C\ref{cond: C1poi}), we have proved \eqref{eq: lemma1 2} based on \eqref{eq: lemma 1 1}. 
	
	For $\mu_{N,Poi}^{(3)}=\sum_{i=1}^Ny_{N,i}^3(1-\pi_{N,i})\{(1-\pi_{N,i})^2\pi_{N,i}^{-2}-1\}$, we have
	\begin{eqnarray}
	\left|n_0^2N^{-3}\mu_{Poi}^{(3)}\right|&=&\left\lvert n_0^2N^{-3}\sum_{i=1}^Ny_{N,i}^3(1-\pi_{N,i})\{(1-\pi_{N,i})^2\pi_{N,i}^{-2}-1\}\right\rvert\notag\\
	&\leq& 2n_0^2N^{-3}\sum_{i=1}^N\lvert y_{N,i}\rvert^3\pi_{N,i}^{-2}\notag \\ 
	&\leq&2C_1^{-2}N^{-1}\sum_{=1}^N\lvert y_{N,i}\rvert^3 = O(1)\notag,
	\end{eqnarray}
	where the first inequality holds by $0<\pi_{N,i}<1$ and $0<1-\pi_{N,i}<1$,   the second inequality holds by (C\ref{cond: C1poi}), and the last equality holds by (C\ref{cond: C4poi}).
	
	Mentioned that
	\begin{eqnarray}
	\notag & & E\left(n_0^2N^{-3}\hat{\mu}_{N,Poi}^{(3)}\mid \mathcal{F}_N\right) \\
	&=&E\left[n_0^2N^{-3}\sum_{i=1}^ny_{N,i}^3(1-\pi_{N,i})\pi_{N,i}^{-1}\{(1-\pi_{N,i})^{2}\pi_{N,i}^{-2}-1\}\mid \mathcal{F}_N\right] \notag\\ &=& \notag n_0^2N^{-3}\sum_{i=1}^Ny_{N,i}^3(1-\pi_{N,i})\{(1-\pi_{N,i})^{2}\pi_{N,i}^{-2}-1\} \\
	& = & n_0^2N^{-3}\mu_{N,Poi}^{(3)} \notag
	\end{eqnarray}
	and
	\begin{eqnarray}
	\notag & & \mathrm{var}\left(\hat{\mu}_{N,Poi}^{(3)}\mid \mathcal{F}_N\right) \\
	&=&\mathrm{var}\left[n_0^2N^{-3}\sum_{i=1}^ny_{N,i}^3(1-\pi_{N,i})\pi_{N,i}^{-1}\{(1-\pi_{N,i})^{2}\pi_{N,i}^{-2}-1\}\mid \mathcal{F}_N\right]\notag \\ 
	&\leq&  4n_0^4N^{-6}\sum_{i=1}^Ny_{N,i}^6\pi_{N,i}^{-5} \notag \\
	& = & O(n_0^{-1}), \notag
	\end{eqnarray}
	we can obtain that $n_0^2N^{-3}(\hat{\mu}_{N,Poi}^{(3)}-\mu_{N,Poi}^{(3)})=O_p(n_0^{-1/2})$ by the Markov's inequality. The results concerning $\tau_{N,Poi}^{(3)}$ and $\hat{\tau}_{N,Poi}^{(3)}$ can be proved similarly and   is omitted here. 
	Thus, we finalize the proof of Lemma \ref{lemma: almost sure variance poi}.
\end{proof}

The following lemmas are useful in establishing Theorem \ref{theorem: edgeworth expansion Poisson} and \ref{theorem: e 2}.

\begin{lemma}\label{lem_01}
	Denote $X_{N,i}=V_{N,Poi}^{-1/2}y_{N,i}\pi_{N,i}^{-1}(I_{N,i}-\pi_{N,i})$ for $i=1,\ldots,N$. Let $\Delta_{N,1}=\sum_{i=1}^{N}X_{N,i}$ and $\phi_{\Delta_{N,1}}(t)=\E\left\{\exp(\iota t\Delta_{N,1})\mid \mathcal{F}_N\right\}$ be the characteristic function (c.f.) of $\Delta_{N,1}$, where $\iota$ is the imaginary unit. Then under conditions (C\ref{cond: C1poi})--(C\ref{cond: C4poi}),
	\begin{eqnarray}\label{lem_01_eq_01}
	\left|\phi_{\Delta_{N,1}}(t)\right| & \leq & \exp(-t^2/3), \\
	\left|\phi_{\Delta_{N,1}}(t)-\exp(-t^2/2)\right| & \leq & 16|t|^3V_{N,Poi}^{-3/2}\nu_{N,Poi}^{(3)}\exp(-t^2/3)
	\end{eqnarray}
	for all $|t|\leq V_{N,Poi}^{3/2}/\left(4\nu_{N,Poi}^{(3)}\right)$, where $\nu_{N,Poi}^{(3)}=\sum_{i=1}^{N}|y_{N,i}|^3(1-\pi_{N,i})\big\{(1-\pi_{N,i})^2\pi_{N,i}^{-2}+1\big\}$. Furthermore, 
	\begin{eqnarray}\label{lem_01_eq_03}
	& &\left|\phi_{\Delta_{N,1}}(t)-\exp(-t^2/2)-6^{-1}(\iota t)^3V_{N,Poi}^{-3/2}\mu_{N,Poi}^{(3)}\exp(-t^2/2)\right|\notag \\
	&\leq & C_{2,1}\exp(-19t^2/48)\left(t^4n_0^{-1}+t^6n_0^{-1}\right)
	\end{eqnarray}
	for all $|t|\leq\min\left(\big\{\max_{1\leq i\leq N}\E(X_{N,i}^2\mid \mathcal{F}_N)\big\}^{-1/2},V_{N,Poi}^{3/2}/\left(4\nu_{N,Poi}^{(3)}\right)\right)$, where $C_{2,1}$ is a positive constant and recall that $\mu_{N,Poi}^{(3)}=\sum_{i=1}^Ny_{N,i}^3(1-\pi_{N,i})\{(1-\pi_{N,i})^2\pi_{N,i}^{-2}-1\}$.
\end{lemma}

\begin{proof}
	As $I_{N,i}\sim\mathrm{Ber}(\pi_{N,i})$, $\E(X_{N,i}\mid \mathcal{F}_N)=0$ and $\E(X_{N,i}^2\mid \mathcal{F}_N)=V_{N,Poi}^{-1}y_{N,i}^2(1-\pi_{N,i})\pi_{N,i}^{-1}$ for $i=1,\ldots,N$. In addition,
	\begin{eqnarray}
	\E(X_{N,i}^3\mid \mathcal{F}_N) = V_{N,Poi}^{-3/2}y_{N,i}^3(1-\pi_{N,i})\big\{(1-\pi_{N,i})^2\pi_{N,i}^{-2}-1\big\} \notag
	\end{eqnarray}
	and
	\begin{eqnarray}
	\E\{|X_{N,i}|^3\mid \mathcal{F}_N\} = V_{N,Poi}^{-3/2}|y_{N,i}|^3(1-\pi_{N,i})\big\{(1-\pi_{N,i})^2\pi_{N,i}^{-2}+1\big\}. \notag
	\end{eqnarray}
	Thus, $\sum_{i=1}^{N}\E(X_{N,i}^2\mid \mathcal{F}_N)=1$, $\sum_{i=1}^{N}\E(X_{N,i}^3\mid \mathcal{F}_N)=V_{N,Poi}^{-3/2}\mu_{N,Poi}^{(3)}=O(n_0^{-1/2})$ by (C\ref{con: variance}) and Lemma \ref{lemma: almost sure variance poi}. In addition, $\nu_{N,Poi}^{(3)}=\sum_{i=1}^{N}|y_{N,i}|^3(1-\pi_{N,i})\big\{(1-\pi_{N,i})^2\pi_{N,i}^{-2}+1\big\}=O(n_0^{-2}N^{-3})$, which implies $\sum_{i=1}^{N}\E\{|X_{N,i}|^3\mid \mathcal{F}_N\}=V_{N,Poi}^{-3/2}\nu_{N,Poi}^{(3)}=O(n_0^{-1/2}).$
	By the fact that $\max_{1\leq i\leq N}\E\{|X_{N,i}|^3\mid \mathcal{F}_N\}\leq\sum_{i=1}^{N}\E\{|X_{N,i}|^3\mid \mathcal{F}_N\}=O(n_0^{-1/2})$, $\E\{|X_{N,i}|^3\mid \mathcal{F}_N\}<\infty$ for $i=1,\ldots,N$. By Lemma 5.1 of \citet{Petrov1995},
	\begin{eqnarray}
	\left|\phi_{\Delta_{N,1}}(t)\right| & \leq & \exp(-t^2/3), \notag \\
	\left|\phi_{\Delta_{N,1}}(t)-\exp(-t^2/2)\right| & \leq & 16|t|^3V_{N,Poi}^{-3/2}\nu_{N,Poi}^{(3)}\exp(-t^2/3) \notag
	\end{eqnarray}
	for all $|t|\leq V_{N,Poi}^{3/2}/\left(4\nu_{N,Poi}^{(3)}\right)$.
	
	It remains to show \eqref{lem_01_eq_03}. Denote $\phi_{X_{N,i}}(t)=\E\left\{\exp(\iota tX_{N,i})\mid \mathcal{F}_N\right\}$ as the characteristic function of $X_{N,i}$. Note that for any complex numbers $z$, $w$, $|\exp(z)-1-w|\leq(|z-w|+|w|^2/2)\exp\{\max(|z|,|w|)\}$, it follows that
	\begin{eqnarray}
	\notag && \left|\phi_{\Delta_{N,1}}(t)-\exp(-t^2/2)-6^{-1}(\iota t)^3V_{N,Poi}^{-3/2}\mu_{N,Poi}^{(3)}\exp(-t^2/2)\right| \\
	\notag &= & \left|\prod_{i=1}^{N}\phi_{X_{N,i}}(t)-\exp(-t^2/2)-6^{-1}(\iota t)^3V_{N,Poi}^{-3/2}\mu_{N,Poi}^{(3)}\exp(-t^2/2)\right| \\
	\notag &= & \exp(-t^2/2)\left|\exp\left[\sum_{i=1}^{N}\log\{\phi_{X_{N,i}}(t)\}+t^2/2\right]-1-6^{-1}(\iota t)^3V_{N,Poi}^{-3/2}\mu_{N,Poi}^{(3)}\right| \\
	\notag &\leq & \exp(-t^2/2)\bigg[\bigg|\sum_{i=1}^{N}\log\{\phi_{X_{N,i}}(t)\}+t^2/2-6^{-1}(\iota t)^3V_{N,Poi}^{-3/2}\mu_{N,Poi}^{(3)}\bigg| \\
	\notag & & ~~~~~~~~~~~~~~~~~~~~+2^{-1}\left|6^{-1}(\iota t)^3V_{N,Poi}^{-3/2}\mu_{N,Poi}^{(3)}\right|^2\bigg] \\
	\notag & & \times \exp\left\{\max\left(\left|\sum_{i=1}^{N}\log\{\phi_{X_{N,i}}(t)\}+t^2/2\right|, \left|6^{-1}(\iota t)^3V_{N,Poi}^{-3/2}\mu_{N,Poi}^{(3)}\right|\right)\right\}.
	\end{eqnarray}	
	By Lemma 11.4.3 of \citet{athreya2006measure}, 
	\begin{eqnarray}\label{lem_01_eq_04}
	&& \left|\sum_{i=1}^{N}\log\{\phi_{X_{N,i}}(t)\}+t^2/2-6^{-1}(\iota t)^3V_{N,Poi}^{-3/2}\mu_{N,Poi}^{(3)}\right| \\
	\notag &\leq & \sum_{i=1}^{N}\big|\log\{\phi_{X_{N,i}}(t)\}-2^{-1}(\iota t)^2\E(X_{N,i}^2\mid \mathcal{F}_N) -6^{-1}(\iota t)^3\E(X_{N,i}^3\mid \mathcal{F}_N)\big| \\
	\notag & \leq & \sum_{i=1}^{N}\left(\E\left[\min\left\{|tX_{N,i}|^3/3, (tX_{N,i})^4/24\right\}\mid \mathcal{F}_N\right]+t^4\{\E(X_{N,i}^2\mid \mathcal{F}_N)\}^2/4\right) \\
	& \leq & C_{2,2}t^4n_0^{-1} \notag
	\end{eqnarray}
	for all $|t|\leq\big\{\max_{1\leq i\leq N}\E(X_{N,i}^2\mid \mathcal{F}_N)\big\}^{-1/2}$, where $C_{2,2}$ is a positive constant and the last inequality is by the fact that 
	\begin{eqnarray}
	\notag & & \sum_{i=1}^{N}\E(X_{N,i}^4\mid \mathcal{F}_N) \\
	\notag & = & V_{N,Poi}^{-2}\sum_{i=1}^{N}y_{N,i}^4(1-\pi_{N,i})\{(1-\pi_{N,i})^3\pi_{N,i}^{-3}+1\} \\
	\notag & = & O(n_0^{-1}).
	\end{eqnarray}
	Similarly, by Lemma 11.4.3 of \citet{athreya2006measure}, 
	\begin{eqnarray}
	\notag && \left|\sum_{i=1}^{N}\log\{\phi_{X_{N,i}}(t)\}+t^2/2\right| \\
	\notag &\leq & \sum_{i=1}^{N}\big|\log\{\phi_{X_{N,i}}(t)\}-2^{-1}(\iota t)^2\E(X_{N,i}^2\mid \mathcal{F}_N)\big| \\
	\notag &\leq & 5|t|^3\sum_{i=1}^{N}\E\{|X_{N,i}|^3\mid \mathcal{F}_N\}/12=5|t|^3V_{N,Poi}^{-3/2}\nu_{N,Poi}^{(3)}/12
	\end{eqnarray}
	for all $|t|\leq\big\{\max_{1\leq i\leq N}\E(X_{N,i}^2\mid \mathcal{F}_N)\big\}^{-1/2}$. 
	
	Thus, if $|t|\leq\min\left(\big\{\max_{1\leq i\leq N}\E(X_{N,i}^2\mid \mathcal{F}_N)\big\}^{-1/2},V_{N,Poi}^{3/2}/\left(4\nu_{N,Poi}^{(3)}\right)\right)$,
	\begin{eqnarray}\label{lem_01_eq_05}
	& & \max\left(\left|\sum_{i=1}^{N}\log\{\phi_{X_{N,i}}(t)\}+t^2/2\right|, \left|6^{-1}(\iota t)^3V_{N,Poi}^{-3/2}\mu_{N,Poi}^{(3)}\right|\right) \\
	& \leq & 5t^2/48. \notag
	\end{eqnarray}
	
	Mentioned that
	\begin{align}\label{lem_01_eq_06}
	2^{-1}\left|6^{-1}(\iota t)^3V_{N,Poi}^{-3/2}\mu_{N,Poi}^{(3)}\right|^2 \leq C_{2,3}t^6n_0^{-1}
	\end{align}
	for some positive constant $C_{2,3}$. 
	
	Finally, by \eqref{lem_01_eq_04}, \eqref{lem_01_eq_05} and \eqref{lem_01_eq_06}, it follows that
	\begin{eqnarray}
	\notag && \left|\phi_{\Delta_{N,1}}(t)-\exp(-t^2/2)-6^{-1}(\iota t)^3V_{N,Poi}^{-3/2}\mu_{N,Poi}^{(3)}\exp(-t^2/2)\right| \\
	\notag &\leq & \exp(-t^2/2)\left(C_{2,2}t^4n_0^{-1}+C_{2,3}t^6n_0^{-1}\right)\exp(5t^2/48) \\
	\notag &\leq & C_{2,1}\exp(-19t^2/48)\left(t^4n_0^{-1}+t^6n_0^{-1}\right)
	\end{eqnarray} 
	for all $|t|\leq\min\left(\big\{\max_{1\leq i\leq N}\E(X_{N,i}^2\mid \mathcal{F}_N)\big\}^{-1/2},V_{N,Poi}^{3/2}/\left(4\nu_{N,Poi}^{(3)}\right)\right)$. This concludes the proof of this Lemma.
\end{proof}

\begin{lemma}\label{lem_02}
	Denote 
	\begin{eqnarray}
	\Delta_{N,2}=-2^{-1}V_{N,Poi}^{-3/2}\sum_{(i,j)\in\Gamma_N}y_{N,i}y_{N,j}^2(1-\pi_{N,j})\pi_{N,i}^{-1}\pi_{N,j}^{-2}(I_{N,i}-\pi_{N,i})(I_{N,j}-\pi_{N,j}), \notag
	\end{eqnarray}
	then under conditions (C\ref{cond: C1poi})--(C\ref{cond: C4poi}),
	\begin{eqnarray}
	\notag && \E\left\{\Delta_{N,2}\exp(\iota t\Delta_{N,1})\mid \mathcal{F}_N\right\} \\
	\notag &= & 2^{-1}t^2\exp(-t^2/2)V_{N,Poi}^{-5/2}\Theta_{N,Poi}^{(2,3)}+\varpi_{N,1}(t),
	\end{eqnarray}
	for all $|t|\leq \big\{\max_{1\leq i\leq N}\E(X_{N,i}^2\mid \mathcal{F}_N)\big\}^{-1/2}$, where  $X_{N,i}=V_{N,Poi}^{-1/2}y_{N,i}\pi_{N,i}^{-1}(I_{N,i}-\pi_{N,i})$, $\Delta_{N,1}=\sum_{i=1}^{N}X_{N,i}$,  $\Gamma_N=\{(i,j): i,j=1,\ldots,N \text{~and~} i\neq j\}$,
	\begin{eqnarray}
	\notag \Theta_{N,Poi}^{(2,3)}= \sum_{(i,j)\in\Gamma_N}y_{N,i}^2y_{N,j}^3(1-\pi_{N,i})(1-\pi_{N,j})^2\pi_{N,i}^{-1}\pi_{N,j}^{-2},
	\end{eqnarray}
	and
	$\varpi_{N,1}(t)$ satisfies
	\begin{eqnarray}
	\notag && |\varpi_{N,1}(t)| \\
	\notag &\leq & C_{3,1}\exp\left\{-t^2/2+2|t|^3V_{N,Poi}^{-3/2}\nu_{N,Poi}^{(3)}/3+t^2\max_{1\leq i\leq N}\E(X_{N,i}^2\mid \mathcal{F}_N)\right\} \\
	\notag & & ~~~~~~~~~~~~~~~~~~~~~~~~~~~~~~~~~~~\times\left(|t|^3n_0^{-1}+t^4n_0^{-3/2}+|t|^5n_0^{-1}\right),
	\end{eqnarray}
	with $C_{3,1}$ being a positive constant and $\nu_{N,Poi}^{(3)}=\sum_{i=1}^{N}|y_{N,i}|^3(1-\pi_{N,i})\big\{(1-\pi_{N,i})^2\pi_{N,i}^{-2}+1\big\}$.
\end{lemma}

\begin{proof}

	First, write
	\begin{eqnarray}
	\notag \E\left\{(I_{N,i}-\pi_{N,i})\exp(\iota tX_{N,i})\mid \mathcal{F}_N\right\}=\iota tV_{N,Poi}^{-1/2}y_{N,i}(1-\pi_{N,i})+\varpi_{N1,i}(t),
	\end{eqnarray}
	for $i=1,\ldots,N$, where
	\begin{eqnarray}
	\notag & & |\varpi_{N1,i}(t)| \\
	\notag & = & |\E\left\{(I_{N,i}-\pi_{N,i})\exp(\iota tX_{N,i})\mid \mathcal{F}_N\right\}-\iota tV_{N,Poi}^{-1/2}y_{N,i}(1-\pi_{N,i})| \\
	\notag & = & (1-\pi_{N,i})\pi_{N,i}\bigg|\bigg[\exp\left\{\iota tV_{N,Poi}^{-1/2}y_{N,i}(1-\pi_{N,i})\pi_{N,i}^{-1}\right\}-1 \\
	\notag & & ~~~~~~~~~~~~~~~~~~~~~~~~~~~~~~~~~~~~~~~~~~~-\iota tV_{N,Poi}^{-1/2}y_{N,i}(1-\pi_{N,i})\pi_{N,i}^{-1}\bigg] \\
	\notag & & ~~~~~~~~~~~~~~~~~~~~~~~ -\left\{\exp\left(-\iota tV_{N,Poi}^{-1/2}y_{N,i}\right)-1+\iota tV_{N,Poi}^{-1/2}y_{N,i}\right\}\bigg| \\
	\notag & \leq & 2^{-1}(1-\pi_{N,i})\pi_{N,i}\left\{\left|tV_{N,Poi}^{-1/2}y_{N,i}(1-\pi_{N,i})\pi_{N,i}^{-1}\right|^2+\left|tV_{N,Poi}^{-1/2}y_{N,i}\right|^2\right\} \\
	\notag & \leq & C_{3,2}t^2N^{-1}y_{N,i}^2,
	\end{eqnarray}
	and $C_{3,2}$ is a positive constant. The last but one inequality is due to the fact that for any real number $x$, $|\exp(\iota x)-1-\iota x|\leq|x|^2/2$. As a consequence, for any $(i,j)\in\Gamma_N$,
	\begin{eqnarray}
	\notag & & \E\left\{(I_{N,i}-\pi_{N,i})\exp(\iota tX_{N,i})\mid \mathcal{F}_N\right\}\E\left\{(I_{N,j}-\pi_{N,j})\exp(\iota tX_{N,j})\mid \mathcal{F}_N\right\} \\
	\notag & = & \left\{\iota tV_{N,Poi}^{-1/2}y_{N,i}(1-\pi_{N,i})+\varpi_{N1,i}(t)\right\}\left\{\iota tV_{N,Poi}^{-1/2}y_{N,j}(1-\pi_{N,j})+\varpi_{N1,j}(t)\right\} \\
	\notag & = & -t^2V_{N,Poi}^{-1}y_{N,i}y_{N,j}(1-\pi_{N,i})(1-\pi_{N,j})+\varpi_{N2,ij}(t),
	\end{eqnarray}
	where
	\begin{eqnarray}
	\notag && |\varpi_{N2,ij}(t)| \\
	\notag &\leq & \left|tV_{N,Poi}^{-1/2}y_{N,i}(1-\pi_{N,i})\varpi_{N1,j}(t)\right|+\left|tV_{N,Poi}^{-1/2}y_{N,j}(1-\pi_{N,j})\varpi_{N1,i}(t)\right| \\
	\notag & & +\left|\varpi_{N1,i}(t)\varpi_{N1,j}(t)\right| \\
	\notag &\leq & C_{3,3}\left\{|t|^3n_0^{1/2}N^{-2}(y_{N,i}^2|y_{N,j}|+|y_{N,i}|y_{N,j}^2)+t^4N^{-2}y_{N,i}^2y_{N,j}^2\right\}
	\end{eqnarray}
	with $C_{3,3}$ being a positive constant.
	
	Denote $\phi_{X_{N,i}}(t)=\E\left\{\exp(\iota tX_{N,i})\mid \mathcal{F}_N\right\}$ for $i=1,\ldots,N$. By the same technique as in the proof of Lemma 5.1 of \citet{Petrov1995}, we can show that
	\begin{eqnarray}
	\notag & & \left|\prod_{k\neq i,j}\phi_{X_{N,k}}(t)\right| \\
	\notag & \leq & \exp\left\{-t^2\sum_{k\neq i,j}\E(X_{N,k}^2\mid \mathcal{F}_N)/2+2|t|^3\sum_{k\neq i,j}\E\{|X_{N,k}|^3\mid \mathcal{F}_N\}/3\right\}.
	\end{eqnarray}
	
	Using the inequality $|\exp(z)-1|\leq|z|\exp(|z|)$ for all complex number $z$, we can obtain that
	\begin{eqnarray}
	\notag & & \left|\prod_{k\neq i,j}\phi_{X_{N,k}}(t)-\exp(-t^2/2)\right| \\
	\notag &= & \exp(-t^2/2)\left|\exp\left[\sum_{k\neq i,j}\log\{\phi_{X_{N,k}}(t)\}+t^2/2\right]-1\right| \\
	\notag &\leq & \exp(-t^2/2)\left|\sum_{k\neq i,j}\log\{\phi_{X_{N,k}}(t)\}+t^2/2\right| \exp\left[\left|\sum_{k\neq i,j}\log\{\phi_{X_{N,k}}(t)\}+t^2/2\right| \right]
	\end{eqnarray}
	By Lemma 11.4.3 of \citet{athreya2006measure}, 
	\begin{eqnarray}
	\notag & & \left|\sum_{k\neq i,j}\log\{\phi_{X_{N,k}}(t)\}+t^2/2\right| \\
	\notag & \leq & \sum_{k\neq i,j}\big|\log\{\phi_{X_{N,k}}(t)\}-2^{-1}(\iota t)^2\E(X_{N,k}^2\mid \mathcal{F}_N)\big| \\
	\notag & & ~~~~~~~~~~~~~~~~~~~~~~~~~~~~~~+2^{-1}t^2\E(X_{N,i}^2\mid \mathcal{F}_N)+2^{-1}t^2\E(X_{N,j}^2\mid \mathcal{F}_N) \\
	\notag & \leq & 5|t|^3\sum_{k\neq i,j}\E\{|X_{N,k}|^3\mid \mathcal{F}_N\}/12+2^{-1}t^2\left\{\E(X_{N,i}^2\mid \mathcal{F}_N)+\E(X_{N,j}^2\mid \mathcal{F}_N)\right\}.
	\end{eqnarray}
	for all $|t|\leq\big\{\max_{1\leq i\leq N}\E(X_{N,i}^2\mid \mathcal{F}_N)\big\}^{-1/2}$. 
	Thus,
	we obtain 
	\begin{eqnarray}
	\notag & & \left|\prod_{k\neq i,j}\phi_{X_{N,k}}(t)-\exp(-t^2/2)\right| \\
	\notag & \leq & \exp\left\{-t^2\sum_{k\neq i,j}\E(X_{N,k}^2\mid \mathcal{F}_N)/2+5|t|^3\sum_{k\neq i,j}\E\{|X_{N,k}|^3\mid \mathcal{F}_N\}/12\right\}\\
	\notag & & ~~~~~~\times\bigg[5|t|^3\sum_{k\neq i,j}\E\{|X_{N,k}|^3\mid \mathcal{F}_N\}/12 \\
	\notag & & ~~~~~~~~~~~~~~~+2^{-1}t^2\left\{\E(X_{N,i}^2\mid \mathcal{F}_N)+\E(X_{N,j}^2\mid \mathcal{F}_N)\right\}\bigg]
	\end{eqnarray}
	for all $|t|\leq\big\{\max_{1\leq i\leq N}\E(X_{N,i}^2\mid \mathcal{F}_N)\big\}^{-1/2}$.
	
	Finally, we have
	\begin{eqnarray}
	\notag & & \E\left\{\Delta_{N,2}\exp(\iota t\Delta_{N,1})\mid \mathcal{F}_N\right\} \\
	\notag &= & -2^{-1}V_{N,Poi}^{-3/2}\sum_{(i,j)\in\Gamma_N}y_{N,i}y_{N,j}^2(1-\pi_{N,j})\pi_{N,i}^{-1}\pi_{N,j}^{-2} \\
	\notag & & ~~~~~~~~~~~~~~~~~~~~~~~\times \E\left\{(I_{N,i}-\pi_{N,i})(I_{N,j}-\pi_{N,j})\exp(\iota t\Delta_{N,1})\mid \mathcal{F}_N\right\} \\
	\notag &= & -2^{-1}V_{N,Poi}^{-3/2}\sum_{(i,j)\in\Gamma_N}y_{N,i}y_{N,j}^2(1-\pi_{N,j})\pi_{N,i}^{-1}\pi_{N,j}^{-2}\prod_{k\neq i,j}\phi_{X_{N,k}}(t) \\
	\notag && ~~~~~~~~~~~~~~~~~~~~~~~\times \E\left\{(I_{N,i}-\pi_{N,i})\exp(\iota tX_{N,i})\mid \mathcal{F}_N\right\} \\
	\notag & &~~~~~~~~~~~~~~~~~~~~~~~\times \E\left\{(I_{N,j}-\pi_{N,j})\exp(\iota tX_{N,j})\mid \mathcal{F}_N\right\} \\
	\notag &= & -2^{-1}V_{N,Poi}^{-3/2}\sum_{(i,j)\in\Gamma_N}y_{N,i}y_{N,j}^2(1-\pi_{N,j})\pi_{N,i}^{-1}\pi_{N,j}^{-2}\prod_{k\neq i,j}\phi_{X_{N,k}}(t) \\
	\notag && ~~~~~~~~~~~~~~~ \times \left\{-t^2V_{N,Poi}^{-1}y_{N,i}y_{N,j}(1-\pi_{N,i})(1-\pi_{N,j})+\varpi_{N2,ij}(t)\right\} \\
	\notag &= & 2^{-1}t^2V_{N,Poi}^{-5/2}\sum_{(i,j)\in\Gamma_N}y_{N,i}^2y_{N,j}^3(1-\pi_{N,i})(1-\pi_{N,j})^2\pi_{N,i}^{-1}\pi_{N,j}^{-2}\prod_{k\neq i,j}\phi_{X_{N,k}}(t) \\
	\notag && ~~-2^{-1}V_{N,Poi}^{-3/2}\sum_{i\neq j}y_{N,i}y_{N,j}^2(1-\pi_{N,j})\pi_{N,i}^{-1}\pi_{N,j}^{-2}\prod_{k\neq i,j}\phi_{X_{N,k}}(t)\varpi_{N2,ij}(t) \\
	\notag &= & 2^{-1}t^2\exp(-t^2/2)V_{N,Poi}^{-5/2}\Theta_{N,Poi}^{(2,3)}+\varpi_{N,1}(t)
	\end{eqnarray}
	for all $|t|\leq\big\{\max_{1\leq i\leq N}\E(X_{N,i}^2\mid \mathcal{F}_N)\big\}^{-1/2}$ and $\varpi_{N,1}(t)$ satisfies
	\begin{eqnarray}
	\notag && |\varpi_{N,1}(t)| \\
	\notag &\leq & C_{3,4}\exp\left\{-t^2/2+2|t|^3\sum_{i=1}^{N}\E\{|X_{N,i}|^3\mid \mathcal{F}_N\}/3+t^2\max_{1\leq i\leq N}\E(X_{N,i}^2\mid \mathcal{F}_N)\right\} \\
	\notag && ~~~~~~~~~~~~~~\times \bigg\{|t|^3n_0^{-1}N^{-2}\sum_{i\neq j}\left(|y_{N,i}|^3|y_{N,j}|^3+y_{N,i}^2y_{N,j}^4\right)\\
	\notag & & ~~~~~~~~~~~~~~~~~~~+t^4n_0^{-3/2}N^{-2}\sum_{i\neq j}|y_{N,i}|^3y_{N,j}^4 \\
	\notag && ~~~~~~~~~~~~~~~~~~~+t^4n_0^{-1/2}N^{-3}\sum_{i\neq j}(y_{N,i}^4|y_{N,j}|^3+y_{N,i}^2|y_{N,j}|^5) \\
	\notag && ~~~~~~~~~~~~~~~~~~~+|t|^5n_0^{-1}\bigg(N^{-2}\sum_{i\neq j}y_{N,i}^2|y_{N,j}|^3\bigg)\bigg(N^{-1}\sum_{i=1}^{N}|y_{N,i}|^3\bigg)\bigg\} \\
	\notag &\leq & C_{3,1}\exp\left\{-t^2/2+2|t|^3V_{N,Poi}^{-3/2}\nu_{N,Poi}^{(3)}/3+t^2\max_{1\leq i\leq N}\E(X_{N,i}^2\mid \mathcal{F}_N)\right\} \\
	\notag & & ~~~~~~~~~~~~~~~~~~~~~~~~~~~~~~~~~~~~~~~~~~~~~\times\left(|t|^3n_0^{-1}+t^4n_0^{-3/2}+|t|^5n_0^{-1}\right),
	\end{eqnarray}
	where $C_{3,4}$ is a positive constant.
\end{proof}

\begin{lemma}\label{lem_03}
	Denote $\hat{Y}_{N,Poi}=\sum_{i=1}^{N}y_{N,i}\pi_{N,i}^{-1}I_{N,i}$, $Y_N=\sum_{i=1}^{N}y_{N,i}$ and $\hat{V}_{N,Poi}=\sum_{i=1}^{N}y_{N,i}^2(1-\pi_{N,i})\pi_{N,i}^{-2}I_{N,i}$, then under conditions (C\ref{cond: C1poi})--(C\ref{cond: C4poi}),
	\begin{eqnarray}
	\notag \hat{V}_{N,Poi}^{-1/2}\left(\hat{Y}_{N,Poi}-Y_N\right)=  \Delta_{N,1}+\Delta_{N,2}+\Delta_{N,3}+\Delta_{N,4},
	\end{eqnarray}
	where 
	\begin{eqnarray}
	\notag \Delta_{N,1} &= & V_{N,Poi}^{-1/2}\sum_{i=1}^{N}y_{N,i}\pi_{N,i}^{-1}(I_{N,i}-\pi_{N,i}), \\
	\notag \Delta_{N,2} &= & -2^{-1}V_{N,Poi}^{-3/2}\sum_{(i,j)\in\Gamma_N}y_{N,i}y_{N,j}^2(1-\pi_{N,j})\pi_{N,i}^{-1}\pi_{N,j}^{-2}(I_{N,i}-\pi_{N,i})(I_{N,j}-\pi_{N,j}), \\
	\notag \Delta_{N,3} & = & -2^{-1}V_{N,Poi}^{-3/2}\sum_{i=1}^{N}y_{N,i}^3(1-\pi_{N,i})^2\pi_{N,i}^{-2},
	\end{eqnarray}
	and recall that $\Gamma_N=\{(i,j): i,j=1,\ldots,N \text{~and~} i\neq j\}$. In addition, $\Delta_{N,4}$ satisfies
	$$ \P_{Poi}\left(|\Delta_{N,4}|\geq n_0^{-1/2}(\log n_0)^{-1}\right)=o(n_0^{-1/2}). $$
\end{lemma}

\begin{proof}
	Denote $\Lambda_{N,1}=V_{N,Poi}^{-1}\sum_{i=1}^{N}y_{N,i}^2(1-\pi_{N,i})\pi_{N,i}^{-2}(I_{N,i}-\pi_{N,i})$. Mentioned that $\E(\Lambda_{N,1}\mid \mathcal{F}_N)=0$ and
	\begin{eqnarray}
	\notag \E\left(\Lambda_{N,1}^2\mid \mathcal{F}_N\right) = V_{N,Poi}^{-2}\sum_{i=1}^{N}y_{N,i}^4(1-\pi_{N,i})^3\pi_{N,i}^{-3}=O(n_0^{-1}),
	\end{eqnarray}
	we have that $\Lambda_{N,1}=O _p(n_0^{-1/2})$. In addition,
	\begin{eqnarray}
	\notag & & \E(\Lambda_{N,1}^4\mid \mathcal{F}_N) \\
	\notag & = & \E\left[\bigg\{V_{N,Poi}^{-1}\sum_{i=1}^{N}y_{N,i}^2(1-\pi_{N,i})\pi_{N,i}^{-2}(I_{N,i}-\pi_{N,i})\bigg\}^4\mid \mathcal{F}_N\right] \\
	\notag & = & V_{N,Poi}^{-4}\bigg[\sum_{i=1}^{N}y_{N,i}^8(1-\pi_{N,i})^5\pi_{N,i}^{-4}\left\{(1-\pi_{N,i})^3\pi_{N,i}^{-3}+1\right\} \\
	\notag & & ~~~~~~~~~~~~~~+\sum_{(i,j)\in\Gamma_N}y_{N,i}^4y_{N,j}^4(1-\pi_{N,i})^3(1-\pi_{N,j})^3\pi_{N,i}^{-3}\pi_{N,j}^{-3}\bigg] \\
	\notag  &= & O(n_0^{-2}).
	\end{eqnarray}
	By some algebra,
	\begin{eqnarray}
	\notag & & \hat{V}_{N,Poi}^{-1/2}\left(\hat{Y}_{N,Poi}-Y_N\right) \\
	\notag & = & \left\{V_{N,Poi}^{-1/2}\sum_{i=1}^{N}y_{N,i}\pi_{N,i}^{-1}(I_{N,i}-\pi_{N,i})\right\}\left(1+\Lambda_{N,1}\right)^{-1/2} \\
	\notag & = & \left\{V_{N,Poi}^{-1/2}\sum_{i=1}^{N}y_{N,i}\pi_{N,i}^{-1}(I_{N,i}-\pi_{N,i})\right\}\{1-2^{-1}\Lambda_{N,1} +O_p(\Lambda_{N,1}^2)\}.
	\end{eqnarray}
	Use the notations of $\Delta_{N,1}$, $\Delta_{N,2}$ and $\Delta_{N,3}$, we have
	\begin{eqnarray}
	\notag & & \hat{V}_{N,Poi}^{-1/2}\left(\hat{Y}_{N,Poi}-Y_N\right) \\
	\notag & = & \Delta_{N,1}\left\{1-2^{-1}\Lambda_{N,1}+O_p(\Lambda_{N,1}^2)\right\} \\
	\notag & = & \Delta_{N,1}-2^{-1}\Delta_{N,1}\Lambda_{N,1}+O_p(\Delta_{N,1}\Lambda_{N,1}^2) \\
	\notag & = & \Delta_{N,1}+\Delta_{N,2}+\Delta_{N,3}+\Lambda_{N,2}+O_p(\Delta_{N,1}\Lambda_{N,1}^2)
	\end{eqnarray}
	where $$\Lambda_{N,2}=-2^{-1}V_{N,Poi}^{-3/2}\sum_{i=1}^{N}y_{N,i}^3(1-\pi_{N,i})\pi_{N,i}^{-3}\left\{(I_{N,i}-\pi_{N,i})^2-\E(I_{N,i}-\pi_{N,i})^2\right\}.$$ It remains to show
	\begin{eqnarray}\label{lem_03_eq_01}
	\P_{Poi}\left(|\Lambda_{N,2}|\geq n_0^{-1/2}(\log n_0)^{-1}\right)=o(n_0^{-1/2})
	\end{eqnarray}
	and
	\begin{eqnarray}\label{lem_03_eq_02}
	\P_{Poi}\left(|\Delta_{N,1}\Lambda_{N,1}^2|\geq n_0^{-1/2}(\log n_0)^{-1}\right)=o(n_0^{-1/2}). 
	\end{eqnarray}
	For \eqref{lem_03_eq_01}, it is a directly consequence of $$\E\left(\Lambda_{N,2}\mid \mathcal{F}_N\right)=0$$ and
	$$ \E\left(\Lambda_{N,2}^2\mid \mathcal{F}_N\right)=O(n_0^{-2}).$$
	For \eqref{lem_03_eq_02}, as
	\begin{eqnarray}
	\notag & & \E(\Delta_{N,1}^4\mid \mathcal{F}_N) \\
	\notag & = & \sum_{i=1}^{N}\E(X_{N,i}^4\mid \mathcal{F}_N)+\sum_{(i,j)\in\Gamma_N}\E(X_{N,i}^2\mid \mathcal{F}_N)\E(X_{N,j}^2\mid \mathcal{F}_N) \\
	\notag & = & O(1),
	\end{eqnarray}	
	where $X_{N,i}=V_{N,Poi}^{-1/2}y_{N,i}\pi_{N,i}^{-1}(I_{N,i}-\pi_{N,i})$, we have that
	\begin{eqnarray}
	\notag & & \P_{Poi}\left(|\Delta_{N,1}\Lambda_{N,1}^2|\geq n_0^{-1/2}(\log n_0)^{-1}\right) \\
	\notag & \leq & \P_{Poi}\left(|\Delta_{N,1}|\geq n_0^{1/8}(\log n_0)\right) + \P_{Poi}\left(|\Lambda_{N,1}^2|\geq n_0^{-5/8}(\log n_0)^{-2}\right) \\
	\notag & \leq & n_0^{-1/2}(\log n_0)^{-4}\E(\Delta_{N,1}^4\mid \mathcal{F}_N)+n_0^{5/4}(\log n_0)^{4}\E(\Lambda_{N,1}^4\mid \mathcal{F}_N) \\
	\notag & =  & o(n_0^{-1/2}).
	\end{eqnarray}
	Thus, we finish the proof of this lemma.
\end{proof}

\begin{lemma}\label{lem_04}
	Assume condition (C\ref{cond: C4poi}) holds. Then for any positive integer $s$ satisfies $s\to\infty$ as $N\to\infty$ and $s=o(N)$, there exists a subset $\{y_{N,\ell_1},\ldots,y_{N,\ell_s}\}\subset \{y_{N,1},\ldots,y_{N,N}\}$ such that
	\begin{align}\label{lem_04_eq_01}
	\lim_{s\to\infty}s^{-1}\sum_{i=1}^{s}y_{N,\ell_i}^4<\infty.
	\end{align}
\end{lemma}

\begin{proof}
	We prove this lemma by contradiction.
	
	First, we split the population $\{y_{N,1},\ldots,y_{N,N}\}$ into $\lfloor N/s\rfloor$ subsets, with the first $\lfloor N/s\rfloor-1$ subsets as $\{y_{(j-1)s+1},\ldots,y_{N,js}\}_{j=1}^{\lfloor N/s\rfloor-1}$ and the last subset as $\{y_{(\lfloor N/s\rfloor-1)s+1},\ldots,y_{N,N}\}$. Here $\lfloor x\rfloor$ denotes the integer part of $x\in\mathbb{R}$. Assume that \eqref{lem_04_eq_01} is not satisfied by any subset  of $\{y_{N,1},\ldots,y_{N,N}\}$ with cardinality $s$. Then
	$$ \lim_{s\to\infty}s^{-1}\sum_{i=1}^{s}y_{N,(j-1)s+i}^4=\infty, $$
	for all $j=1,\ldots,\lfloor N/s\rfloor-1$. This implies
	\begin{eqnarray}
	& & N^{-1}\sum_{i=1}^{N}y_{N,i}^4 \notag \\
	& = & N^{-1}s\sum_{j=1}^{\lfloor N/s\rfloor-1}s^{-1}\sum_{i=1}^{s}y_{N,(j-1)s+i}^4+N^{-1}\sum_{i=(\lfloor N/s\rfloor-1)s+1}^{N}y_{N,i}^4 \notag \\
	& \to & \infty \notag
	\end{eqnarray}
	as $N\to\infty$. 
	
	This result is contradicted with condition (C\ref{cond: C4poi}). Thus, there exists at least one subset of $\{y_{N,1},\ldots,y_{N,N}\}$ with cardinality $s$ satisfies \eqref{lem_04_eq_01}.
\end{proof}

\begin{lemma}\label{lemma: 05}
	Denote $\hat{Y}_{N,Poi}=\sum_{i=1}^{N}y_{N,i}\pi_{N,i}^{-1}I_{N,i}$, $Y_N=\sum_{i=1}^{N}y_{N,i}$ and $\hat{V}_{N,Poi}=\sum_{i=1}^{N}y_{N,i}^2(1-\pi_{N,i})\pi_{N,i}^{-2}I_{N,i}$. Let $\hat{F}_{N,Poi}(z)=\P_{Poi}(T_{N,Poi}\leq z)$ be the cumulative distribution function (cdf) of $T_{N,Poi}$ under Poisson sampling, where $T_{N,Poi}=\hat{V}_{N,Poi}^{-1/2}\left(\hat{Y}_{N,Poi}-Y_N\right)$. Then, under conditions (C\ref{cond: C1poi})--(C\ref{cond: Pois 5}),
	\begin{eqnarray}
	& & \hat{F}_{N,Poi}(z) = \Phi(z)+\left(\frac{\mu_{N,Poi}^{(3)}}{6V_{N,Poi}^{3/2}}(1-z^2)+\frac{\tau_{N,Poi}^{(3)}}{2V_{N,Poi}^{3/2}}z^2\right)\phi(z) +o(n_0^{-1/2}) \notag
	\end{eqnarray}
	uniformly in $z\in\mathbb{R}$, where
	\begin{eqnarray}
	\notag \mu_{N,Poi}^{(3)}=\sum_{i=1}^Ny_{N,i}^3(1-\pi_{N,i})\{(1-\pi_{N,i})^2\pi_{N,i}^{-2}-1\}
	\end{eqnarray}
	and
	\begin{eqnarray}
	\notag \tau_{N,Poi}^{(3)}=\sum_{i=1}^{N}y_{N,i}^3(1-\pi_{N,i})^2\pi_{N,i}^{-2}.
	\end{eqnarray}
\end{lemma}

\begin{proof}
	According to Lemma \ref{lem_03}, $T_{N,Poi}=\Delta_{N,1}+\Delta_{N,2}+\Delta_{N,3}+\Delta_{N,4}$, where $\Delta_{N,1}$, $\Delta_{N,2}$, $\Delta_{N,3}$ are defined in Lemma \ref{lem_03} and $\Delta_{N,4}$ satisfies
	$$ \P_{Poi}\left(|\Delta_{N,4}|\geq n_0^{-1/2}(\log n_0)^{-1}\right)=o(n_0^{-1/2}). $$	
	Thus, it suffices to show that
	$$\sup_{z\in\mathbb{R}}\left|\P_{Poi}(\Delta_{N,1}+\Delta_{N,2}+\Delta_{N,3}\leq z)-F_{E,N}(z)\right|=o(n_0^{-1/2}), $$
	where 
	\begin{eqnarray}
	F_{E,N}(z) & = & \Phi(z)+\left(6^{-1}V_{N,Poi}^{-3/2}\mu_{N,Poi}^{(3)}(1-z^2)+2^{-1}V_{N,Poi}^{-3/2}\tau_{N,Poi}^{(3)}z^2\right)\phi(z). \notag
	\end{eqnarray}
	
	Define
	$$ F_{E,N1}(z)=\Phi(z)+6^{-1}\left(V_{N,Poi}^{-3/2}\mu_{N,Poi}^{(3)}-3V_{N,Poi}^{-3/2}\tau_{N,Poi}^{(3)}\right)(1-z^2)\phi(z). $$
	As $\Delta_{N,3}$ is nonrandom and $\Delta_{N,3}=O(n_0^{-1/2})$, by the fact that $$\sup_{z\in\mathbb{R}}|F_{E,N1}(z-\Delta_{N,3})-F_{E,N}(z)|=o(n_0^{-1/2}),$$ it is enough to prove that
	\begin{eqnarray}\label{theo_eq_01}
	\sup_{z\in\mathbb{R}}\left|\P_{Poi}(\Delta_{N,1}+\Delta_{N,2}\leq z)-F_{E,N1}(z)\right|=o(n_0^{-1/2}).
	\end{eqnarray}
	
	Denote $W_N=\Delta_{N,1}+\Delta_{N,2}$ and let $\phi_{W_N}(t)$ be the characteristic function (c.f.) of $W_N$, that is
	$$ \phi_{W_N}(t)=\E\{\exp(\iota tW_N)\mid \mathcal{F}_N\}=\E\left[\exp\{\iota t(\Delta_{N,1}+\Delta_{N,2})\}\mid \mathcal{F}_N\right]. $$
	In addition, denote
	\begin{eqnarray}
	\notag & & \phi_{E,N1}(t) \\
	\notag & = & \int \exp(\iota tz)dF_{E,N1}(z) \\
	\notag & = & \exp(-t^2/2)\left\{1+6^{-1}(\iota t)^3\left(V_{N,Poi}^{-3/2}\mu_{N,Poi}^{(3)}-3V_{N,Poi}^{-3/2}\tau_{N,Poi}^{(3)}\right)\right\}.
	\end{eqnarray}
	By Esseen's smoothing lemma \citep[Theorem 5.1]{Petrov1995}, for any arbitrary $\varepsilon\in(0,1)$,
	\begin{eqnarray}
	\notag & &  \sup_{z\in\mathbb{R}}\left|\P_{Poi}(W_N\leq z)-F_{E,N1}(z)\right| \\
	\notag & \leq & \frac{1}{\pi}\int_{|t|\leq a_{\varepsilon}n_0^{1/2}}\frac{1}{|t|}|\phi_{W_N}(t)-\phi_{E,N1}(t)|dt+n_0^{-1/2}{\varepsilon},
	\end{eqnarray}
	where $a_{\varepsilon}$ is chosen to satisfy $|dF_{E,N1}(z)/dx|\leq a_{\varepsilon}\varepsilon$. Thus, it suffices to prove
	\begin{eqnarray}
	\frac{1}{\pi}\int_{|t|\leq a_{\varepsilon}n_0^{1/2}}\frac{1}{|t|}|\phi_{W_N}(t)-\phi_{E,N1}(t)|dt = o(n_0^{-1/2}). \notag
	\end{eqnarray}
	
	Recall $\Theta_{N,Poi}^{(2,3)}= \sum_{(i,j)\in\Gamma_N}y_{N,i}^2y_{N,j}^3(1-\pi_{N,i})(1-\pi_{N,j})^2\pi_{N,i}^{-1}\pi_{N,j}^{-2}$, where $\Gamma_N=\{(i,j): i,j=1,\ldots,N \text{~and~} i\neq j\}$, then
	\begin{eqnarray}
	V_{N,Poi}^{-5/2}\Theta_{N,Poi}^{(2,3)}&=& V_{N,Poi}^{-3/2}\tau_{N,Poi}^{(3)}-V_{N,Poi}^{-5/2}\sum_{i=1}^{N}y_{N,i}^5(1-\pi_{N,i})^3\pi_{N,i}^{-3}, \notag
	\end{eqnarray}
	where $V_{N,Poi}^{-3/2}\tau_{N,Poi}^{(3)}=O(n_0^{-1/2})$ and $$V_{N,Poi}^{-5/2}\sum_{i=1}^{N}y_{N,i}^5(1-\pi_{N,i})^3\pi_{N,i}^{-3}=O(n_0^{-1/2}N^{-1}).$$
	
	Denote
	\begin{eqnarray}
	\notag & & \phi_{E,N2}(t) = \exp(-t^2/2)\left\{1+6^{-1}(\iota t)^3\left(V_{N,Poi}^{-3/2}\mu_{N,Poi}^{(3)}-3V_{N,Poi}^{-5/2}\Theta_{N,Poi}^{(2,3)}\right)\right\},
	\end{eqnarray}
	then
	\begin{eqnarray}
	\frac{1}{\pi}\int_{|t|\leq a_{\varepsilon}n_0^{1/2}}\frac{1}{|t|}|\phi_{E,N1}(t)-\phi_{E,N2}(t)|dt = o(n_0^{-1/2}). \notag
	\end{eqnarray}
	So it is sufficient to show that
	\begin{eqnarray}\label{theo_eq_02}
	\frac{1}{\pi}\int_{|t|\leq a_{\varepsilon}n_0^{1/2}}\frac{1}{|t|}|\phi_{W_N}(t)-\phi_{E,N2}(t)|dt = o(n_0^{-1/2}).
	\end{eqnarray}
	
	A simple calculation yields $\E(\Delta_{N,1}\mid \mathcal{F}_N)=\E(\Delta_{N,2}\mid \mathcal{F}_N)=0$, $\E(\Delta_{N,1}^2\mid \mathcal{F}_N)=1$ and
	\begin{eqnarray}
	\notag & & \E(\Delta_{N,2}^2\mid \mathcal{F}_N) \\
	\notag & = & 4^{-1}V_{N,Poi}^{-3}\sum_{(i,j)\in\Gamma_N}y_{N,i}^2y_{N,j}^4(1-\pi_{N,i})(1-\pi_{N,j})^3\pi_{N,i}^{-1}\pi_{N,j}^{-3}=O(n_0^{-1})
	\end{eqnarray}
	This implies that for $|t|\leq b_Nn_0^{1/2}$ where $b_N\to 0$ as $N\to\infty$, $|t\Delta_{N,2}|=o_p(1)$. By the inequality that $|\exp(\iota x)-1-\iota x|\leq|x|^2/2$ for any real number $x$, we write
	\begin{eqnarray}
	\notag \phi_{W_N}(t) & = & \E\{\exp(\iota tW_N)\mid \mathcal{F}_N\}=\E\left[\exp\{\iota t(\Delta_{N,1}+\Delta_{N,2})\}\mid \mathcal{F}_N\right] \\
	\notag & = & \E\left[\exp(\iota t\Delta_{N,1})\left\{1+\iota t\Delta_{N,2}+O_p(|t\Delta_{N,2}|^2)\right\}\mid \mathcal{F}_N\right] \\
	\notag & = & \E\{\exp(\iota t\Delta_{N,1})\mid \mathcal{F}_N\}+\iota t\E\{\Delta_{N,2}\exp(\iota t\Delta_{N,1})\mid \mathcal{F}_N\} \\
	\notag & & ~~~~~~~~~~~~~~~~~~~~~~~~~~~~+O(\E\{|t\Delta_{N,2}|^2\mid \mathcal{F}_N\}).
	\end{eqnarray}
	
	According to Lemma \ref{lem_01} and Lemma \ref{lem_02},
	\begin{eqnarray}
	\notag & & \left|\phi_{\Delta_{N,1}}(t)-\exp(-t^2/2)-6^{-1}(\iota t)^3V_{N,Poi}^{-3/2}\mu_{N,Poi}^{(3)}\exp(-t^2/2)\right| \\
	\notag & \leq & C_{2,1}\exp(-19t^2/48)\left(t^4n_0^{-1}+t^6n_0^{-1}\right)
	\end{eqnarray}
	and
	\begin{eqnarray}
	\notag & & \E\left\{\Delta_{N,2}\exp(\iota t\Delta_{N,1})\mid \mathcal{F}_N\right\} \\
	\notag & = & 2^{-1}t^2\exp(-t^2/2)V_{N,Poi}^{-5/2}\Theta_{N,Poi}^{(2,3)}+\varpi_{N,1}(t),
	\end{eqnarray}
	for all $|t|\leq\min\left(\big\{\max_{1\leq i\leq N}\E(X_{N,i}^2\mid \mathcal{F}_N)\big\}^{-1/2},V_{N,Poi}^{3/2}/\left(4\nu_{N,Poi}^{(3)}\right)\right)$, where $\varpi_{N,1}(t)$ satisfies
	\begin{eqnarray}
	\notag& & |\varpi_{N,1}(t)| \\
	\notag& \leq & C_{3,1}\exp\left\{-t^2/2+2|t|^3V_{N,Poi}^{-3/2}\nu_{N,Poi}^{(3)}/3+t^2\max_{1\leq i\leq N}\E(X_{N,i}^2\mid \mathcal{F}_N)\right\} \\
	\notag& & ~~~~~~~~~~~~~~~~~~~~~~~~~~~~~~~~~~~~~~~\times \left(|t|^3n_0^{-1}+t^4n_0^{-3/2}+|t|^5n_0^{-1}\right).
	\end{eqnarray}
	Recall that $\E(\Delta_{N,2}^2\mid \mathcal{F}_N)=O(n_0^{-1})$, it can be easily verified that
	\begin{align*}
	\phi_{W_N}(t) = \phi_{E,N2}(t)+\varpi_N(t)
	\end{align*}
	for $|t|\leq \min\left(b_Nn_0^{1/2}, \big\{\max_{1\leq i\leq N}\E(X_{N,i}^2\mid \mathcal{F}_N)\big\}^{-1/2},V_{N,Poi}^{3/2}/\left(4\nu_{N,Poi}^{(3)}\right)\right)$, where
	\begin{eqnarray}
	\notag|\varpi_N(t)| & \leq & C_{4,1}\bigg[t^2n_0^{-1}+\exp(-19t^2/48)\left(t^4n_0^{-1}+t^6n_0^{-1}\right)\\
	\notag& & + \exp\left\{-t^2/2+2|t|^3V_{N,Poi}^{-3/2}\nu_{N,Poi}^{(3)}/3+t^2\max_{1\leq i\leq N}\E(X_{N,i}^2\mid \mathcal{F}_N)\right\} \\
	\notag& & ~~~~~~~~~\times\left(t^4n_0^{-1}+|t|^5n_0^{-3/2}+t^6n_0^{-1}\right)\bigg]
	\end{eqnarray}
	and $C_{4,1}$ is a positive constant.
	
	Under the assumption that $\lim_{N\to\infty}N^{-1}\sum_{i=1}^{N}y_{N,i}^8=C_3$ for a positive constant $C_3$, we have $\max_{1\leq i\leq N}|y_{N,i}|=O(N^{1/8})$. Then, $\max_{1\leq i\leq N}\E(X_{N,i}^2\mid \mathcal{F}_N)=O(N^{-3/4})$. It follows that for $|t|\leq n_0^{1/4}(\log n_0)^{-1}$,
	\begin{align*}
	& \frac{1}{\pi}\int_{|t|\leq n_0^{1/4}(\log n_0)^{-1}}\frac{1}{|t|}|\phi_{W_N}(t)-\phi_{E,N2}(t)|dt \\
	\leq & \frac{1}{\pi}\int_{|t|\leq n_0^{1/4}(\log n_0)^{-1} }\frac{1}{|t|}|\varpi_N(t)|dt = o(n_0^{-1/2}).
	\end{align*}
	
	It is obvious that
	$$ \frac{1}{\pi}\int_{n_0^{1/4}(\log n_0)^{-1}\leq|t|\leq a_{\varepsilon}n_0^{1/2}}\frac{1}{|t|}|\phi_{E,N2}(t)|dt = o(n_0^{-1/2}), $$
	it remains to establish
	\begin{align}\label{theo_eq_03}
	\frac{1}{\pi}\int_{n_0^{1/4}(\log n_0)^{-1}\leq|t|\leq a_{\varepsilon}n_0^{1/2}}\frac{1}{|t|}|\phi_{W_N}(t)|dt = o(n_0^{-1/2}).
	\end{align}
	
	Denote
	\begin{eqnarray}
	\notag& & U_{N,i,j}=-2^{-1}V_{N,Poi}^{-3/2}\{y_{N,i}(1-\pi_{N,i})\pi_{N,i}^{-1}+y_{N,j}(1-\pi_{N,j})\pi_{N,j}^{-1}\} \\
	\notag& & ~~~~~~~~~~~~~~~~~~~~~~~~~~~~\times y_{N,i}y_{N,j}\pi_{N,i}^{-1}\pi_{N,j}^{-1}(I_{N,i}-\pi_{N,i})(I_{N,j}-\pi_{N,j}),
	\end{eqnarray}
	then $\Delta_{N,2}=\sum_{1\leq i<j\leq N}U_{N,i,j}$ and $W_{N}=\sum_{i=1}^{N}X_{N,i}+\sum_{1\leq i<j\leq N}U_{N,i,j}$.
	
	Take $m=\lfloor n_0^{-1/2}N/(\log n_0)\rfloor$. According to Lemma \ref{lem_04}, we assume that $m^{-1}\sum_{i=1}^{m}y_{N,i}^4=O(1)$ for sufficient large $N$ without loss of generality. Define
	$$ \Delta_{N,2}(m)=\sum_{i=1}^{m}\sum_{j=i+1}^{N}U_{N,i,j}. $$
	By simple algebra,
	\begin{eqnarray}
	\notag& & \E\left[\{\Delta_{N,2}(m)\}^2\mid \mathcal{F}_N\right] \\
	\notag& = & \E\left\{\left(\sum_{i=1}^{m}\sum_{j=i+1}^{N}U_{N,i,j}\right)^2\mid \mathcal{F}_N\right\} \\
	\notag& = & 4^{-1}V_{N,Poi}^{-3}\sum_{i=1}^{m}\sum_{j=i+1}^{N}\{y_{N,i}(1-\pi_{N,i})\pi_{N,i}^{-1}+y_{N,j}(1-\pi_{N,j})\pi_{N,j}^{-1}\}^2 \\
	\notag& & ~~~~~~~~~~~~~~~~~~~~~~~~~~~~~~~\times y_{N,i}^2y_{N,j}^2(1-\pi_{N,i})(1-\pi_{N,j})\pi_{N,i}^{-1}\pi_{N,j}^{-1} \\
	\notag& = & O\left(n_0^{-1}N^{-2}\sum_{i=1}^{m}\sum_{j=i+1}^{N}(y_{N,i}^4y_{N,j}^2+y_{N,i}^2y_{N,j}^4)\right)=O\left(n_0^{-1}N^{-1}m\right).
	\end{eqnarray}
	By the inequality that $|\exp(\iota x)-1-\iota x|\leq 2^{-1}|x|^2$ for all real $x$, we have
	\begin{eqnarray}
	\notag& & \big|\E\{\exp(\iota t W_{N})\mid \mathcal{F}_N\}-\E(\exp[\iota t\{W_{N}-\Delta_{N,2}(m)\}]\mid \mathcal{F}_N) \\
	\notag& & ~~~~~~~~~~~~~~~~~~~~~~-\iota t\E(\Delta_{N,2}(m)\exp[\iota t\{W_{N}-\Delta_{N,2}(m)\}]\mid \mathcal{F}_N)\big| \\
	\notag&\leq & \E\left(\left|\exp[\iota t\{W_{N}-\Delta_{N,2}(m)\}]\right|\left|\exp\{\iota t \Delta_{N,2}(m)\}-1-\iota t\Delta_{N,2}(m)\right|\mid \mathcal{F}_N\right) \\
	\notag&\leq & 2^{-1}\E\{|\iota t\Delta_{N,2}(m)|^2\mid \mathcal{F}_N\} \\
	\notag&\leq & C_{4,2}t^2n_0^{-1}N^{-1}m,
	\end{eqnarray}
	where $C_{4,2}$ is a positive constant. This clearly indicates that
	\begin{eqnarray}
	\notag\left|\E\{\exp(\iota t W_{N})\mid \mathcal{F}_N\}\right| & \leq & \left|\E(\exp[\iota t\{W_{N}-\Delta_{N,2}(m)\}])\mid \mathcal{F}_N\right| \\
	\notag& & +|t|\left|\E(\Delta_{N,2}(m)\exp[\iota t\{W_{N}-\Delta_{N,2}(m)\}]\mid \mathcal{F}_N)\right| \\
	\notag& & +C_{4,2}t^2n_0^{-1}N^{-1}m.
	\end{eqnarray}
	In view of the fact that $X_{N,1},\ldots,X_{N,m}$ are the only terms in $W_{N}-\Delta_{N,2}(m)$ that depend on $I_{N,1},\ldots,I_{N,m}$, for a positive constant $C_{4,3}$,
	\begin{align*}
	& \left|\E(\exp[\iota t\{W_{N}-\Delta_{N,2}(m)\}])\mid \mathcal{F}_N\right| \\
	= & \left|\E\left(\exp\left[\iota t\left\{\sum_{i=1}^{m}X_{N,i}+\sum_{i=m+1}^{N}X_{N,i}+\Delta_{N,2}-\Delta_{N,2}(m)\right\}\right]\mid \mathcal{F}_N\right)\right| \\
	\leq & \left|\prod_{i=1}^{m}\E\{\exp(\iota tX_{N,i})\mid \mathcal{F}_N\}\right|\left|\E\left(\exp\left[\iota t\left\{\sum_{i=m+1}^{N}X_{N,i}+\Delta_{N,2}-\Delta_{N,2}(m)\right\}\right]\mid \mathcal{F}_N\right)\right| \\
	\leq & \left|\prod_{i=1}^{m}\E\{\exp(\iota tX_{N,i})\mid \mathcal{F}_N\}\right| \leq C_{4,3}m^{-a}
	\end{align*}  by condition (C\ref{cond: extra poi}).
	In addition, there exists a constant $C_{4,4}>0$ such that
	\begin{eqnarray}
	\notag& & \left|\E(\Delta_{N,2}(m)\exp[\iota t\{W_{N}-\Delta_{N,2}(m)\}]\mid \mathcal{F}_N)\right| \\
	\notag& = & \left|\E\left(\sum_{i=1}^{m}\sum_{j=i+1}^{N}U_{N,i,j}\exp\left[\iota t\left\{\sum_{\ell=1}^{N}X_{N,\ell}+\Delta_{N,2}-\Delta_{N,2}(m)\right\}\right]\mid \mathcal{F}_N\right)\right| \\
	\notag& \leq & \sum_{i=1}^{m}\sum_{j=i+1}^{N}\left|\prod_{\substack{\ell=1,\ldots,m \\ \ell\neq i,j}}\E\{\exp(\iota tX_{N,\ell})\mid \mathcal{F}_N\}\right| \\
	\notag& & \times\bigg|\E\left(U_{N,i,j}\exp\left[\iota t\left\{X_{N,i}+X_{N,j}+\Delta_{N,2}-\Delta_{N,2}(m)\right\}\right]\mid \mathcal{F}_N\right)\bigg| \\
	\notag& \leq & C_{4,4}(m-2)^{-a}n_0^{1/2}N^{-1}m.
	\end{eqnarray}
	
	Finally, for a positive constant $C_{4,5}$,
	\begin{align*}
	& \frac{1}{\pi}\int_{n_0^{1/4}(\log n_0)^{-1}\leq|t|\leq a_{\varepsilon}n_0^{1/2}}\frac{1}{|t|}|\phi_{W_N}(t)|dt \\ \leq &  \frac{C_{4,5}}{\pi}\int_{n_0^{1/4}(\log n_0)^{-1}\leq|t|\leq a_{\varepsilon}n_0^{1/2}}\frac{1}{|t|}(m^{-a}+|t|m^{1-a}n_0^{1/2}N^{-1}+t^2n_0^{-1}N^{-1}m)dt \\
	= & o(n_0^{-1/2}),
	\end{align*}
	as $a>2$. Therefore, we finish the proof of this lemma.
\end{proof}

\begin{proof}[Proof of Theorem \ref{theorem: edgeworth expansion Poisson}]	
	According to Lemma \ref{lemma: almost sure variance poi}, $\mu_{N,Poi}^{(3)} = O(n_0^{-2}N^{3})$. This cooperates with (C\ref{con: variance}) indicates that
	\begin{equation}
	\frac{\mu_{N,Poi}^{(3)}}{{V}_{N,Poi}^{3/2}} = O(n_0^{-1/2}). \label{eq: theorem e 03}
	\end{equation}
	In addition, by \eqref{eq: lemma1 2} and \eqref{eq: lemma1 1.2} of Lemma \ref{lemma: almost sure variance poi}, we can prove that $$\hat{V}_{N,Poi}^{-3/2}\hat{\mu}_{N,Poi}^{(3)} =O_p(n_0^{-1/2}).$$
	Similarly, we can show that $\hat{V}_{N,Poi}^{-3/2}\hat{\tau}_{N,Poi}^{(3)} =O_p(n_0^{-1/2})$ according to Lemma \ref{lemma: almost sure variance poi}. Finally, by (C\ref{con: variance}) and Lemma \ref{lemma: almost sure variance poi},
	$$ \hat{V}_{N,Poi}^{-3/2}\hat{\mu}_{N,Poi}^{(3)}-V_{N,Poi}^{-3/2}\mu_{N,Poi}^{(3)}=o_p(n_0^{-1/2}) $$
	and
	$$ \hat{V}_{N,Poi}^{-3/2}\hat{\tau}_{N,Poi}^{(3)}-V_{N,Poi}^{-3/2}\tau_{N,Poi}^{(3)}=o_p(n_0^{-1/2}). $$
	Combine these results with Lemma \ref{lemma: 05} , we have proved Theorem \ref{theorem: edgeworth expansion Poisson}.

\end{proof}

The proof of Theorem \ref{theorem: e 2} uses the following lemma.

\begin{lemma} \label{lemma_06}
	Let $(N_1^*,\ldots,N_n^*)$ be a multinomial random vector with distribution $\mathrm{MN}(N;\rho)$, where $\rho=(\rho_1,\ldots,\rho_n)$ and $\rho_i=\pi_{N,i}^{-1}/\left(\sum_{j=1}^n\pi_{N,j}^{-1}\right)$ for $i=1,\ldots,n$. Denote $\mathcal{F}_N^*=\{y_{N,1}^*, \ldots, y_{N,N}^*\}$ to be the bootstrap finite population generated from the realized sample $\{y_{N,1}, \ldots, y_{N,n}\}$ and the random vector $(N_1^*,\ldots,N_n^*)$ with each $N_i^*$ indicates the number of replicates of $y_{N,i}$ in $\mathcal{F}_N^*$. Define $n_1=\sum_{i=1}^{n}\mathbb{I}(N_i^*\geq 1)$ as the number of distinct $y_{N,i}$, $i=1,\ldots, n$ in $\mathcal{F}_N^*$. Then, as $N\to\infty$, 
	\begin{eqnarray}
	\mathbb{P}_{*}\left(n_1\geq \min\{n_0, m\}\right)\to 1,
	\end{eqnarray}
	where $m=\lfloor n_0^{-1/2}N/(\log n_0) \rfloor$ is the integer part of $n_0^{-1/2}N/(\log n_0)$ and $\mathbb{P}_{*}$ is the probability measure for the first step of the proposed bootstrap method conditional on the realized sample $\{y_{N,1},\ldots,y_{N,n}\}$.
\end{lemma}

\begin{proof}[Proof of Lemma \ref{lemma_06}]
	We prove this lemma under two different case scenarios: $m\geq n_0$ and $m<n_0$.
	
	First, consider the case of $m\geq n_0$, so we have that $n_0=o(N^{2/3})$. By the strong law of large numbers, $nn_0^{-1}=\sum_{i=1}^NI_{N,i}(\sum_{i=1}^{N}\pi_{N,i})^{-1} \to 1$ with probability $1$. It suffices to show that
	\begin{eqnarray}\label{eq:lemma_06_00}
	\mathbb{P}_{*}\left(n_1\neq n\right)\to 0.
	\end{eqnarray}	
	Mentioned that
	\begin{eqnarray}
	\mathbb{P}_{*}\left(n_1\neq n\right) & = & \mathbb{P}_{*}\left(\sum_{i=1}^{n}\mathbb{I}(N_i^*\geq 1)< n\right) \notag \\
	& = & \mathbb{P}_{*}\left(\bigcup_{i=1}^n\left\{N_i^*=0\right\}\right) \notag \\
	& \leq & \sum_{i=1}^{n}\mathbb{P}_{*}\left(N_i^*=0\right) \notag \\
	& \leq & \sum_{i=1}^{n}(1-\rho_i)^N \notag \\
	& \leq & n(1-C_1C_2^{-1}n^{-1})^N \notag \\
	& \to & 0, \notag
	\end{eqnarray}
	where the last inequality is due to that $C_1C_2^{-1}n^{-1}\leq \rho_i\leq C_1^{-1}C_2n^{-1}$ for $i=1,\ldots,n$.
	
	Next, consider the case of $m<n_0$ and it is sufficient to prove that
	\begin{eqnarray}\label{eq:lemma_06_01}
	\mathbb{P}_{*}\left(n_1\geq m\right)\to 1.
	\end{eqnarray}
	When $\alpha<1$, $n_0=o(N)$ and $m=o(N^{2/3})$. Thus,
	\begin{eqnarray}
	\mathbb{P}_{*}\left(n_1<m\right) & = & \mathbb{P}_{*}\left(\sum_{i=1}^{n}\mathbb{I}(N_i^*\geq 1)< m\right) \notag \\
	& \leq & \mathbb{P}_{*}\left(\sum_{i=1}^{m}\mathbb{I}(N_i^*\geq 1)< m\right) \notag \\
	& = & \mathbb{P}_{*}\left(\bigcup_{i=1}^m\left\{N_i^*=0\right\}\right) \notag \\
	& \leq & \sum_{i=1}^{m}\mathbb{P}_{*}\left(N_i^*=0\right) \notag \\
	& \leq & \sum_{i=1}^{m}(1-\rho_i)^N \notag \\
	& \leq & m(1-C_1C_2^{-1}n^{-1})^N \notag \\
	& \to & 0, \notag
	\end{eqnarray}
	which implies \eqref{eq:lemma_06_01} immediately.
	
	When $\alpha=1$, $n_0=O(N)$ and $m=o(N^{1/2})$. Using Stirling's formula, there exists a positive constant $C_{5,1}$ such that
	\begin{eqnarray}
	\notag & & \mathbb{P}_{*}(n_1<m) \\
	\notag & \leq & \sum_{i=n-m}^{n}\mathbb{P}_{*}(n_1=n-i) \\
	\notag & \leq & \sum_{i=n-m}^{n}\frac{n!}{(n-i)!i!}\left\{(n-i)C_1^{-1}C_2n^{-1}\right\}^{N} \\
	\notag & \leq & C_{5,1}\sum_{i=n-m}^{n}\left\{\frac{n}{2\pi i(n-i)}\right\}^{1/2}\left(\frac{n-i}{i}\right)^i\left(\frac{n}{n-i}\right)^n\left\{\frac{C_2(n-i)}{C_1n}\right\}^n,
	\end{eqnarray}
	where the second inequality uses the fact that $C_1C_2^{-1}n^{-1}\leq \rho_i\leq C_1^{-1}C_2n^{-1}$.
	Mentioned that $m=o(N^{1/2})$ and $n=O(N)$, we have
	\begin{eqnarray}
	\notag & & \left\{\frac{n}{2\pi i(n-i)}\right\}^{1/2} = O(1), \\
	\notag & & \left(\frac{n-i}{i}\right)^i \leq \left(\frac{m}{n-m}\right)^i=o(N^{-i/2}) \\
	\notag & & \left(\frac{n}{n-i}\right)^n\left\{\frac{C_2(n-i)}{C_1n}\right\}^n = \left(\frac{C_2}{C_1}\right)^n.
	\end{eqnarray}
	for $i=n-m,\ldots,n$. Thus,
	\begin{eqnarray}
	\notag & & \left\{\frac{n}{2\pi i(n-i)}\right\}^{1/2}\left(\frac{n-i}{i}\right)^i\left(\frac{n}{n-i}\right)^n\left\{\frac{C_2(n-i)}{C_1n}\right\}^n \\
	\notag & = & o\left(N^{-i/2}\left(C_1^{-1}C_2\right)^n\right).
	\end{eqnarray}
	
	Finally,
	\begin{eqnarray}
	\notag & & \mathbb{P}_{*}(n_1<m) \\
	\notag & = & o\left(\sum_{i=n-m}^{n}N^{-i/2}\left(C_1^{-1}C_2\right)^n\right) \\
	\notag & = & o\left(mN^{-(n-m)/2}\left(C_1^{-1}C_2\right)^n\right) \\
	\notag & = &o(\exp\left\{-(n-m)\log N/2+n\log(C_1^{-1}C_2)+\log m\right\}) \\
	\notag & \to & 0,
	\end{eqnarray}
	which finalize the proof of this lemma.
\end{proof}

\begin{proof}[Proof of Theorem \ref{theorem: e 2}]
	
	We first show 
	\begin{equation}
	N^{-1}\sum_{i=1}^N(\pi_{N,i}^{-1}I_{N,i} - 1) \to 0 \label{eq:proof_theo_3.2_01}
	\end{equation}	
	almost surely $(\mathbb{P}_{Poi})$.	
	Denote $D_N^{(2)}$ to be the event $\big\{N^{-1}\lvert\sum_{i=1}^N(\pi_{N,i}^{-1}I_{N,i} - 1)\lvert >\epsilon\big\}$, where $\epsilon$ is a fixed positive number. Similar to the argument used in proving Lemma \ref{lemma: almost sure variance poi}, we have
	\begin{eqnarray}
	& & \mathbb{P}_{Poi}\left(D_N^{(2)}\right) \notag \\
	&\leq&  \epsilon^{-4}N^{-4}E\left\{\left\lvert\sum_{i=1}^N(\pi_{N,i}^{-1}I_{N,i} - 1)\right\rvert^4\mid \mathcal{F}_N\right\}\notag \\ 
	&=& \epsilon^{-4}N^{-4}\left[\sum_{i=1}^N (1-\pi_{N,i})\{(1-\pi_{N,i})^3\pi_{N,i}^{-3}+1\}\right.\notag\\
	&&~~~~~~~~~~~~~~~~~~~~~~~~~~\left. +\sum_{(i,j)\in\Gamma_N}(1-\pi_{N,i})(1-\pi_{N,j})\pi_{N,i}^{-1}\pi_{N,j}^{-1}\right]\notag \\
	&\leq& C_{6,1}\epsilon^{-4} (n_0^{-3} + n_0^{-2}),\notag
	\end{eqnarray}
	where $C_{6,1}$ is a positive  constant with respect to $N$ and $\Gamma_N=\{(i,j): i,j=1,\ldots,N \text{~and~} i\neq j\}$. This immediately implies
	$$
	\sum_{N=1}^\infty\mathbb{P}_{Poi}(D_N^{(2)})<\infty,
	$$
	for any arbitrary positive $\epsilon$ and we have proved \eqref{eq:proof_theo_3.2_01} by the Borel-Cantelli Lemma.
	
	Next, for any $0\leq\delta\leq8$,
	\begin{eqnarray}
	E\left(n_0^{-1}\sum_{i=1}^N|y_{N,i}|^{\delta}I_{N,i}\mid \mathcal{F}_N\right) &=& 
	n_0^{-1}\sum_{i=1}^N|y_{N,i}|^{\delta}\pi_{N,i}\notag \\ 
	&\leq& C_2N^{-1}\sum_{i=1}^N|y_{N,i}|^{\delta} \notag \\
	&<&\infty, \label{eq:proof_theo_3.2_02}
	\end{eqnarray}	
	where the first inequality holds by (C\ref{cond: C1poi}) and the last inequality holds by (C\ref{cond: C4poi}). Thus, by \eqref{eq:proof_theo_3.2_02} and the Markov's inequality, we have 
	\begin{equation}
	n_0^{-1}\sum_{i=1}^Ny_{N,i}^{\delta}I_{N,i} = O_p(1) \label{eq:proof_theo_3.2_03}
	\end{equation}
	for $0\leq\delta\leq8$. In addition, as 
	\begin{eqnarray}
	E\left(n_0^{-1}\sum_{i=1}^NI_{N,i}\mid\mathcal{F}_N\right)&=&1,\notag\\
	\mathrm{var}\left(n_0^{-1}\sum_{i=1}^NI_{N,i}\mid\mathcal{F}_N\right)&\leq&  n_0^{-2}\sum_{i=1}^N\pi_{N,i}\notag\\
	&\leq& C_2n_0^{-1}\label{eq:proof_theo_3.2_03.2}, \notag
	\end{eqnarray}
	we have
	\begin{equation}
	n_0^{-1}\sum_{i=1}^NI_{N,i} = 1+o_p(1).\label{eq:proof_theo_3.2_04}
	\end{equation}
	
	In the first step of our proposed bootstrap method, $y_{N,1}^*,\ldots,y_{N,N}^*$ are independently and identically distributed (i.i.d.) with $\mathbb{P}_*(y_{N,i}^*=y_{N,j})=\rho_{N,j}=\pi_{N,j}^{-1}\big(\sum_{\ell=1}^n\pi_{N,\ell}^{-1}\big)^{-1}$. Mentioned that, for a positive constant $C_{6,2}$,
	\begin{eqnarray}
	\notag & & E_*\left\{N^{-1}\sum_{i=1}^N(y_{N,i}^*)^8\right\} \\
	&=& \sum_{i=1}^n\pi_{N,i}^{-1}\left(\sum_{j=1}^n\pi_{N,j}^{-1}\right)^{-1}y_{N,i}^8 \notag \\ 
	&=& \left(N^{-1}\sum_{i=1}^N\pi_{N,i}^{-1}I_{N,i}\right)^{-1}\sum_{i=1}^{N}(N^{-1}\pi_{N,i}^{-1})y_{N,i}^8I_{N,i} \notag\\
	&\leq& C_{6,2}n_0^{-1}\sum_{i=1}^Ny_{N,i}^8I_{N,i} \notag \\ &=&O_p(1), \label{eq:proof_theo_3.2_05}
	\end{eqnarray}
	where the first equality holds by the property of the proposed bootstrap method, the inequality holds by \eqref{eq:proof_theo_3.2_01} and (C\ref{cond: C1poi}),  and the last equality holds by \eqref{eq:proof_theo_3.2_03}. Thus, by \eqref{eq:proof_theo_3.2_05} and Markov's inequality, we have 
	\begin{equation}
	N^{-1}\sum_{i=1}^N(y_{N,i}^*)^8 = O_p(1).\label{eq:proof_theo_3.2_06}
	\end{equation}
	Similarly, we can prove that for any subset of the bootstrap finite population, say, $\{y_{N,\ell_1}^*,\ldots,y_{N,\ell_{m_0}}^*\}\subset\{y_{N,1}^*,\ldots,y_{N,N}^*\}$ and all $0\leq\delta\leq8$,
	\begin{equation}
	m_0^{-1}\sum_{i=1}^{m_0}(y_{N,\ell_i}^*)^{\delta} = O_p(1),\label{eq:proof_theo_3.2_06_1}
	\end{equation}
	here $m_0$ can be any positive integer less than $N$.
	
	Denote ${V}_{N,Poi}^* = \sum_{i=1}^{N}E_{**}(\hat{Y}_{N,Poi}^{*}-Y_{N}^{*})^2=\sum_{i=1}^N(y_{N,i}^*)^2(1-\pi_{N,i}^*)(\pi_{N,i}^*)^{-1}$, then by Lemma \ref{lemma: almost sure variance poi}, \eqref{eq:proof_theo_3.2_01} and Condition (C\ref{con: variance}),
	\begin{eqnarray}
	\notag & & \E_{*}\left(n_0N^{-2}{V}_{N,Poi}^*\right) \\
	\notag & = & \E_{*}\left\{n_0N^{-2}\sum_{i=1}^{n}N_i^*y_{N,i}^2(1-\pi_{N,i})\pi_{N,i}^{-1}\right\} \\
	\notag & = & n_0N^{-2}\sum_{i=1}^{n}N\pi_{N,i}^{-1}\left(\sum_{j=1}^{n}\pi_{N,j}^{-1}\right)^{-1}y_{N,i}^2(1-\pi_{N,i})\pi_{N,i}^{-1} \\
	\notag & = & \left(N^{-1}\sum_{i=1}^{N}\pi_{N,i}^{-1}I_{N,i}\right)^{-1}\left(n_0N^{-2}\sum_{i=1}^{n}y_{N,i}^2(1-\pi_{N,i})\pi_{N,i}^{-2}\right) \\
	\notag & = & \left(N^{-1}\sum_{i=1}^{N}\pi_{N,i}^{-1}I_{N,i}\right)^{-1}\left(n_0N^{-2}\hat{V}_{N,Poi}\right) \\
	\notag & \to & \sigma_1^2
	\end{eqnarray}
	in probability. In addition, for a positive constant $C_{6,3}$,
	\begin{eqnarray}
	\notag & & \mathrm{var}_{*}\left(n_0N^{-2}{V}_{N,Poi}^*\right) \\
	\notag & = & n_0^2N^{-4}\sum_{i=1}^{n}\mathrm{var}_{*}(N_i^*)y_{N,i}^4(1-\pi_{N,i})^2\pi_{N,i}^{-2} \\
	\notag & & ~~~~~~~~+n_0^2N^{-4}\sum_{(i,j)\in\Gamma_n}\mathrm{cov}_{*}(N_i^*,N_j^*)y_{N,i}^2y_{N,j}^2(1-\pi_{N,i})(1-\pi_{N,j})\pi_{N,i}^{-1}\pi_{N,j}^{-1} \\
	\notag & = & n_0^2N^{-4}\sum_{i=1}^{n}N\rho_{N,i}\left(1-\rho_{N,i}\right)y_{N,i}^4(1-\pi_{N,i})^2\pi_{N,i}^{-2} \\
	\notag & & ~~~~~~~~-n_0^2N^{-4}\sum_{(i,j)\in\Gamma_n}N\rho_{N,i}\rho_{N,j}y_{N,i}^2y_{N,j}^2(1-\pi_{N,i})(1-\pi_{N,j})\pi_{N,i}^{-1}\pi_{N,j}^{-1} \\
	\notag & \leq & C_{6,3}\left\{N^{-1}n_0^{-1}\sum_{i=1}^{N}y_{N,i}^4I_{N,i}+N^{-1}\left(n_0^{-1}\sum_{i=1}^{N}y_{N,i}^2I_{N,i}\right)\left(n_0^{-1}\sum_{i=1}^{N}y_{N,i}^2I_{N,i}\right)\right\} \\
	\notag & = & o_p(N^{-1}),
	\end{eqnarray}
	where $\Gamma_n=\{(i,j): i,j=1,\ldots,n \text{~and~} i\neq j\}$. Thus, we have
	\begin{eqnarray}
	n_0N^{-2}{V}_{N,Poi}^* = \sigma_1^2+o_p(1). \label{eq:proof_theo_3.2_07}
	\end{eqnarray}
	
	Denote $T_{N,Poi}^{*}=\big(\hat{V}_{N,Poi}^{*}\big)^{-1/2}\big(\hat{Y}_{N,Poi}^*-Y_N^*\big)$, where $Y_{N}^*=\sum_{i=1}^{N}y_{N,i}^*$,
	\begin{eqnarray}
	\notag & & \hat{Y}_{N,Poi}^* = \sum_{i=1}^{n}m_i^*y_{N,i}\pi_{N,i}^{-1} \\
	\notag & = & \sum_{i=1}^{N}y_{N,i}^*(\pi_{N,i}^*)^{-1}I_{N,i}^*
	\end{eqnarray}
	and
	\begin{eqnarray}
	\notag & & \hat{V}_{N,Poi}^{*} = \sum_{i=1}^{n}m_i^*y_{N,i}^2(1-\pi_{N,i})\pi_{N,i}^{-2} \\
	\notag & = & \sum_{i=1}^{N}(y_{N,i}^*)^2(1-\pi_{N,i}^*)(\pi_{N,i}^*)^{-2}I_{N,i}^*,
	\end{eqnarray}
	here $I_{N,i}^*$ is the bootstrap counterpart of $I_{N,i}$ and $I_{N,i}^*\sim \mathrm{Ber}(\pi_{N,i}^*)$ conditional on the bootstrap finite population $\mathcal{F}_N^*$.
	
	In view of \eqref{eq:proof_theo_3.2_06} and \eqref{eq:proof_theo_3.2_07}, by similar arguments used in the proof of Lemma \ref{lem_03}, we can show that
	\begin{eqnarray}
	\notag T_{N,Poi}^{*}=  \Delta_{N,1}^*+\Delta_{N,2}^*+\Delta_{N,3}^*+\Delta_{N,4}^*,
	\end{eqnarray}
	where 
	\begin{eqnarray}
	\notag \Delta_{N,1}^* &= & (V_{N,Poi}^*)^{-1/2}\sum_{i=1}^{N}y_{N,i}^*(\pi_{N,i}^*)^{-1}(I_{N,i}^*-\pi_{N,i}^*), \\
	\notag \Delta_{N,2}^* &= & -2^{-1}(V_{N,Poi}^*)^{-3/2}\sum_{(i,j)\in\Gamma_N}y_{N,i}^*(y_{N,j}^*)^2(1-\pi_{N,j}^*)(\pi_{N,i}^*)^{-1}(\pi_{N,j}^*)^{-2} \\
	\notag & & ~~~~~~~~~~~~~~~~~~~~~~~~~~~~~~~~~~~~~~~~~~~~~~~\times(I_{N,i}^*-\pi_{N,i}^*)(I_{N,j}^*-\pi_{N,j}^*), \\
	\notag \Delta_{N,3}^* & = & -2^{-1}(V_{N,Poi}^*)^{-3/2}\sum_{i=1}^{N}(y_{N,i}^*)^3(1-\pi_{N,i}^*)^2(\pi_{N,i}^*)^{-2},
	\end{eqnarray}
	and $\Delta_{N,4}^*$ satisfies
	$$ \P_{Poi}^*\left(|\Delta_{N,4}^*|\geq n_0^{-1/2}(\log n_0)^{-1}\right)=o_p(n_0^{-1/2}), $$
	where $\mathbb{P}_{Poi}^*$ is the counterpart of $\mathbb{P}_{Poi}$ conditional on the bootstrap finite population $\{y_{N,1}^*,\ldots,y_{N,N}^*\}$.
	
	Let $\hat{F}_{N,Poi}^*(z)=\P_{Poi}^*(T_{N,Poi}^*\leq z)$ be the cumulative distribution function of $T_{N,Poi}^*$ conditional on the bootstrap finite population $\mathcal{F}_N^*$, we proceed to prove
	\begin{eqnarray}\label{eq:proof_theo_3.2_08}
	& & \hat{F}_{N,Poi}^*(z) \\
	& = & \Phi(z)+\left\{\frac{\mu_{N,Poi}^{(3)*}}{6(V_{N,Poi}^*)^{3/2}}(1-z^2)+\frac{\tau_{N,Poi}^{(3)*}}{2(V_{N,Poi}^*)^{3/2}}z^2\right\}\phi(z) +o_p(n_0^{-1/2}), \notag
	\end{eqnarray}
	uniformly in $z\in\mathbb{R}$, where $$\mu_{N,Poi}^{(3)*} = \sum_{i=1}^{N}E_{**}(\hat{Y}_{N,Poi}^{*}-Y_{N}^{*})^3 = \sum_{i=1}^N(y_{N,i}^{*})^3(1-\pi_{N,i}^{*})\{(1-\pi_{N,i}^{*})^2(\pi_{N,i}^{*})^{-2}-1\}$$ and $$\tau_{N,Poi}^{(3)*} = \sum_{i=1}^N(y_{N,i}^{*})^3(1-\pi_{N,i}^{*})(\pi_{N,i}^{*})^{-2}.$$
	
	Denote $W_N^*=\Delta_{N,1}^*+\Delta_{N,2}^*$ and let $\phi_{W_N^*}(t)$ be the characteristic function (c.f.) of $W_N^*$ conditional on the bootstrap finite population $\mathcal{F}_N^*$, that is,
	$$ \phi_{W_N^*}(t)=\E_{**}\{\exp(\iota tW_N^*)\}=\E_{**}\left[\exp\{\iota t(\Delta_{N,1}^*+\Delta_{N,2}^*)\}\right].$$ 
	
	Regarding the proof of Lemma \ref{lemma: 05}, it is enough to show that
	\begin{align}\label{eq:proof_theo_3.2_09}
	\frac{1}{\pi}\int_{n_0^{1/4}(\log n_0)^{-1}\leq|t|\leq a_{\varepsilon}n_0^{1/2}}\frac{1}{|t|}|\phi_{W_N^*}(t)|dt = o_p(n_0^{-1/2}),
	\end{align}
	for any arbitrary $\varepsilon\in(0,1)$ in order to finalize \eqref{eq:proof_theo_3.2_08}.
	Here $a_{\varepsilon}$ is chosen to satisfy $|dF_{E,N1}(z)/dx|\leq a_{\varepsilon}\varepsilon$ with $F_{E,N1}(z)$ defined as	
	$$ F_{E,N1}(z)=\Phi(z)+6^{-1}\left(V_{N,Poi}^{-3/2}\mu_{N,Poi}^{(3)}-3V_{N,Poi}^{-3/2}\tau_{N,Poi}^{(3)}\right)(1-z^2)\phi(z). $$
	
	Denote $X_{N,i}^*=(V_{N,Poi}^*)^{-1/2}y_{N,i}^*(\pi_{N,i}^*)^{-1}(I_{N,i}^*-\pi_{N,i}^*)$ and
	\begin{eqnarray}
	\notag& & U_{N,i,j}^*=-2^{-1}V_{N,Poi}^{-3/2}\{y_{N,i}^*(1-\pi_{N,i}^*)(\pi_{N,i}^*)^{-1}+y_{N,j}^*(1-\pi_{N,j}^*)(\pi_{N,j}^*)^{-1}\} \\
	\notag& & ~~~~~~~~~~~~~~~~~~~~~~~\times y_{N,i}^*y_{N,j}^*(\pi_{N,i}^*)^{-1}(\pi_{N,j}^*)^{-1}(I_{N,i}^*-\pi_{N,i}^*)(I_{N,j}^*-\pi_{N,j}^*),
	\end{eqnarray}
	then $\Delta_{N,1}^*=\sum_{i=1}^{N}X_{N,i}^*$, $\Delta_{N,2}^*=\sum_{1\leq i<j\leq N}U_{N,i,j}^*$ and $$W_{N}^*=\sum_{i=1}^{N}X_{N,i}^*+\sum_{1\leq i<j\leq N}U_{N,i,j}^*.$$
	
	Take $m=\lfloor n_0^{-1/2}N/(\log n_0)\rfloor$, we prove \eqref{eq:proof_theo_3.2_09} under two different case scenarios: $m\geq n_0$ and $m<n_0$.
	
	First, consider the case of $m\geq n_0$, we have that $n_0=o(N^{2/3})$. According to Lemma \ref{lemma_06}, we assume $y_{N,1}^*=y_{N,1}, \ldots, y_{N,n}^*=y_{N,n}$, without loss of generality. Then, by \eqref{eq:proof_theo_3.2_07}, condition (C\ref{cond: extra poi 02}) and the fact that
	\begin{eqnarray}
	\notag & & \E\left\{\exp\left(\iota tX_{N,i}\right)\mid \mathcal{F}_N\right\} \\
	\notag & = & \pi_{N,i}\exp\big\{\iota tV_{N,Poi}^{-1/2}y_{N,i}(1-\pi_{N,i})\pi_{N,i}^{-1}\big\}+(1-\pi_{N,i})\exp\big(-\iota tV_{N,Poi}^{-1/2}y_{N,i}\big),
	\end{eqnarray}
	we arrive at
	\begin{eqnarray}
	\notag &  & \left|\prod_{i=1}^{m}\E_{**}\{\exp(\iota tX_{N,i}^*)\}\right| \\
	\notag & \leq & \left|\prod_{i=1}^{n}\E_{**}\{\exp(\iota tX_{N,i}^*)\}\right| \\
	\notag & = & \bigg|\prod_{i=1}^{n}\bigg[\pi_{N,i}^*\exp\left\{\iota t\left(V_{N,Poi}^*\right)^{-1/2}y_{N,i}^*(1-\pi_{N,i}^*)\left(\pi_{N,i}^*\right)^{-1}\right\} \\
	\notag & & ~~~~~~~~~~~~~~~~~~+(1-\pi_{N,i}^*)\exp\left\{-\iota t\left(V_{N,Poi}^*\right)^{-1/2}y_{N,i}^*\right\}\bigg]\bigg| \\
	\notag & = & \bigg|\prod_{i=1}^{n}\bigg[\pi_{N,i}\exp\left\{\iota t\left(V_{N,Poi}^*\right)^{-1/2}y_{N,i}(1-\pi_{N,i})\pi_{N,i}^{-1}\right\} \\
	\notag & & ~~~~~~~~~~~~~~~~~~+(1-\pi_{N,i})\exp\left\{-\iota t\left(V_{N,Poi}^*\right)^{-1/2}y_{N,i}\right\}\bigg]\bigg| \\
	& = & O_p(n_0^{-a}). \label{eq:proof_theo_3.2_11}
	\end{eqnarray} 
	By \eqref{eq:proof_theo_3.2_06_1}, we have that $m^{-1}\sum_{i=1}^{m}(y_{N,i}^*)^4=O_p(1)$. Similar to the proof of Lemma \ref{lem_03}, define $$\Delta_{N,2}^*(m)=\sum_{i=1}^{m}\sum_{j=i+1}^{N}U_{N,i,j}^*,$$ then $\E_{**}\{\Delta_{N,2}^*(m)\}^2=O_p\left(n_0^{-1}N^{-1}m\right)$. Furthermore, for positive constants $C_{6,4}$ and $C_{6,5}$,
	\begin{eqnarray}
	\notag & & |\phi_{W_N^*}(t)| \\
	\notag & \leq & \left|\E_{**}(\exp[\iota t\{W_{N}^*-\Delta_{N,2}^*(m)\}])\right| \\
	\notag& & +|t|\left|\E_{**}(\Delta_{N,2}^*(m)\exp[\iota t\{W_{N}^*-\Delta_{N,2}^*(m)\}]\right| \\
	\notag& & +C_{6,4}t^2n_0^{-1}N^{-1}m \\
	\notag & \leq & C_{6,5}\left(n_0^{-a}+|t|n_0^{3/2-a}N^{-1}+t^2n_0^{-1}N^{-1}m\right)
	\end{eqnarray}
	in probability and this immediately implies
	\begin{eqnarray}
	\notag & & \frac{1}{\pi}\int_{n_0^{1/4}(\log n_0)^{-1}\leq|t|\leq a_{\varepsilon}n_0^{1/2}}\frac{1}{|t|}|\phi_{W_N}^*(t)|dt \\
	\notag & \leq &  \frac{C_{6,5}}{\pi}\int_{n_0^{1/4}(\log n_0)^{-1}\leq|t|\leq a_{\varepsilon}n_0^{1/2}}\frac{1}{|t|}(n_0^{-a}+|t|n_0^{3/2-a}N^{-1}+t^2n_0^{-1}N^{-1}m)dt \\
	\notag & = & o_p(n_0^{-1/2}).
	\end{eqnarray}
	
	Next, consider the case of $m<n_0$. According to Lemma \ref{lemma_06}, we assume $y_{N,1}^*=y_{N,1}, \ldots, y_{N,m}^*=y_{N,m}$, without loss of generality. Then, use the same technique as \eqref{eq:proof_theo_3.2_11}, we can show that
	\begin{eqnarray}
	\notag \left|\prod_{i=1}^{m}\E_{**}\{\exp(\iota tX_{N,i}^*)\}\right|  = O_p(m^{-a}).
	\end{eqnarray}
	The following procedure is similar to the case of $m\geq n_0$. Thus, we arrive at
	\begin{eqnarray}
	\notag & & \hat{F}_{N,Poi}^*(z) \\
	& = & \Phi(z)+\left\{\frac{\mu_{N,Poi}^{(3)*}}{6(V_{N,Poi}^*)^{3/2}}(1-z^2)+\frac{\tau_{N,Poi}^{(3)*}}{2(V_{N,Poi}^*)^{3/2}}z^2\right\}\phi(z) +o_p(n_0^{-1/2}). \notag
	\end{eqnarray}
	It remains to show that
	$$ \hat{V}_{N,Poi}^{-3/2}\hat{\mu}_{N,Poi}^{(3)}-(V_{N,Poi}^*)^{-3/2}\mu_{N,Poi}^{(3)*}=o_p(n_0^{-1/2}) $$
	and
	$$ \hat{V}_{N,Poi}^{-3/2}\hat{\tau}_{N,Poi}^{(3)}-(V_{N,Poi}^*)^{-3/2}\tau_{N,Poi}^{(3)*}=o_p(n_0^{-1/2}). $$
	Mentioned that
	\begin{eqnarray}
	\notag & & \hat{V}_{N,Poi}^{-3/2}\hat{\mu}_{N,Poi}^{(3)}-(V_{N,Poi}^*)^{-3/2}\mu_{N,Poi}^{(3)*} \\
	\notag & = & n_0^{-1/2}\left\{\left(n_0N^{-2}\hat{V}_{N,Poi}\right)^{-3/2}\left(n_0^2N^{-3}\hat{\mu}_{N,Poi}^{(3)}\right)\right.\\
	\notag & & ~~~~~~~~~~~~~\left.-\left(n_0N^{-2}V_{N,Poi}^*\right)^{-3/2}\left(n_0^2N^{-3}\mu_{N,Poi}^{(3)*}\right)\right\}
	\end{eqnarray}
	and
	\begin{eqnarray}
	\notag & & \hat{V}_{N,Poi}^{-3/2}\hat{\tau}_{N,Poi}^{(3)}-(V_{N,Poi}^*)^{-3/2}\tau_{N,Poi}^{(3)*} \\
	\notag & = & n_0^{-1/2}\left\{\left(n_0N^{-2}\hat{V}_{N,Poi}\right)^{-3/2}\left(n_0^2N^{-3}\hat{\tau}_{N,Poi}^{(3)}\right)\right.\\
	\notag & & ~~~~~~~~~~~~~\left.-\left(n_0N^{-2}V_{N,Poi}^*\right)^{-3/2}\left(n_0^2N^{-3}\tau_{N,Poi}^{(3)*}\right)\right\},
	\end{eqnarray}
	it suffices to prove that
	\begin{eqnarray}
	n_0N^{-2}({V}_{N,Poi}^* -\hat{V}_{N,Poi})&\to& 0, \label{eq: th 2 4} \\
	n_0^2N^{-3}(\mu_{N,Poi}^{(3)*}- \hat{\mu}_{N,Poi}^{(3)}) &\to& 0, \label{eq: th 2 5} \\
	n_0^2N^{-3}(\tau_{N,Poi}^{(3)*}- \hat{\tau}_{N,Poi}^{(3)}) &\to& 0, \label{eq: th 2 6}
	\end{eqnarray}
	in probability conditional on the series of realized samples. For (\ref{eq: th 2 4}), it is a consequence result of
	$n_0N^{-2}{V}_{N,Poi}^* = \sigma_1^2+o_p(1)$ and Lemma \ref{lemma: almost sure variance poi}. For (\ref{eq: th 2 5}), 
	consider
	\begin{eqnarray}
	&&	E_*(n_0^2N^{-3}\mu_{N,Poi}^{(3)*}) \notag \\ 
	&=& n_0^2N^{-3}\sum_{i=1}^n\frac{N\pi_{N,i}^{-1}}{\sum_{j=1}^n\pi_{N,j}^{-1} }y_{N,i}^3(1-\pi_{N,i})\{(1-\pi_{N,i})^2\pi_{N,i}^{-2}-1\}\notag \\ 
	&=& n_0^2N^{-3}\hat{\mu}_{N,Poi}^{(3)}\{1+o(1)\} \label{eq: E*}
	\end{eqnarray}
	almost surely, where the second equality holds by (\ref{eq:proof_theo_3.2_01}).	
	Next, consider 
	\begin{eqnarray}
	&&\mathrm{var}_*(n_0^2N^{-3}\mu_{N,Poi}^{(3)*})\notag \\ 
	&\leq&4n_0^4N^{-6}\sum_{i=1}^n\frac{N\pi_{N,i}^{-1}}{\sum_{j=1}^n\pi_{N,j}^{-1} }y_{N,i}^6\pi_{N,i}^{-4}\notag \\ 
	&\leq& 4C_1^{-5}n_0^{-1}N^{-1}\sum_{i=1}^Ny_{N,i}^6\{1+o(1)\}=o(1)\label{eq: v*}
	\end{eqnarray}
	almost surely, where the second inequality holds by (\ref{eq:proof_theo_3.2_01}) and the last equality holds by (C\ref{cond: C4poi}).
	By (\ref{eq: lemma1 1.2}) in Lemma \ref{lemma: almost sure variance poi}, \eqref{eq: E*} and \eqref{eq: v*}, we have proved (\ref{eq: th 2 5}), and the proof of (\ref{eq: th 2 6}) is similar.	This concludes the proof of Theorem \ref{theorem: e 2}.
\end{proof}

\begin{lemma}\label{lemma: srs expectations}
	Let $i,j,k,l$ be pairwise distinct positive integers, which are no larger than $N$. Suppose that  (C\ref{cond: srs2}) holds. Under SRS, we have
	\begin{eqnarray}
	&&E[(I_{N,i}\pi_{N,i}^{-1}-1)^4\mid\mathcal{F}_N] =O(n^{-3}N^3),\label{eq: srs 1}\\
	&&E[(I_{N,i}\pi_{N,i}^{-1}-1)^3(I_{N,j}\pi_{N,j}^{-1}-1)\mid\mathcal{F}_N] = O(N^2n^{-2}), \label{eq: srs 2}\\
	&&E[(I_{N,i}\pi_{N,i}^{-1}-1)^2(I_{N,j}\pi_{N,j}^{-1}-1)^2\mid\mathcal{F}_N] = O(N^2n^{-2}), \label{eq: srs 3}\\
	&&E[(I_{N,i}\pi_{N,i}^{-1}-1)(I_{N,j}\pi_{N,j}^{-1}-1)(I_{N,k}\pi_{N,k}^{-1}-1)^2\mid\mathcal{F}_N] \notag \\ 
	&&= O(Nn^{-2}), \label{eq: srs 4}\\
	&&E[(I_{N,i}\pi_{N,i}^{-1}-1)(I_{N,j}\pi_{N,j}^{-1}-1)(I_{N,k}\pi_{N,k}^{-1}-1)(I_{N,l}\pi_{N,l}^{-1}-1)\mid\mathcal{F}_N] \notag \\ 
	&&= O(n^{-2}). \label{eq: srs 5}
	\end{eqnarray}
\end{lemma}
\begin{proof}[Proof of Lemma \ref{lemma: srs expectations}]
	Consider
	\begin{eqnarray}
	E[(I_{N,i}\pi_{N,i}^{-1}-1)^4\mid\mathcal{F}_N] &=& \pi_{N,i}(\pi_{N,i}^{-1}-1)^4  +(1-\pi_{N,i}) \notag\\ 
	&=&(1-\pi_{N,i})[(1-\pi_{N,i})^3\pi_{N,i}^{-3}+1]\notag\\
	&\leq& N^3n^{-3}, \notag
	\end{eqnarray}
	where the last inequality holds by the fact that $(1-x)^3 + x^3 \leq1$ for $x\in[0,1]$. Thus, we have proved (\ref{eq: srs 1}).
	
	Denote $\#{A}$ to be the number of elements that equal to 1 in set $A$. Under SRS, we have 
	\begin{eqnarray}
	P_N(\#\{I_{N,i},I_{N,j}\}=2) &=& \displaystyle\frac{n(n-1)}{N(N-1)},\notag\\ 
	P_N(\{I_{N,i}=1,I_{N,j}=0\}) &=& \displaystyle\frac{n(N-n)}{N(N-1)} ,\notag \\
	P_N(\#\{I_{N,i},I_{N,j}\}=0) &=& \displaystyle\frac{(N-n)(N-n-1)}{N(N-1)},\notag
	\end{eqnarray}
	Under SRS, we have $\pi_{N,i}^{-1}-1 = (N-n)n^{-1}$ for $i=1,\ldots,N$. Consider 
	\begin{eqnarray}
	&&E[(I_{N,i}\pi_{N,i}^{-1}-1)^3(I_{N,j}\pi_{N,j}^{-1}-1)\mid\mathcal{F}_N] \notag \\
	&=&   \frac{(N-n)^4}{n^4}\frac{n(n-1)}{N(N-1)} - \frac{(N-n)^3}{n^3}\frac{n(N-n)}{N(N-1)} \notag\\
	&&-\frac{N-n}{n}\frac{n(N-n)}{N(N-1)} + \frac{(N-n)(N-n-1)}{N(N-1)} \notag\\
	&=& -\frac{(N-n)^4}{n^3N(N-1)} + O(1)\notag \\
	&=& O(N^2n^{-2}),\label{eq: lemma 30}
	\end{eqnarray}
	where the last equality holds by the facts that ${(N-n)^4}{[n^3N(N-1)]^{-1}}=O(N^2n^{-3})$ and $N^2n^{-2} = O(1)$ if $n\asymp N$. Thus, we have proved (\ref{eq: srs 2}) by (\ref{eq: lemma 30}). 
	
	Consider 
	\begin{eqnarray}
	&&E[(I_{N,i}\pi_{N,i}^{-1}-1)^2(I_{N,j}\pi_{N,j}^{-1}-1)^2\mid\mathcal{F}_N]\notag \\ 
	&=&\frac{(N-n)^4}{n^4}\frac{n(n-1)}{N(N-1)}  + 2\frac{(N-n)^2}{n^2}\frac{n(N-n)}{N(N-1)} + \frac{(N-n)(N-n-1)}{N(N-1)} \notag \\ 
	&=&O(N^2n^{-2}) + O(Nn^{-1}) + O(1)\notag \\ 
	&=& O(N^2n^{-2}),\notag
	\end{eqnarray}
	which proves (\ref{eq: srs 3}).
	
	Similar with the case for two terms, we have the following results under SRS. That is, 
	\begin{eqnarray}
	P_N(\#\{I_{N,i},I_{N,j},I_{N,k}\}=3) &=&\displaystyle\frac{n(n-1)(n-2)}{N(N-1)(N-2)} ,\notag\\ 
	P_N(\{I_{N,i}=1,I_{N,j}=1,I_{N,k}=0\}) &=& \displaystyle\frac{n(n-1)(N-n)}{N(N-1)(N-2)} ,\notag\\
	P_N(\{I_{N,i}=1,I_{N,j}=0,I_{N,k}=0\}) &=& \displaystyle\frac{n(N-n)(N-n-1)}{N(N-1)(N-2)},\notag\\
	P_N(\#\{I_{N,i},I_{N,j},I_{N,k}\}=0) &=& \displaystyle\frac{(N-n)(N-n-1)(N-n-2)}{N(N-1)(N-2)}.\notag
	\end{eqnarray}
	
	Consider
	\begin{eqnarray}
	&&E[(I_{N,i}\pi_{N,i}^{-1}-1)(I_{N,j}\pi_{N,j}^{-1}-1)(I_{N,k}\pi_{N,k}^{-1}-1)^2\mid\mathcal{F}_N] \notag \\
	&=& \frac{(N-n)^4}{n^4}\frac{n(n-1)(n-2)}{N(N-1)(N-2)} +  \frac{(N-n)^2}{n^2}\frac{n(n-1)(N-n)}{N(N-1)(N-2)} \notag 
	\\  &&-2\frac{(N-n)^3}{n^3}\frac{n(n-1)(N-n)}{N(N-1)(N-2)}-	2\frac{N-n}{n}\frac{n(N-n)(N-n-1)}{N(N-1)(N-2)}	\notag \\ 
	&&	+\frac{(N-n)^2}{n^2}\frac{n(N-n)(N-n-1)}{N(N-1)(N-2)} +\frac{(N-n)(N-n-1)(N-n-2)}{N(N-1)(N-2)}\notag \\
	&=& \frac{(n-1)(n-2)(N-n)^4}{n^3N(N-1)(N-2)} + \frac{(N-n)^3}{nN(N-1)}-2\frac{(n-1)(N-n)^4}{n^2N(N-1)(N-2)}  \notag \\ 
	&& +\frac{(N-n)(N-n-1)(N-n-2) - 2(N-n)^2(N-n-1)}{N(N-1)(N-2)}.\label{eq: srs 1.1}
	\end{eqnarray}
	After some algebra, the first three terms of (\ref{eq: srs 1.1}) are 
	\begin{eqnarray}
	&&\frac{(n-1)(n-2)(N-n)^4}{n^3N(N-1)(N-2)} + \frac{(N-n)^3}{nN(N-1)}\notag \\ &&-2\frac{(n-1)(N-n)^4}{n^2N(N-1)(N-2)}\notag \\ 
	&=& \frac{(N-n)^3}{N(N-1)(N-2)}+ O(Nn^{-2}).\label{eq: srs 1.2}
	\end{eqnarray}
	Together with (\ref{eq: srs 1.1}) and (\ref{eq: srs 1.2}), we have 
	\begin{eqnarray}
	&&E[(I_{N,i}\pi_{N,i}^{-1}-1)(I_{N,j}\pi_{N,j}^{-1}-1)(I_{N,k}\pi_{N,k}^{-1}-1)^2\mid\mathcal{F}_N] \notag \\
	&=& O(Nn^{-2})+\frac{(N-n)^3-(N-n)(N-n-1)(N-n+2)}{N(N-1)(N-2)} \notag\\
	&=& O(Nn^{-2}), \label{eq: lemma 4}
	\end{eqnarray}
	where the last equality holds by (C\ref{cond: srs2}) and the fact that the second term in the first equality converges to 0. Thus, we have shown (\ref{eq: srs 4}) by (\ref{eq: lemma 4}).

	For four terms under SRS, we have 
	\begin{eqnarray}
	P_N(\#\{I_{N,i},I_{N,j},I_{N,k},I_{N,l}\}=4) &=&\displaystyle \frac{n(n-1)(n-2)(n-3)} {N(N-1)(N-2)(N-3)},\notag\\ 
	P_N(\{I_{N,i}=1,I_{N,j}=1,I_{N,k}=1,I_{N,l}=0\}) &=&\displaystyle \frac{n(n-1)(n-2)(N-n)} {N(N-1)(N-2)(N-3)},\notag\\
	P_N(\{I_{N,i}=1,I_{N,j}=1,I_{N,k}=0,I_{N,l}=0\}) &=& \displaystyle\frac{n(n-1)(N-n)(N-n-1)} {N(N-1)(N-2)(N-3)},\notag\\
	P_N(\{I_{N,i}=1,I_{N,j}=0,I_{N,k}=0,I_{N,l}=0\}) &=& \displaystyle\frac{n(N-n)(N-n-1)(N-n-2)} {N(N-1)(N-2)(N-3)},\notag\\
	P_N(\#\{I_{N,i},I_{N,j},I_{N,k},I_{N,l}\}=0) &=&
	\displaystyle	\frac{(N-n)(N-n-1)(N-n-2)(N-n-3)} {N(N-1)(N-2)(N-3)}.\notag
	\end{eqnarray}
	
	Now, consider 
	\begin{eqnarray}
	&&E[(I_{N,i}\pi_{N,i}^{-1}-1)(I_{N,j}\pi_{N,j}^{-1}-1)(I_{N,k}\pi_{N,k}^{-1}-1)(I_{N,l}\pi_{N,l}^{-1}-1)\mid\mathcal{F}_N]\notag \\ 
	&=&\frac{(N-n)^4}{n^4} \frac{n(n-1)(n-2)(n-3)} {N(N-1)(N-2)(N-3)}\notag \\ 
	&& - 4\frac{(N-n)^3}{n^3}\frac{n(n-1)(n-2)(N-n)} {N(N-1)(N-2)(N-3)}\notag \\
	&&+6\frac{(N-n)^2}{n^2}\frac{n(n-1)(N-n)(N-n-1)} {N(N-1)(N-2)(N-3)} \notag \\ 
	&& -4\frac{(N-n)}{n}\frac{n(N-n)(N-n-1)(N-n-2)} {N(N-1)(N-2)(N-3)} \notag \\
	&&+\frac{(N-n)(N-n-1)(N-n-2)(N-n-3)} {N(N-1)(N-2)(N-3)}.\label{eq: srsp 2}
	\end{eqnarray}
	
	Consider 
	\begin{eqnarray}
	&&\frac{(n-1)(n-2)(n-3)}{n^3}(N-n)^4 \notag\\
	&& -4\frac{(n-1)(n-2)}{n^2}(N-n)^4  \notag \\ 
	&&+6 \frac{(n-1)}{n}(N-n)^3(N-n-1)  \notag\\
	&&-4(N-n)^2(N-n-1)(N-n-2)  \notag\\
	&&+(N-n)(N-n-1)(N-n-2)(N-n-3)\notag \\ 
	&=&\frac{3(N-n)^4}{n^2} - \frac{6(N-n)^4}{n^3} + \frac{6(N-n)^3}{n} \notag\\
	&& +3(N-n)^2 - 6(N-n)\notag \\ 
	&=& O(N^4n^{-2}),\label{eq: psrs 2}
	\end{eqnarray}
	where the last equality is valid by (C\ref{cond: srs2}). Together with (\ref{eq: srsp 2}) and (\ref{eq: psrs 2}), we have 
	\begin{eqnarray}
	&&E[(I_{N,i}\pi_{N,i}^{-1}-1)(I_{N,j}\pi_{N,j}^{-1}-1)(I_{N,k}\pi_{N,k}^{-1}-1)(I_{N,l}\pi_{N,l}^{-1}-1)\mid\mathcal{F}_N]\notag \\ 
	&=& O(N^4n^{-2}) \{N(N-1)(N-2)(N-3)\}^{-1} \notag \\
	&=& O(n^{-2}).\notag
	\end{eqnarray}
	Thus, we have proved (\ref{eq: srs 5}).
\end{proof}

\begin{proof}[Proof of Lemma \ref{lemma: srs 1}] Based on basic algebra and (C\ref{cond: C4poi}), we can show   (\ref{eq: lemma 2.2}), so the proof is omitted here.
	
	Note that 
	\begin{eqnarray}
	s_{N,SRS}^2 &=& n^{-1}\sum_{i=1}^{n}y_i^2 -\left(n^{-1}\sum_{i=1}^{n}y_i\right)^2,\notag \\ 
	\sigma_{N,SRS}^2 &=& N^{-1}\sum_{i=1}^{N}y_i^2 - \left(N^{-1}\sum_{i=1}^{N}y_i\right)^2.\notag
	\end{eqnarray}
	To show (\ref{eq: lemma srs 1.1}), it is enough to show that
	\begin{eqnarray}
	n^{-1}\sum_{i=1}^{n}y_i^2 - N^{-1}\sum_{i=1}^{N}y_i^2 &\to&0,\label{eq: psrs 1.1}\\
	n^{-1}\sum_{i=1}^{n}y_i - N^{-1}\sum_{i=1}^{N}y_i &\to&0\label{eq: psrs 1.2}
	\end{eqnarray} 
	almost surely. 
	
	First, we show (\ref{eq: psrs 1.1}), and we have 
	$$
	n^{-1}\sum_{i=1}^{n}y_i^2 - N^{-1}\sum_{i=1}^{N}y_i^2 = N^{-1}\sum_{i=1}^{N}(I_{N,i}\pi_{N,i}^{-1}-1)y_{N,i}^2.
	$$
	Based on (C\ref{cond: srs2}) and the proof of  Lemma \ref{lemma: almost sure variance poi}, it is enough to show that 
	$$
	N^{-4}E\left[\left\{\sum_{i=1}^{N}(I_{N,i}\pi_{N,i}^{-1}-1)y_{N,i}^2\right\}^4\mid\mathcal{F}_N \right]=O(n^{-2}).\label{eq: psrs 1}
	$$
	
	By some basic algebra and (C\ref{cond: C4poi}), we have 
	\begin{eqnarray}
	\sum_{i\neq j}y_{N,i}^6y_{N,j}^2 &=& O(N^2), \label{eq: sum 6,2}\\
	\sum_{i\neq j\neq k}y_{N,i}^2y_{N,j}^2y_{N,k}^4  &=& O(N^3), \label{eq: sum 2,2,4}\\
	\sum_{i\neq j\neq k\neq l}y_{N,i}^2y_{N,j}^2y_{N,k}^2y_{N,l}^2  &=& O(N^4), \label{eq: sum 2,2,2,2}
	\end{eqnarray}
	where $i\neq j\neq k$  and $i\neq j\neq k\neq l$ are defined similarly as $i\neq j$ for (\ref{eq: sum 4,4}).
	
	Consider 
	\begin{eqnarray}
	&& N^{-4}E\left[\left\{\sum_{i=1}^{N}(I_{N,i}\pi_{N,i}^{-1}-1)y_{N,i}^2\right\}^4\mid\mathcal{F}_N \right]\notag \\ 
	&=& N^{-4}\sum_{i=1}^{N}E\{(I_{N,i}\pi_{N,i}^{-1}-1)^4\mid\mathcal{F}_N\}y_{N,i}^8\notag\\
	&& +\notag  N^{-4}\sum_{i\neq j}E\{(I_{N,i}\pi_{N,i}^{-1}-1)^2(I_{N,j}\pi_{N,j}^{-1}-1)^2\mid\mathcal{F}_N\}y_{N,i}^4y_{N,j}^4\notag\\
	&& + N^{-4}\sum_{i\neq j}E\{(I_{N,i}\pi_{N,i}^{-1}-1)^3(I_{N,j}\pi_{N,j}^{-1}-1)\mid\mathcal{F}_N\}y_{N,i}^6y_{N,j}^2  \notag \\
	&& + N^{-4}\sum_{i\neq j\neq k}E\{(I_{N,i}\pi_{N,i}^{-1}-1)(I_{N,j}\pi_{N,j}^{-1}-1)(I_{N,k}\pi_{N,k}^{-1}-1)^2\mid\mathcal{F}_N\}y_{N,i}^2y_{N,j}^2y_{N,k}^4  \notag \\ 
	&&+ N^{-4}\sum_{i\neq j\neq k\neq l}E\{(I_{N,i}\pi_{N,i}^{-1}-1)(I_{N,j}\pi_{N,j}^{-1}-1)(I_{N,k}\pi_{N,k}^{-1}-1)(I_{N,l}\pi_{N,l}^{-1}-1)\mid\mathcal{F}_N\}y_{N,i}^2y_{N,j}^2y_{N,k}^2y_{N,l}^2\notag \\
	&=& O(n^{-2}),\notag
	\end{eqnarray}
	where the last equality holds by Lemma \ref{lemma: srs expectations}, (\ref{eq: sum 4,4}) and (\ref{eq: sum 6,2}) to (\ref{eq: sum 2,2,2,2}). Thus, we have proved (\ref{eq: psrs 1.1}). Similarly, we can prove (\ref{eq: psrs 1.2}). 
	
	Note that 
	$$
	{\mu}^{(3)}_{N,SRS}=N^{-1}\sum_{i=1}^{N}y_{N,i}^3 - 3\bar{Y}_NN^{-1}\sum_{i=1}^{N}y_{N,i}^2 + 2\bar{Y}_N^3.
	$$
	To show (\ref{eq: lemma srs 1.2}), by (\ref{eq: psrs 1.1})  and (\ref{eq: psrs 1.2}), it remains to show 
	\begin{eqnarray}
	n^{-1}\sum_{i=1}^ny_i^3 - N^{-1}\sum_{i=1}^{N}y_{N,i}^3  &\to& 0, \label{eq: psrs 2.3} 
	\end{eqnarray}
	in probability.
	
	Note that 
	$$
	n^{-1}\sum_{i=1}^ny_i^3 - N^{-1}\sum_{i=1}^{N}y_{N,i}^3 = N^{-1}\sum_{i=1}^{N}(I_{N,i}\pi_{N,i}^{-1} -1)y_{N,i}^3. 
	$$
	Consider 
	\begin{eqnarray}
	E\left(n^{-1}\sum_{i=1}^ny_i^3 - N^{-1}\sum_{i=1}^{N}y_{N,i}^3\mid\mathcal{F}_N\right) &=& 0\label{eq: psrs 3.1}\\
	\mathrm{var}\left(n^{-1}\sum_{i=1}^ny_i^3 - N^{-1}\sum_{i=1}^{N}y_{N,i}^3\mid\mathcal{F}_N\right)&=& n^{-1}(1-nN^{-1})\sigma_{N,3}^2\notag \\ 
	&=& O(n^{-1}),\label{eq: psrs 3.2}
	\end{eqnarray}
	where (\ref{eq: psrs 3.1}) holds by the sampling design, $\sigma_{N,3}^2$ is the finite population variance of $\{y_{N,1}^3,\ldots,y_{N,N}^3\}$, and the second equality of (\ref{eq: psrs 3.2}) holds by (C\ref{cond: C4poi}).
	Together with (\ref{eq: psrs 3.1}) and (\ref{eq: psrs 3.2}), we have proved (\ref{eq: psrs 2.3}).
\end{proof}

\begin{proof}[Proof of Theorem \ref{theorem: srs 2}]
	First, we show that 
	\begin{equation}
	N^{-1}\sum_{i=1}^N(y_{N,i}^*)^8 = O_p(1).\label{eq: srs2 1}
	\end{equation}
	Consider 
	\begin{eqnarray}
	N^{-1}E_*\left\{\sum_{i=1}^N(y_{N,i}^*)^8\right\} &=& n^{-1}\sum_{i=1}^ny_{N,i}^8\label{eq: srsp 1.1} \\ 
	E\left(n^{-1}\sum_{i=1}^ny_{N,i}^8\mid\mathcal{F}_N\right) &=& N^{-1}\sum_{i=1}^Ny_{N,i}^8 \notag \\ 
	&=& O(1).\label{eq: srsp 1.2}
	\end{eqnarray}
	Together with (\ref{eq: srsp 1.1}) and (\ref{eq: srsp 1.2}), we have proved (\ref{eq: srs2 1}) using the Markov inequality.
	
	By (C\ref{cond: srs1}), there exists a strongly non-latticed distribution $G_{SRS}$ such that 
	\begin{equation}
	N^{-1}\sum_{i=1}^N\exp(\iota ty_i)\to \int \exp(\iota tx)\mathrm{d}G_{SRS}(x) \label{eq: srs2 22}
	\end{equation} 
	as $N\to\infty$ for $t\in\mathbb{R}$. Next, we show that 
	\begin{equation}
	N^{-1}\sum_{i=1}^N\exp(\iota ty_i^*)\to \int \exp(\iota tx)\mathrm{d}G_{SRS}(x) \label{eq: srs2 2}
	\end{equation}
	in probability 	as $N\to\infty$ for $t\in\mathbb{R}$, where $\{y_i^*:i=1,\ldots,N\}$ is the bootstrap finite population.
	
	By Euler's formula, we have
	\begin{eqnarray}
	&\displaystyle N^{-1}\sum_{i=1}^N\exp(\iota ty_i) = N^{-1}\sum_{i=1}^N\cos(ty_i) + \iota N^{-1}\sum_{i=1}^N\sin(ty_i),\notag \\ 
	&\displaystyle N^{-1}\sum_{i=1}^N\exp(\iota ty_i^*) = N^{-1}\sum_{i=1}^n N_i^*\cos(ty_i) + \iota N^{-1}\sum_{i=1}^nN_i^*\sin(ty_i).\notag 
	\end{eqnarray} 
	It is enough to show that 
	\begin{eqnarray}
	& \displaystyle N^{-1}\sum_{i=1}^n N_i^*\cos(ty_i) = N^{-1}\sum_{i=1}^N\cos(ty_i) + o_p(1) , \label{eq: srs2 2.1} \\ 
	& \displaystyle N^{-1}\sum_{i=1}^nN_i^*\sin(ty_i) =  N^{-1}\sum_{i=1}^N\sin(ty_i) + o_p(1).\label{eq: srs2 2.2} 
	\end{eqnarray}
	We only show the result (\ref{eq: srs2 2.1}), and (\ref{eq: srs2 2.2}) can be obtained in a similar manner.
	
	By the first step of the proposed bootstrap method, we have \
	\begin{eqnarray}
	N^{-1}E_*\left\{\sum_{i=1}^n N_i^*\cos(ty_i)\right\} &=& n^{-1}\sum_{i=1}^n \cos(ty_i)\label{eq: srs2 3.1}\\
	N^{-2}\mathrm{var}_*\left\{\sum_{i=1}^n N_i^*\cos(ty_i)\right\}&\leq&	N^{-2}\sum_{i=1}^nNn^{-1}\notag \\ 
	&=& o(1)\label{eq: srs2 3.2},
	\end{eqnarray}
	where the inequality of (\ref{eq: srs2 3.2}) holds by the negative correlation among $N_i^*$ and the fact that $\lvert\cos(ty_i)\rvert\leq 1$. Using similar argument in (\ref{eq: psrs 3.1}) and (\ref{eq: psrs 3.2}), we can have shown (\ref{eq: srs2 2.1}) by results in (\ref{eq: srs2 3.1}) and (\ref{eq: srs2 3.2}).	By (\ref{eq: srs2 2.1}) and (\ref{eq: srs2 2.2}), we have proved (\ref{eq: srs2 2}).
	
	Thus, by (C\ref{cond: srs2}), (\ref{eq: srs2 1}) and (\ref{eq: srs2 2}), we have 
	\begin{equation}
	\hat{F}^*_{N,SRS}(z) = \Phi(z)  - \frac{(1-n/N)^{1/2}{\mu}^{(3)*}_{N,SRS}}{6n^{1/2} (\sigma_{N,SRS}^*)^3}\left\{\frac{1-2n/N}{1-n/N}(z^2-1)-3z^2\right\}\phi(z) + o_p(n^{-1/2})
	\notag\label{eq: srs2 boot edg}
	\end{equation}
	almost surely conditional on the generated bootstrap finite population, where ${\mu}^{(3)*}_{SRS}$ and $(\sigma_{SRS}^*)^2$ are the bootstrap central third moment and variance.
	
	Based on Lemma \ref{lemma: srs 1}, it remains to show that 
	\begin{eqnarray}
	{\mu}^{(3)*}_{N,SRS} - \hat{\mu}^{(3)}_{N,SRS} &\to& 0 \label{eq: srs2 4}\\
	(\sigma_{N,SRS}^*)^2 - s_{N,SRS}^2 &\to& 0 \label{eq: srs2 5}
	\end{eqnarray}
	in probability. By some algebra, it is equivalent to show 
	\begin{eqnarray}
	N^{-1}\sum_{i=1}^N(y_{N,i}^*)^3  - n^{-1}\sum_{i=1}^ny_{N,i}^3 &\to& 0, \label{eq: srs2 4.1} \\ 
	N^{-1}\sum_{i=1}^N(y_{N,i}^*)^2  - n^{-1}\sum_{i=1}^ny_{N,i}^2 &\to& 0, \label{eq: srs2 4.2} \\ 
	N^{-1}\sum_{i=1}^Ny_{N,i}^*  - n^{-1}\sum_{i=1}^ny_{N,i} &\to& 0, \label{eq: srs2 4.3}  
	\end{eqnarray}
	in probability. 
	
	Consider 
	\begin{eqnarray}
	N^{-1}E_*\left\{\sum_{i=1}^N(y_{N,i}^*)^3\right\} &=&n^{-1}\sum_{i=1}^ny_{N,i}^3\label{eq: srs 5.1} \\ 
	\mathrm{var}_* \left\{N^{-1}\sum_{i=1}^N(y_{N,i}^*)^3\right\} &\leq& N^{-2} \sum_{i=1}^nNn^{-1}y_{N,i}^6\notag \\ 
	&=& o_p(1), \label{eq: srs 5.2}
	\end{eqnarray}
	where the last equality of (\ref{eq: srs 5.2}) is derived by the Markov inequality and a similar procedure for (\ref{eq: srsp 1.2}). Thus, we have proved (\ref{eq: srs2 4.1}) by (\ref{eq: srs 5.1}) and (\ref{eq: srs 5.2}). Similarly, we can prove (\ref{eq: srs2 4.2}) and (\ref{eq: srs2 4.3}). Therefore, we have shown (\ref{eq: srs2 4}) and (\ref{eq: srs2 5}), which concludes the proof of Theorem \ref{theorem: srs 2}.
\end{proof}

\begin{proof}[Proof of Lemma \ref{lemma: pps 1}]
	
	Mentioned that $Z_{N,1},\ldots,Z_{N,n}\overset{i.i.d.}{\sim}G_{N,PPS}$ and $\mathbb{P}(Z_{N,i}=p_{N,k}^{-1}y_{N,k})=p_{N,k}$ for $i=1,\ldots,n$; $k=1,\ldots,N$,
	we have 
	\begin{eqnarray}
	E(N^{-\delta}\lvert Z_{N,i}\rvert^{\delta}\mid \mathcal{F}_N) &=& N^{-\delta}\sum_{i=1}^Np_{N,i}^{-(\delta-1)}\lvert y_{N,i}\rvert^{\delta}\notag \\ 
	&=& O\left(N^{-1}\sum_{i=1}^N\lvert y_{N,i}\rvert^{\delta}\right)\notag \\ 
	&=& O(1),\label{lemma: 1 1}
	\end{eqnarray}
	for all positive $\delta\leq 8$, where the second equality holds by (C\ref{cond: C4}), and the last equality holds by (C\ref{cond: C4poi}). 

	By the strong law of large numbers,
	\begin{equation}
	n^{-1}\sum_{i=1}^nN^{-2}Z_{N,i}^2 - E(N^{-2}Z_{N,i}^2\mid \mathcal{F}_N)\to 0\label{eq: lemma 2 1.1}
	\end{equation}
	almost surely and
	\begin{equation}
	n^{-1}\sum_{i=1}^nN^{-1}Z_{N,i} - E(N^{-1}Z_{N,i}\mid \mathcal{F}_N)\to 0\label{eq: lemma 2 1.2}
	\end{equation}
	almost surely.
	
	Note that 
	\begin{eqnarray}
	N^{-2}s_{N,PPS}^2 &=& n^{-1}\sum_{i=1}^nN^{-2}Z_{N,i}^2 - big(n^{-1}\sum_{i=1}^nN^{-1}Z_{N,i}\big)^2,\notag\\
	N^{-2} \sigma_{N,PPS}^2&=&E(N^{-2}Z_{N,i}^2\mid \mathcal{F}_N) - \{E(N^{-1}Z_{N,i}\mid \mathcal{F}_N)\}^2.\notag
	\end{eqnarray}
	By (\ref{eq: lemma 2 1.1}) and (\ref{eq: lemma 2 1.2}), we have proved (\ref{eq: lemma pps 1.2}). 
	
	Notice that 
	\begin{eqnarray}
	N^{-3}{\mu}^{(3)}_{PPS} &=& N^{-3}\sum_{i=1}^Np_{N,i}(p_{N,i}^{-3}y_{N,i}^3 - 3p_{N,i}^{-2}y_{N,i}^2Y_N+3p_{N,i}^{-1}y_{N,i}Y_N^2 - Y_N^3)\notag \\
	&=& N^{-3}\sum_{i=1}^Np_{N,i}^{-2}y_{N,i}^3 - 3N^{-3}Y_N\sum_{i=1}^Np_{N,i}^{-1}y_{N,i}^2+2N^{-3}Y_N^3\notag \\ &=& O(1),\label{eq: third moment}
	\end{eqnarray}
	where $Y_N = E(Z_{N,i}\mid \mathcal{F}_N) = \sum_{i=1}^Ny_{N,i}$, and the last equality of (\ref{eq: third moment}) holds by (C\ref{cond: C4}) and (\ref{lemma: 1 1}). In addition, for $\zeta=1,2,3$, we have
	\begin{equation}
	\notag E\left[\left\{n^{-1}\sum_{i=1}^nN^{-\zeta}Z_{N,i}^{\zeta} -E(N^{-\zeta}Z_{N,i}^{\zeta}\mid \mathcal{F}_N)\right\}\mid \mathcal{F}_N\right] = 0
	\end{equation}
	and
	\begin{eqnarray}
	\mathrm{var}\left(n^{-1}\sum_{i=1}^nN^{-\zeta}Z_{N,i}^{\zeta}\mid \mathcal{F}_N\right) \leq n^{-1}N^{-2\zeta}E(Z_{N,i}^{2\zeta}\mid \mathcal{F}_N)=O(n^{-1}).\notag 
	\end{eqnarray}
	By Markov's inequality, 
	\begin{eqnarray}
	\notag 	n^{-1}\sum_{i=1}^nN^{-\zeta}Z_{N,i}^{\zeta} -E(N^{-\zeta}Z_{N,i}^{\zeta}\mid \mathcal{F}_N) = O_p(n^{-1/2})
	\end{eqnarray}
	for $\zeta=1,2,3$, from which we can prove that $N^{-3}(\hat{\mu}^{(3)}_{N,PPS} - {\mu}^{(3)}_{N,PPS}) = O_p(n^{-1/2})$. Thus, we complete the proof of Lemma \ref{lemma: pps 1}.
\end{proof}

\begin{proof}[Proof of Theorem \ref{theorem: pps edgeworth origin}] 

	Rewrite
	\begin{eqnarray}
	\notag T_{N,PPS} &=& \hat{V}_{N,PPS}^{-1/2}(\hat{Y}_{N,PPS}-Y_N) \\
	\notag &=& n^{1/2}\left(n^{-1}\sum_{i=1}^{n}N^{-1}Z_{N,i}-N^{-1}Y_{N}\right) \\
	\notag & & \times\left\{n^{-1}\sum_{i=1}^{n}(N^{-1}Z_{N,i})^2-\left(n^{-1}\sum_{i=1}^{n}N^{-1}Z_{N,i}\right)^2\right\}^{-1/2},
	\end{eqnarray}
	where $N^{-1}Y_N=E(N^{-1}Z_{N,i}\mid\mathcal{F}_N)$. By \eqref{lemma: 1 1}, we have $E(|N^{-1}Z_{N,i}|^3\mid\mathcal{F}_N)<\infty$. Using the results of \citet{Hall1987Edgeworth}, as the distribution of $Z_{N,i}$ is non-lattice, we have
	\begin{equation}
	\hat{F}_{N,PPS}(z) = \Phi(z)  + \frac{{\mu}^{(3)}_{N,PPS}}{6\sqrt{n} \sigma_{N,PPS}^3}(2z^2+1)\phi(z) + o(n^{-1/2})\label{eq: th 2 0},\notag
	\end{equation}
	where $\sigma_{N,PPS}^2=E\{(Z_{N,i}-Y_N)^2\mid \mathcal{F}_N\}=\sum_{i=1}^{N}p_{N,i}(p_{N,i}^{-1}y_{N,i}-Y_{N})^2$ and $\mu^{(3)}_{N,PPS} = E\{(Z_{N,i}-Y_N)^3\mid \mathcal{F}_N\}=\sum_{i=1}^{N}p_{N,i}(p_{N,i}^{-1}y_{N,i}-Y_{N})^3$. Finally, according to Lemma \ref{lemma: pps 1}, we have shown Theorem \ref{theorem: pps edgeworth origin}.
\end{proof}

\begin{proof}[Proof of Theorem \ref{th: last one}]
	
	Mentioned that $\mathbb{P}_{PPS}(p_{N,a,i}=p_{N,k})=p_{N,k}$ for $i=1,\ldots,n$; $k=1,\ldots,N$ and $p_{N,a,1},\ldots,p_{N,a,n}$ are independent. For any $a$ such that $0\leq \delta\leq 8$,
	\begin{eqnarray}
	\notag E\left\{\left|n^{-1}\sum_{i=1}^nN^{-1}p_{N,a,i}^{-1}y_{N,a,i}^{\delta}\right|\mid\mathcal{F}_N\right\}\leq N^{-1}\sum_{i=1}^{N}|y_{N,i}|^{\delta}<\infty.
	\end{eqnarray}
	Thus, by SLLN,
	\begin{equation}
	n^{-1}\sum_{i=1}^nN^{-1}p_{N,a,i}^{-1}y_{N,a,i}^{\delta} \to N^{-1}\sum_{i=1}^{N}y_{N,i}^{\delta} \label{eq: last 1 1}
	\end{equation}
	with probability 1 for all $0\leq \delta\leq 8$.
	
	In the first step of our proposed bootstrap method, $y_{N,1}^*,\ldots,y_{N,N}^*$ are independently and identically distributed (i.i.d.) with $\mathbb{P}_*(y_{N,i}^*=y_{N,a,j})=\rho_{N,j}=p_{N,a,j}^{-1}\big(\sum_{\ell=1}^np_{N,a,\ell}^{-1}\big)^{-1}$ for $i=1,\ldots,N$; $j=1,\ldots,n$. Let $\mathbb{P}_{*}$ be the probability measure for the first step of the proposed bootstrap method conditional on the realized sample $\{y_{N,a,1},\ldots,y_{N,a,n}\}$. Then, we have
	\begin{eqnarray}
	\notag & & E_*\left\{N^{-1}\sum_{i=1}^N(y_{N,i}^*)^8\right\} \\
	&=& \sum_{i=1}^np_{N,a,i}^{-1}\left(\sum_{j=1}^np_{N,a,j}^{-1}\right)^{-1}y_{N,a,i}^8 \notag \\
	&=& \left(n^{-1}\sum_{i=1}^nN^{-1}p_{N,a,i}^{-1}\right)^{-1}n^{-1}\sum_{i=1}^{n}N^{-1}p_{N,a,i}^{-1}y_{N,a,i}^8 \notag \\
	&=& O_p(1), \label{eq: last 11}
	\end{eqnarray}
	from which, we get that	conditional on the series of realized samples,
	\begin{eqnarray}
	N^{-1}\sum_{i=1}^N(y_{N,i}^*)^8=O_p(1). \label{eq: last 12}
	\end{eqnarray}
	
	Recall that $C_N^* = \sum_{i=1}^nN_{a,i}^*p_{N,a,i}$, where $N_{a,i}^*$ is the number of repetitions of the $i$-th realized sample in the proposed bootstrap method. Next, we show
	\begin{eqnarray}
	C_N^* &=& 1+o_p(1) \label{eq: last 2}.
	\end{eqnarray} 
	Consider
	\begin{eqnarray}
	E_*(C_N^*) &=& \left(n^{-1}\sum_{i=1}^n N^{-1}p_{N,a,i}^{-1}\right)^{-1}, \label{eq: last 2.1} \\
	\mathrm{var}_*(C_N^*) &\leq&  N\left(\sum_{i=1}^n p_{N,a,i}^{-1}\right)^{-1}  \sum_{i=1}^np_{N,a,i}\notag \\
	&=& O\left(N^{-1} \left(n^{-1}\sum_{i=1}^n N^{-1}p_{N,a,i}^{-1}\right)^{-1}\right), \label{eq: last 2.2}
	\end{eqnarray}
	where the equality of (\ref{eq: last 2.2}) holds by (C\ref{cond: C4}).
	By \eqref{eq: last 2.1} and \eqref{eq: last 2.2}, we have shown \eqref{eq: last 2} according to \eqref{eq: last 1 1} with $\delta=0$.
	
	Similarly, we can show 
	\begin{equation}
	nN_{a,i}^*p_{N,a,i} = 1 + o_p(1) \label{eq: last 5.1}
	\end{equation}
	for $i=1,\ldots,n$.
	
	Let $((C_N^*)^{-1}p_{N,b,i}^*,y_{N,b,i}^*)$ be the quantities for the $i$-th selected element from the bootstrap finite population $\mathcal{F}_{N}^*$. Denote $Z_{N,i}^*=C_N^*(p_{N,b,i}^*)^{-1}y_{N,b,i}^*$ for $i=1,\ldots,n$. Then $\mathbb{P}_{PPS}^*\big(Z_{N,i}^* = C_N^*(p_{N,k}^*)^{-1}y_{N,k}^*\big) = (C_N^*)^{-1}p_{N,k}^*$ for $i=1,\ldots,n$; $k=1,\ldots,N$, where $\mathbb{P}_{PPS}^*$ is the counterpart of $\mathbb{P}_{PPS}$ conditional on the bootstrap finite population $\mathcal{F}_{N}^*$.
	
	Conditional on the bootstrap finite population $\mathcal{F}_{N}^*$, denote $F^*_{N,PPS}(z)=\mathbb{P}_{PPS}^*(T_{N,PPS}^*\leq z)$ as the distribution of  $T_{N,PPS}^*=(\hat{V}_{N,PPS}^*)^{-1/2}(\hat{Y}_{N,PPS}^*-Y_N^*)$, where $Y_N^* = \sum_{i=1}^Ny_{N,i}^*$, $\hat{Y}_{N,PPS}^*=n^{-1}\sum_{i=1}^{n}Z_{N,i}^*$ and
	${V}_{N,PPS}^* = n^{-2}\sum_{i=1}^n(Z_{N,i}^*-\bar{Z}_{N}^*)^2$ with $\bar{Z}_{N}^*=n^{-1}\sum_{i=1}^{n}Z_{N,i}^*=\hat{Y}_{N,PPS}^*$.
	
	Consider 
	\begin{eqnarray}
	E\left(n^{-1}\sum_{i=1}^nN^{-3}|Z_{N,i}|^3  \mid \mathcal{F}_N\right) &=& E(N^{-3}|Z_{N,i}|^3\mid \mathcal{F}_N) \notag \\
	&=& O(1), \label{eq: firstone} \\ 
	\mathrm{var}\left(n^{-1}\sum_{i=1}^nN^{-3}|Z_{N,i}|^3  \mid \mathcal{F}_N\right) &=& n^{-1}\mathrm{var}(N^{-3}|Z_{N,i}|^3\mid \mathcal{F}_N)\notag \\ 
	&\leq&  n^{-1}E(N^{-6}Z_{N,i}^6\mid \mathcal{F}_N)\notag \\
	&=&O(n^{-1}),\label{eq: secondone}
	\end{eqnarray}
	where the results of (\ref{eq: firstone}) and (\ref{eq: secondone}) are based on (\ref{lemma: 1 1}).
	
	Recall that $E_{**}(\cdot)$ is the expectation with respect to the sampling design conditional on the bootstrap finite population and $\{p_{N,k}^*:k=1,\ldots,N\}$ consists of $N_{a,i}^*$ copies of $p_{N,a,i}$ for $i=1,\ldots,n$. Consider
	\begin{eqnarray}
	\notag & & E_{**}\left(|N^{-1}Z_{N,i}^*|^3\right) \\
	&=& N^{-3}\sum_{i=1}^N(C_N^*)^{-1}p_{N,i}^* |C_N^*(p_{N,i}^*)^{-1}y_{N,i}^*|^3 \notag \\ 
	&=& N^{-3}(C_N^*)^2\sum_{i=1}^N(p_{N,i}^*)^{-2}|y_{N,i}^*|^3\notag \\ 
	&=& N^{-3}(C_N^*)^2\sum_{i=1}^nN_{a,i}^*p_{N,a,i}^{-2}|y_{N,a,i}|^3\notag\\
	&=& n^{-1}\sum_{i=1}^n|N^{-1}p_{N,a,i}^{-1}y_{N,a,i}|^3\{1+o_p(1)\} \notag\\
	&=&O_p(1),\label{eq: last 3}
	\end{eqnarray}
	where the fourth equality holds by (\ref{eq: last 5.1}), and last equality holds by Lemma \ref{lemma: pps 1}, (C\ref{cond: C4}), (\ref{eq: firstone}) and (\ref{eq: secondone}).

	Consider the characteristic function of $N^{-1}Z_{N,i}$ and $N^{-1}Z_{N,i}^*$. Specifically, the characteristic function of $N^{-1}Z_{N,i}$ is  
	\begin{eqnarray}
	\notag \psi_{Z,N}(t) &=& \sum_{i=1}^Np_{N,i}\exp(\iota t N^{-1}y_{N,i}/p_{N,i}) 
	\end{eqnarray}
	and the characteristic function of $N^{-1}Z_{N,i}^*$, conditional on the bootstrap finite population $\mathcal{F}_N^*$, is 
	\begin{eqnarray}
	\psi_{Z,N}^*(t)&=&\sum_{i=1}^N(C_N^*)^{-1}p_{N,i}^*\exp(\iota t N^{-1}C_N^*y^*_{N,i}/p^*_{N,i})\notag 
	\end{eqnarray}
	
	To show that the distribution of $Z_{N,i}^*$ is non-lattice in probability conditional on the bootstrap finite population $\mathcal{F}_N^*$, it is enough to show that, for any fixed $t_0>0$,
	\begin{eqnarray}\label{eq:nonlattice}
	\sup_{|t|\leq t_0}\left|\psi_{Z,N}^*(t)-\psi_{Z,N}(t)\right|\to 0
	\end{eqnarray}
	in probability as $n\to\infty$. By remarking that
	\begin{eqnarray}
	\notag & & \sup_{|t|\leq t_0}\left|\psi_{Z,N}^*(t)-\psi_{Z,N}(t)\right|\\
	\notag & \leq & \sup_{|t|\leq t_0}\left|\psi_{Z,N}^*(t)-\sum_{i=1}^N(C_N^*)^{-1}p_{N,i}^*\exp(\iota t N^{-1}y^*_{N,i}/p^*_{N,i})\right| \\
	& & + \sup_{|t|\leq t_0}\left|\sum_{i=1}^N(C_N^*)^{-1}p_{N,i}^*\exp(\iota t N^{-1}y^*_{N,i}/p^*_{N,i})-\psi_{Z,N}(t)\right|.
	\end{eqnarray}
	First,
	\begin{eqnarray}
	\notag & & \sup_{|t|\leq t_0}\left|\psi_{Z,N}^*(t)-\sum_{i=1}^N(C_N^*)^{-1}p_{N,i}^*\exp(\iota t N^{-1}y^*_{N,i}/p^*_{N,i})\right| \\
	\notag & \leq & \sup_{|t|\leq t_0}\sum_{i=1}^N(C_N^*)^{-1}p_{N,i}^*\left|\exp(\iota t N^{-1}C_N^*y^*_{N,i}/p^*_{N,i})-\exp(\iota t N^{-1}y^*_{N,i}/p^*_{N,i})\right| \\
	\notag & \leq & t_0N^{-1}\sum_{i=1}^N(C_N^*)^{-1}\left|y^*_{N,i}\right|\left|C_N^*-1\right| \to 0
	\end{eqnarray}
	in probability as $n\to\infty$.
	Second,
	\begin{eqnarray}
	\notag & & \sup_{|t|\leq t_0}\left|\sum_{i=1}^N(C_N^*)^{-1}p_{N,i}^*\exp(\iota t N^{-1}y^*_{N,i}/p^*_{N,i})-\psi_{Z,N}(t)\right| \\
	\notag & = & \sup_{|t|\leq t_0}\bigg|(C_N^*)^{-1}\sum_{i=1}^nN_i^*p_{N,a,i}\exp(\iota t N^{-1}y_{N,a,i}/p_{N,a,i}) \\
	\notag & & ~~~~~~-\sum_{i=1}^Np_{N,i}\cos( t N^{-1}y_{N,i}/p_{N,i}) - \iota\sum_{i=1}^Np_{N,i}\sin( t N^{-1}y_{N,i}/p_{N,i})\bigg| \\
	\notag & \leq & \bigg| (C_N^*)^{-1}\sum_{i=1}^nN_i^*p_{N,a,i}\cos(t N^{-1}y_{N,a,i}/p_{N,a,i}) \\
	\notag & & ~~~~~~~~~~~~~~~~~~~~~~~~~~~~~~~~~~-\sum_{i=1}^Np_{N,i}\cos( t N^{-1}y_{N,i}/p_{N,i})\bigg| \\
	\notag & & + \bigg|(C_N^*)^{-1}\sum_{i=1}^nN_i^*p_{N,a,i}\sin(t N^{-1}y_{N,a,i}/p_{N,a,i}) \\
	\notag & & ~~~~~~~~~~~~~~~~~~~~~~~~~~~~~~~~~~-\sum_{i=1}^Np_{N,i}\sin( t N^{-1}y_{N,i}/p_{N,i})\bigg|.
	\end{eqnarray}
	
	It suffices to show that 
	\begin{eqnarray}
	&&(C_N^*)^{-1}\sum_{i=1}^nN_i^*p_{N,a,i}\cos(t N^{-1}y_{N,a,i}/p_{N,a,i}) \notag \\
	&=& \sum_{i=1}^Np_{N,i}\cos( t N^{-1}y_{N,i}/p_{N,i})+o_p(1), \label{eq: pps 1.1}\\
	&&(C_N^*)^{-1}\sum_{i=1}^nN_i^*p_{N,a,i}\sin(t N^{-1}y_{N,a,i}/p_{N,a,i})\notag \\ 
	&=& \sum_{i=1}^Np_{N,i}\sin( t N^{-1}y_{N,i}/p_{N,i})+o_p(1). \label{eq: pps 1.2}
	\end{eqnarray} 
	We can show (\ref{eq: pps 1.1}) since (\ref{eq: last 2}) and (\ref{eq: last 5.1}) hold, $\lvert\cos(t y_{N,a,i}/p_{N,a,i})\rvert\leq 1$ and $Z_{N,1},\ldots,Z_{N,n}$ are independent and identically distributed random variables. Similarly, we can show (\ref{eq: pps 1.2}).
	By (\ref{eq: last 11}), (C\ref{cond: PPS1}), and the fact that the distribution of $Z_{N,i}^*$ is non-lattice in probability, we have the following result by \citet{Hall1987Edgeworth}:
	\begin{equation}
	\hat{F}_{N,PPS}^*(z) = \Phi(z)  + \frac{{\mu}^{(3)*}_{N,PPS}}{6\sqrt{n} (\sigma_{N,PPS}^*)^3}(2z^2+1)\phi(z) + o_p(n^{-1/2})\label{eq: last 5}
	\end{equation} 
	uniformly in $z\in\mathbb{R}$, where $\mu_{N,PPS}^{(3)*} = E_{**}\{(Z_{N,i}^* - Y_N^*)^3\}$.

	Based on a similar argument made for (\ref{eq: last 3}), we have 
	\begin{eqnarray}
	N^{-3}(\mu_{N,PPS}^{(3)*} - \hat{\mu}_{N,PPS}^{(3)}) &=&o_p(1) ,\label{eq: last 6.1}\\
	N^{-2}\{(\sigma^*_{N,PPS})^2 - s^2_{N,PPS}\} &=& o_p(1) \label{eq: last 6.2}.	
	\end{eqnarray}
	
	Together with Lemma \ref{lemma: pps 1},  (\ref{eq: last 5}) to (\ref{eq: last 6.2}) and (C\ref{cond: C9}), we have proved Theorem \ref{th: last one}. 
\end{proof}

\subsection{Design-unbiased estimates for the two-stage sampling designs}\label{appd: B}
For the two-stage sampling designs in the second simulation study,  Poisson sampling and PPS sampling are used in the first stage, and an SRS design is independently conducted within each selected cluster in the second stage. For the sampling designs in the first stage, denote $\pi_i = n_1N_iN^{-1}$ to be the first-order inclusion probability for Poisson sampling, and $p_i=N_iN^{-1}$ to be the selection probability for PPS sampling. In this section, we comply to the notation convention in Section 1.2.8 of \citet{fuller2009sampling} to discuss the variance estimation under the two-stage sampling designs.  

For the two-stage sampling design, where Poisson sampling is applied in the first stage, the design-based estimator of $\bar{Y}$ is
$$
\tilde{Y} = N^{-1}\sum_{i\in A}{\pi_i}^{-1}\hat{Y}_{i,\cdot} = n_1^{-1}\sum_{i\in A}N_i^{-1}\hat{Y}_{i,\cdot},
$$
where $A$ is the index of the selected clusters in the first stage, $\hat{Y}_{i,\cdot} = N_in_2^{-1}\sum_{j\in B_i}y_{i,j}$ is an design-unbiased estimate of the cluster total $Y_{i,\cdot} = \sum_{j=1}^{N_i}y_{i,j}$ under the SRS design, and $B_i$ is the index set of the sample within the selected cluster indexed by $i$.  It can be shown that the same form holds when PPS sampling is used in the first stage.

First, we discuss the variance estimator of $\tilde{Y}$ for the two-stage sampling design where Poisson sampling is used in the first stage. As shown in Section 1.2.8 by \citet{fuller2009sampling}, the variance of $\tilde{Y}$ can be decomposed into two parts. That is,
\begin{equation}
\mathrm{var}(\tilde{Y}\mid  U_N) = V_1 + V_2,\label{eq: Poi decomps}
\end{equation}
where $V_1 = E[\mathrm{var}\{\tilde{Y}\mid (A, U_N)\}\mid U_N ]$ and $V_2 = \mathrm{var}[E\{\tilde{Y}\mid (A, U_N)\}\mid U_N ].$

Consider
\begin{eqnarray}
\mathrm{var}\{\tilde{Y}\mid (A, U_N)\} &=& N^{-2}\sum_{i\in A}{\pi_i^{-2}}\mathrm{var}\{\hat{Y}_{i,\cdot}\mid(A, U_N)\}\label{eq: V1}
,\end{eqnarray}
where the equality holds since the SRS design is independently conducted within each selected cluster, $\mathrm{var}\{\hat{Y}_{i,\cdot}\mid(A, U_N)\} = N_i(N_i-n_2)n_2^{-1}S_{i}^2$, $S_i^2 = (N_i-1)^{-1}\sum_{j=1}^{N_i}(y_{i,j} - \bar{Y}_{i,\cdot})^2$ is the finite population variance within the $i$-th cluster, and $\bar{Y}_{i,\cdot} = N_i^{-1}Y_{i,\cdot}$ is the finite population mean of the $i$-th cluster. Since the sample variance $s_i^2 = (n_2-1)^{-1} \sum_{j\in B_i}(y_{i,j} - \tilde{Y}_{i,\cdot})^2$ is an unbiased estimator of $S_i^2$, where $\tilde{Y}_{i,\cdot} = N_i^{-1}\hat{Y}_{i,\cdot}$ is the estimated cluster mean, the first term of (\ref{eq: Poi decomps}) can be estimated by
\begin{equation}
\hat{V}_1 = N^{-2}\sum_{i\in A}\pi_i^{-2}\hat{V}\{\hat{Y}_{i,\cdot}\mid(A, U_N)\}, \label{eq: Poi first part apprx}
\end{equation}
where $\hat{V}\{\hat{Y}_{i,\cdot}\mid(A, U_N)\} = N_i(N_i-n_2)n_2^{-1}s_{i}^2$.

For the second term of (\ref{eq: Poi decomps}), consider
\begin{equation}
E\{\tilde{Y}\mid (A, U_N)\} = N^{-1}\sum_{i\in A}\pi_i^{-1}Y_{i,\cdot}.\notag
\end{equation}
Since Poisson sampling is used in the first stage, we have
\begin{eqnarray}
\mathrm{var}[E\{\tilde{Y}\mid (A, U_N)\}\mid U_N ]&=&N^{-2} \sum_{i=1}^H\pi_{i}^{-1}(1-\pi_i)Y_{i,\cdot}^2,\label{eq: Poi 2}
\end{eqnarray}
which can be estimated by
$
N^{-2} \sum_{i\in A}\pi_{i}^{-2}(1-\pi_i)Y_{i,\cdot}^2.
$
Notice that
\begin{eqnarray}
E \{\hat{Y}_{i,\cdot}^2\mid (A, U_N)\} &=& [E \{\hat{Y}_{i,\cdot}\mid (A, U_N)\}]^2 + \mathrm{var}\{\hat{Y}_{i,\cdot}\mid (A, U_N)\}\notag \\
&=& Y_{i,\cdot}^2 + \mathrm{var}\{\hat{Y}_{i,\cdot}\mid (A, U_N)\}.\label{eq: Poi 3}
\end{eqnarray}

By (\ref{eq: Poi 2}) and (\ref{eq: Poi 3}) and the fact that $s_i^2$ is an unbiased estimator of $S_i^2$, the second term of (\ref{eq: Poi decomps}) can be estimated by
\begin{equation}
\hat{V}_2 = N^{-2}\sum_{i\in A}\pi_i^{-2}(1-\pi_i)[\hat{Y}_{i,\cdot}^2 - \hat{V}\{\hat{Y}_{i,\cdot}\mid(A, U_N)\}]. \label{eq: Poi second part apprx}
\end{equation}

By (\ref{eq: Poi first part apprx}) and (\ref{eq: Poi second part apprx}), the variance of $\tilde{Y}$ can be estimated by
$$
\tilde{V} = N^{-2}\left[\sum_{i\in A}\pi_{i}^{-2}(1-\pi_i)\hat{Y}_{i,\cdot}^2 + \sum_{i\in A}\pi_i^{-1}\hat{V}\{\hat{Y}_{i,\cdot}\mid(A, U_N)\}\right],
$$
when Poisson sampling is used in the first stage.

Next, we use variance decomposition (\ref{eq: Poi decomps}) to derive the variance estimator of $\tilde{Y}$ under the two-stage sampling design where  PPS sampling is applied in the first stage. The result shown in (\ref{eq: V1}) holds, and we can still use (\ref{eq: Poi first part apprx}) to approximate $V_1$.

Consider
\begin{eqnarray}
&&\mathrm{var}[E\{\tilde{Y}\mid (A, U_N)\}\mid U_N ]\notag \\&=&N^{-2} n_1^{-1}(n_1-1)^{-1}\left(\sum_{i\in A}Z_{i,\cdot}^2 - n_1\bar{Z}^2 \right),\label{eq: PPS 2}
\end{eqnarray}
where the equality holds by the property of PPS sampling, $Z_{i,\cdot} = Y_{i,\cdot}p_i^{-1}$ and $\bar{Z}=n_1^{-1}\sum_{i\in A}Z_{i,\cdot}$. Based on (\ref{eq: Poi 3}), we can estimate $Z_{i,\cdot}^2$ by
$$
p_i^{-2}[\hat{Y}_{i,\cdot}^2 - \hat{V}\{\hat{Y}_{i,\cdot}\mid(A, U_N)\}].
$$

Consider
\begin{eqnarray}
E\left\{\tilde{Z}^2\mid(A, U_N)\right\} &=& \bar{Z}^2 + \mathrm{var}\{\tilde{Z}\mid(A, U_N)\}\notag\\
&=& \bar{Z}^2 + n_1^{-2}\sum_{i\in A}p_i^{-2}\mathrm{var}\{\hat{Y}_{i,\cdot}\mid(A, U_N)\},\notag
\end{eqnarray}
where $\tilde{Z} = n_1^{-1}\sum_{i\in A}\hat{Y}_{i,\cdot}p_i^{-1}$. Thus,
we can estimate $\bar{Z}^2$ by
$$
\tilde{Z}^2  - n_1^{-2}\sum_{i\in A}p_i^{-2}\hat{V}\{\hat{Y}_{i,\cdot}\mid(A, U_N)\}.
$$

By (\ref{eq: Poi first part apprx}), (\ref{eq: PPS 2}) and the two approximations above, we can obtain the variance estimate of $\tilde{Y}$ by
\begin{eqnarray}
\tilde{V} &=& N^{-2}n_1^{-1}(n_1-1)^{-1}\sum_{i\in A}p_i^{-2}[\hat{Y}_{i,\cdot}^2 + (n_1-2)n_1^{-1}\hat{V}\{\hat{Y}_{i,\cdot}\mid(A, U_N)\}] - N^{-2}\tilde{Z}^2 \notag
\end{eqnarray}
for the two-stage sampling design with PPS sampling is used in the first stage.
\bibliographystyle{dcu}
\bibliography{paper-ref_should}

\end{document}